\documentclass[10pt]{amsart}

\usepackage[T1]{fontenc} 
\usepackage[utf8]{inputenc} 
\usepackage{longtable}
\usepackage{amssymb}
\usepackage{amsthm}
\usepackage{amsmath}
\usepackage{wasysym}
\usepackage{dcpic, pictex}
\usepackage{bbm}
\usepackage{enumerate}
\usepackage{amsfonts}
\usepackage{amsthm}
\usepackage{amsmath}
\usepackage{times}
\usepackage{amscd}
\usepackage[mathscr]{eucal}
\usepackage{indentfirst}
\usepackage{pict2e}
\usepackage{epic}
\usepackage{epstopdf} 
\usepackage{verbatim}
\usepackage{cancel}
\usepackage[foot]{amsaddr}
\usepackage{lipsum}
\usepackage{mathrsfs}
\usepackage{bm}

\usepackage{
amsfonts, 
latexsym, 
enumerate,
}
\usepackage[english]{babel}
\allowdisplaybreaks
\makeindex
\newtheorem{theorem}{Theorem}[section]
\newtheorem{lemma}[theorem]{Lemma}
\newtheorem{proposition}[theorem]{Proposition}
\newtheorem{definition}[theorem]{Definition}

\newtheorem{remark}[theorem]{Remark}

\newtheorem{assumption}[theorem]{Assumption}
\usepackage[a4paper,top=3.5cm,bottom=3.5cm,left=3cm,right=3cm]{geometry}
\numberwithin{equation}{section}

\allowdisplaybreaks

\usepackage{ bbold }
\usepackage{xcolor}
\usepackage{listings}

\newcommand{\nc}{\normalcolor}
\newcommand{\dif}{\mathrm{d}}
\newcommand{\E}{\mathbf{E}}
\newcommand{\R}{\mathbf{R}}
\newcommand{\C}{\mathbf{C}}

\newcommand{\N}{\mathbf{N}}

\newcommand{\ma}{\mathfrak{a}}

\newcommand{\Var}{\mathrm{Var}}
\newcommand{\Tr}{\mathrm{Tr}}

\newcommand{\ii}{\mathrm{i}}
\newcommand{\ee}{\mathrm{e}}
\newcommand{\rd}{\mathrm{d}}

\usepackage{hyperref}

\title[Beyond normal fluctuations]{Beyond normal fluctuations in local laws for Wigner matrices}

\date{\today}

\begin{document}

\maketitle

\vspace{0.25cm}

\renewcommand{\thefootnote}{\fnsymbol{footnote}}

\noindent
\mbox{}%
\hfill%
\begin{minipage}{0.21\textwidth}
	\centering
	{L\'aszl\'o Erd\H{o}s}\footnotemark[1]\\
	\footnotesize{\textit{lerdos@ist.ac.at}}
\end{minipage}
\hfill%
\begin{minipage}{0.21\textwidth}
	\centering
	{Oleksii Kolupaiev}\footnotemark[1]\\
	\footnotesize{\textit{Oleksii.Kolupaiev@ist.ac.at}}
\end{minipage}
\hfill%
\mbox{}%
\footnotetext[1]{Institute of Science and Technology Austria, Am Campus 1, 3400 Klosterneuburg, Austria. 
}
\footnotetext[1]{Supported by the ERC Advanced Grant ``RMTBeyond'' No.~101020331.}

\renewcommand*{\thefootnote}{\arabic{footnote}}
\vspace{0.25cm}

\begin{abstract}  We identify non-Gaussian corrections to the central limit theorem for the global and  local laws, i.e. for Stieltjes transform of the empirical eigenvalue distribution of a large real symmetric or complex Hermitian Wigner matrix. We find that in the real case the rate of convergence in this CLT is substantially slower than in the complex case. For the proofs, we develop new static and dynamic hierarchies of extended cumulants
up to any order that can be analysed by a refinement of the zigzag strategy.
\end{abstract}
\vspace{0.15cm}

\footnotesize \textit{Keywords:} Local Law, Zigzag Strategy, Moment Matching, Cumulant Expansion.

\footnotesize \textit{2020 Mathematics Subject Classification:} 60B20, 82C10.
\vspace{0.25cm}
\normalsize

\section{Introduction}

Let $W=W^*$ be an $N\times N$ real symmetric ($\beta=1$) or complex Hermitian ($\beta=2$) Wigner matrix, meaning that the entries of $W$ are centered, independent up to the symmetry constraint and identically distributed. It is well-known that for any self-adjoint polynomial of $W$,  $P(W)=(P(W))^*$, the fluctuation of $\Tr P(W)$ is asymptotically Gaussian as $N$ goes to infinity \cite{Male2022, Belinschi2025}. Recently, the  speed of convergence was also quantified
in terms of asymptotic vanishing of all  cumulants of order three and higher.
 Namely, the following optimal upper bound was shown in~\cite[Theorem~1.6]{Ultra_high_cumulants}
\begin{equation}
\left\vert\kappa_p \left(Y\right)\right\vert\lesssim 
\frac{1}{N^{(p-2)/2}},\qquad Y:=\frac{\Tr P-\E \Tr P}{\Var[\Tr P]},
\label{eq:kappa_P}
\end{equation}
for any $N$-independent self-adjoint polynomial\footnote{More generally, \cite{Ultra_high_cumulants} covers the case when $P$ is a not necessarily self-adjoint polynomial of several independent Wigner matrices and of several deterministic matrices of the same size. Moreover, $p$ may depend on $N$, but then the rhs. of \eqref{eq:kappa_P} should be multiplied by an appropriate function of $p$.} $P=P(W)$ and integer $p\ge 3$. Here $\kappa_p(Y)$ stands for the $p$-th order cumulant of $Y$, and note that $Y$ is a centered random variable with variance one.

For the special case of the Gaussian Unitary Ensemble (GUE), the speed of convergence is much better, in fact even the nontrivial limit of $N^{p-2}\kappa_p(Y)$ was  identified in \cite{2nd_order_freeness_2006}. The analogous result was obtained later in \cite{cumulant_asymptotic_2024} for Hermitian Wigner matrices with Gaussian entries on the diagonal. Notably, the scaling of cumulants in \cite{cumulant_asymptotic_2024} coincides with the one for GUE, and thus the slower convergence rate in \eqref{eq:kappa_P} stems exclusively from non-Gaussianity of diagonal entries of $W$. We refer to the introduction of \cite{Ultra_high_cumulants} for a more detailed overview of the related literature.

By the matrix-valued functional calculus, $\mathrm{Tr} P = \sum_{i=1}^N P(\lambda_i)$, where $\{\lambda_i\}_{i=1}^N$ are the eigenvalues of $W$ labeled in the increasing order. Since $P$ is a polynomial, most terms in this sum are roughly of the same order, therefore  we may say that $\mathrm{Tr} P$ is a \emph{global} linear eigenvalue statistics of $W$. In this paper we focus on  \emph{mesoscopic} linear eigenvalue statistics and therefore replace $\Tr P$ by the \emph{spectrally localized} quantity $\Tr G$, where $G=G(z):=(W-z)^{-1}$ is the resolvent of $W$. The trace of resolvent is sensitive to the mesoscopic scales: in the bulk regime $|E|<2$ and $N^{-1}\ll \eta\ll 1$, the leading contribution to the sum
\begin{equation}
\Tr\Im G(E+\ii\eta) = \sum_{i=1}^N \frac{\eta}{(\lambda_i-E)^2+\eta^2},
\label{eq:Tr_Im_G}
\end{equation}
comes from the $N\eta$ eigenvalues located at a mesoscopic scale $\eta$ around $E$.

Standard local law (see \eqref{eq:kG_av}  later) asserts that $\Tr G(z) $ strongly concentrates around a deterministic quantity, namely
\begin{equation}\label{LL}
       |\langle G(z)-m(z)\rangle| \le \frac{N^\xi}{N\eta},\qquad \eta:=|\Im z|,
\end{equation}
with very high probability for any fixed $\xi>0$. Here $\langle\cdot\rangle:=N^{-1}\Tr$ is the normalized trace and $m(z)$ is the Stieltjes transform of the semicircular distribution,
\begin{equation}
m(z):=\frac{-z+\sqrt{4-z^2}}{2},\quad \Im z\Im m(z)>0.
\label{eq:def_m}
\end{equation}
What is the nature of the fluctuation in the local law? It is known from \cite[Theorem 2.2, Remark 2.4]{HeKnowles} that appropriately rescaled lhs. of \eqref{LL} converges to a standard complex Gaussian random variable $\mathcal{N}_\C$ in distribution:
\begin{equation}
X:=N\eta \langle G(z)-m(z)\rangle \to \frac{1}{\sqrt{2\beta}}\mathcal{N}_\C,
\label{eq:limit_Gauss}
\end{equation}
as $N$ goes to infinity and $N^{-1}\ll\eta \ll 1$. 

As $X$ is a genuinely complex random variable, \eqref{eq:limit_Gauss} does not follow simply from the convergence of cumulants $\kappa_p(X)$ to zero for $p\ge 3$. Instead of $\kappa_p(X)$, one needs to establish this convergence for a larger family $\kappa_{p,q}(X)$ of joint cumulants\footnote{For the precise definition of $\kappa_{p,q}(X)$ see \eqref{eq:def_kappa_pq} later.} of $X$ and $\overline{X}$ of order $p+q\ge 3$. Inspecting the proof in \cite{HeKnowles}, we find the following explicit non-optimal upper bound:  
\begin{equation}
\left\vert\kappa_{p,q}\left(X\right)\right\vert\lesssim \left(\frac{1}{N\eta}\right)^{1/3}+\eta^{1/3},
\label{eq:limit_Gauss_kappa}
\end{equation}
for any fixed $p,q\in\N\cup\{0\}$ with $p+q\ge 3$.

In this paper we improve the bound \eqref{eq:limit_Gauss_kappa} in two different directions in the bulk regime. First, we identify the global non-Gaussian mode of $X$ by computing the leading order term to
$\kappa_{p,q}(X)$ in the global regime $\eta\sim 1$, and estimate the remainder in terms of $(N\eta)^{-1}$: 
\begin{equation}
\kappa_{p,q}\left(X\right) = \frac{\eta^{p+q}}{N^{(p+q-2)/2}}K_{p,q}(z)\kappa_{p+q}^{\mathrm{d}} + \frac{N^\xi}{N\eta} \mathcal{O}\left( \frac{1}{N^{(p+q-3)/2}}+ \frac{1}{(N\eta)^{p+q-3}}\right),
\label{eq:kappapq_new_intro}
\end{equation}
for any fixed $\xi>0$. The explicit function $K_{p,q}$ is defined later in \eqref{eq:def_K} and $\kappa_{p+q}^\dif$ is the cumulant of order $p+q\ge 3$ of a normalized diagonal entry of $W$, for more details see Theorem~\ref{theo:main}(i). Note that for $\eta\sim 1$ and non-vanishing $\kappa_{p+q}^\dif$, the first term in the rhs. of \eqref{eq:kappapq_new_intro} scales as $N^{-(p+q-2)/2}$, which coincides with the scaling in \eqref{eq:kappa_P}. On the other hand, when the diagonal entries of $W$ are Gaussian, $\kappa_{p+q}^\dif =0$,  then $\kappa_{p,q}(X)$ is much smaller than in the general case. This aligns with the conclusions of \cite{cumulant_asymptotic_2024} discussed above that non-Gaussianity of linear statistics primarily comes from the non-Gaussianity of the diagonal elements. We also show (Theorem~\ref{theo:main1}) that for one-sided cumulants, $\kappa_{p,0}(X)$ and $\kappa_{0,q}(X)$, \eqref{eq:kappapq_new_intro} holds with much better bound on the error term.

To motivate the second improvement of \eqref{eq:limit_Gauss_kappa}, we distinguish between  upper bounds on $\kappa_{p,q}(X)$ in terms of negative powers of $N$ and of $N\eta$. Let $\epsilon>0$ be a small $N$-independent exponent. In the regime $\eta\sim N^{-1+\epsilon}$ bounds of the second type depend on $\epsilon$ and approach the trivial bound \eqref{LL} of order $N^\xi$ as $\epsilon$ decreases to zero. In contrast, bounds in terms of $N$ do not depend on $\epsilon$. While \eqref{eq:limit_Gauss_kappa} and \eqref{eq:kappapq_new_intro} imply the bounds of the second type, the following question naturally arises.

\smallskip

\noindent\textbf{Question:} Fix $p,q\in\N$ with $p+q\ge 3$. Does there exist a positive $\gamma=\gamma(p,q)>0$,
independent of $N$ and $\epsilon$, such that $|\kappa_{p,q}(X)|\lesssim N^{-\gamma}$ uniformly in $N^{-1+\epsilon}\le\eta\lesssim 1$? 

\smallskip

\noindent Our result is somewhat surprising: we answer this question positively for complex Hermitian and negatively for real symmetric Wigner matrices. Namely, for $\beta=2$ we show that
\begin{equation}
\left\vert\kappa_{p,q}\left(X\right)\right\vert \lesssim \frac{1}{\sqrt{N}}\bm{1}_{p+q=3} + \frac{1}{N},
\label{eq:main_pq_intro}
\end{equation} 
while for $\beta=1$ we prove that $\kappa_{2,2}(X)$ scales as $(N\eta)^{-3}$ for $\eta\sim N^{-1+\epsilon}$. In contrast, for one-sided cumulants $\kappa_{p,0}(X)$ and $\kappa_{0,q}(X)$, we find the positive answer in both symmetry classes.

We remark that the difference between the real and complex cases for a  related question that concerns the expectation of $\langle  G \rangle$ and  not its higher cumulants was observed in \cite{Shamis} in the following context. For $W$ from the Gaussian Orthogonal or Unitary Ensemble (GOE/GUE), the expected density of states\footnote{More precisely, $\rho(E):= \pi^{-1} \lim_{\eta\to 0+} \E \langle \Im G(E+\ii \eta)\rangle$. For Gaussian ensembles this limit is well-defined. } $\rho$ satisfies
\begin{equation}
\rho(E) =\rho_{sc}(E) + N^{-1}f(E) + \mathcal{O}\big(N^{-3/2}\big),\quad |E|\le 2-\delta,
\end{equation}
where $\rho_{sc}(E)=(2\pi)^{-1}\sqrt{[4-E^2]_+}$ is the semicircular density and $f$ is an explicit bounded  function. For GOE $f$ does not depend on $N$, while for GUE $f$ oscillates with frequency $N$. Since $\E\langle G(z)-m(z)\rangle$ with $m(z)$ defined in \eqref{eq:def_m} is the Stieltjes transform of $\rho-\rho_{sc}$, for GUE this expectation is smaller in absolute value than for GOE in the regime $\eta\sim N^{-1+\epsilon}$ due to the oscillatory integral, see also \cite[Lemma~4.4]{Knowles20} for a direct proof.

We conclude this introduction by discussing the proof techniques used in this paper and comparing them to the ones present in the literature. First, the arguments in \cite{Ultra_high_cumulants} and in the related free probability literature concerning cumulants of polynomials are based on the moment method, which is limited to the global regime, $\eta\sim1$,  and thus is not applicable in the current setting. In the mesoscopic literature, \cite{HeKnowles} is based on the static cumulant expansion method. Though we follow this approach in the global regime, its technical complexity would not allow us to prove \eqref{eq:kappapq_new_intro} in the mesoscopic regime, as well as to distinguish between the real and complex cases for small $\eta$. Instead, we use a dynamical method, namely a new version of the \emph{zigzag strategy} introduced in \cite{cipolloni2024out,Cipolloni_meso, cipolloni2023eigenstate}. This is a well-established tool for the proof of local laws, that is high probability concentration bounds for resolvents. However, prior to the current paper, it has never been used for computation of expectations and cumulants with precision exceeding the one given by local laws. Zigzag strategy consists of three steps: (i) estimates in the global regime, (ii) their propagation down to the real line via characteristic flow at the cost of adding a Gaussian component to $W$ (\emph{zig step}) and (iii) removal of the Gaussian component via Green function comparison (\emph{GFT/zag step}). In this paper we use (ii) in two different ways: in the proof of \eqref{eq:kappapq_new_intro} as in the original zigzag strategy but for expectations and cumulants, and in the proof of \eqref{eq:main_pq_intro} as a GFT to transfer the fine cancellations within the cumulant structure from GUE to a Gaussian divisible case. Both applications are new from the methodological side, and the second one is also conceptually new. Finally, due to the high precision of \eqref{eq:kappapq_new_intro} and \eqref{eq:main_pq_intro}, the standard zag step is not sufficient and thus we replace it by the novel \emph{multiple moment matching} approach. The latter one can be viewed as a high-precision analogue of the four moment matching lemma \cite[Lemma~6.5]{bulk_univ_genW}, \cite{TaoVu11} under mild regularity assumptions on the single entry distribution.

\subsection*{Notations and conventions}
We denote the sets of positive integers, real numbers and complex numbers by $\N$, $\R$ and $\C$ respectively. We set $[k] := \{1, ... , k\}$ for $k \in \N$ and $\langle A \rangle := N^{-1} \mathrm{Tr}(A)$, $N \in \N$, for the normalized trace of an $N \times N$-matrix $A$. For positive quantities $f, g$ we write $f \lesssim g$, $f \gtrsim g$, to denote that $f \le C g$ and $f \ge c g$, respectively, for some $N$-independent constants $c, C > 0$ that depend only on the basic control parameters of the model in Assumptions~\ref{ass:momass} and~\ref{ass:support} below. In informal explanations, we frequently use the notation $f\ll g$, which indicates that $f$ is "much smaller" than $g$ when $N$ is large. 

We denote vectors by bold-faced lower case Roman letters $\boldsymbol{x}, \boldsymbol{y} \in \C^{N}$, for some $N \in \N$. Moreover, for vectors $\boldsymbol{x}, \boldsymbol{y} \in \C^{N}$ we define
 \begin{equation*}
	\langle \boldsymbol{x}, \boldsymbol{y} \rangle := \sum_i \bar{x}_i y_i. 
\end{equation*}
Matrix entries are indexed by lower case Roman letters $a, b, c , ... ,i,j,k,... $ from the beginning or the middle of the alphabet and unrestricted sums over those are always understood to be over $\{ 1 , ... , N\}$. 

For a pair of real or complex-valued stochastic processes $X=X(t)$ and $Y=Y(t)$, we denote their covariation process by~$[X,Y]_t$. Finally, we will use the concept  \emph{with very high probability},  meaning that for any fixed $D > 0$, the probability of an $N$-dependent event is bigger than $1 - N^{-D}$ for all $N \ge N_0(D)$. We will use the convention that $\xi > 0$ denotes an arbitrarily small positive exponent, independent of $N$.
 Moreover, we introduce the common notion of \emph{stochastic domination} (see, e.g., \cite{loc_sc_gen}): For two families
\begin{equation*}
	X = \left(X^{(N)}(u) \mid N \in \N, u \in U^{(N)}\right) \quad \text{and} \quad Y = \left(Y^{(N)}(u) \mid N \in \N, u \in U^{(N)}\right)
\end{equation*}
of non-negative random variables indexed by $N$, and possibly a parameter $u$, we say that $X$ is stochastically dominated by $Y$, if for all $\epsilon, D >0$ we have 
\begin{equation*}
	\sup_{u \in U^{(N)}} \mathbf{P} \left[X^{(N)}(u) > N^\epsilon Y^{(N)}(u)\right] \le N^{-D}
\end{equation*}
for large enough $N \ge N_0(\epsilon, D)$. In this case we write $X \prec Y$. If for some complex family of random variables we have $\vert X \vert \prec Y$, we also write $X = O_\prec(Y)$.

\section{Main results}

We consider an $N\times N$ Wigner matrix $W=\left(w_{ab}\right)_{a,b=1}^N$ which is either real symmetric ($\beta=1$) or complex Hermitian ($\beta=2$). The entries of $W$ are independent up to the symmetry constraint and distributed according to
\begin{equation*}
w_{aa}\stackrel{\rd}{=}\frac{1}{\sqrt{N}}\chi_{\rm{d}},\quad w_{ab}\stackrel{\rd}{=}\frac{1}{\sqrt{N}}\chi_{\rm{od}},\,\, a>b,
\end{equation*}
where $\chi_\dif, \chi_{\mathrm{od}}$ are $N$-independent random variables. For $\beta=1$ (respectively, $\beta=2$) we have $\chi_{\mathrm{od}}\in\R$ (respectively, $\chi_{\mathrm{od}}\in\C$), while $\chi_\dif\in\R$ in both cases $\beta=1,2$. We impose two assumptions on these random variables. The first one is the standard moment assumption:
\begin{assumption}
\label{ass:momass}
Both $\chi_\dif$, $\chi_{\mathrm{od}}$ are centered $\E\chi_\dif=\E \chi_{\mathrm{od}}=0$, their variances equal to $\E|\chi_{\mathrm{od}}|^2=1$ and $\E|\chi_{\mathrm{d}}|^2=2/\beta$. In the complex case we also assume that $\E\chi_{\mathrm{od}}^2=0$. Furthermore, we assume the existence of high moments, i.e. for any $p\in\N$ there exists a constant $C_p>0$ such that
\begin{equation}
\label{eq:momass}
\E\big[|\chi_\dif|^p+|\chi_{\mathrm{od}}|^p\big]\le C_p.
\end{equation}
\end{assumption} 

In the special case when the entries of $W$ are Gaussian, $W$ is called a GOE matrix for $\beta=1$ and a GUE matrix for $\beta=2$. To state the second assumption, we introduce the following notion.

\begin{definition}[Non-degenerate support]\label{def:support} Let $\chi$ be either real or complex random variable. Denote the law of $\chi$ by $\mu$, which is a probability measure on $\R$ or $\C$, respectively. We say that $\chi$ has non-degenerate support if the following holds:
\begin{itemize}
\item[$\bullet$]  [Real case] The support of $\mu$ contains infinitely many points.
\item[$\bullet$]   [Complex case] There is no nonzero polynomial $P(z,\overline{z})\not\equiv 0$, such that the support of $\mu$ is contained in the zero set of $P$.
\end{itemize}
\end{definition}

In particular, a real (complex) valued random variable $\chi$ has non-degenerate support if its law has an absolutely continuous component wrt. the Lebesgue measure on $\R$ (respectively, $\C$).

\begin{assumption}\label{ass:support} We assume that the random variables $\chi_\dif$ and $\chi_{\mathrm{od}}$ have non-degenerate support.
\end{assumption}

We define the cumulants of a complex-valued random variable $X$ by\footnote{\label{ftn:kappa} We interpret \eqref{eq:def_kappa_pq} in such a way that first the differentiation in $x,y$ is performed formally, and then the expectations in $X$ are evaluated. Then \eqref{eq:def_kappa_pq} is well-defined for any $X$ with all finite moments even when the expectation in the rhs. of \eqref{eq:def_kappa_pq} does not exist.} 
\begin{equation}
\kappa_{p,q}(X):=\partial_{x}^p\partial_{y}^q \log \E \exp\left\lbrace xX + y\overline{X}\right\rbrace\big|_{x=y=0},\quad p,q\in \N\cup\{0\}.
\label{eq:def_kappa_pq}
\end{equation}
For $p=q=1$ \eqref{eq:def_kappa_pq} recovers the variance of $X$, and in the special case when $X$ is Gaussian we have $\kappa_{p,q}(X)=0$ for all $p+q\ge 3$. For the spectral parameters $z$ that we consider, we
define the \emph{bulk spectral domain}:
\begin{equation}
\mathcal{D}=\mathcal{D}(\delta,\epsilon,C):=\left\lbrace E+\ii\eta\,:\, |E|\le 2-\delta,\, N^{-1+\epsilon}\le \eta\le C\right\rbrace,
\label{eq:def_D}
\end{equation}
where the (small) $\delta, \epsilon>0$ and  (large) $C>0$ are fixed control parameters. Note that $\mathcal{D}$ contains the macroscopic scales ($\eta\sim 1$) as well as the entire mesoscopic regime $N^{-1}\ll \eta\ll 1$.

Throughout the paper, all implicit constants in $\lesssim $ and $\mathcal{O}(\cdot)$ may depend on the distribution of $\chi_\dif$ and $\chi_{\rm{od}}$, and in particular on constants $C_p$ from Assumption~\ref{ass:momass}, as well as on $\delta,\epsilon,C$ in the definition of $\mathcal{D}$, unless stated otherwise.

Now we are ready to present our main result. 

\begin{theorem}\label{theo:main} Let $W$ be a Wigner matrix satisfying Assumptions~\ref{ass:momass} and~\ref{ass:support}. Fix (small) $\delta, \epsilon>0$, (large) $C>0$ and consider the bulk spectral domain
$\mathcal{D}=\mathcal{D}(\delta,\epsilon,C)$. Let
\begin{equation*}
X:= N\eta\langle G(z)-m(z)\rangle
\end{equation*}
be the natural rescaling of $\langle G(z)-m(z)\rangle$. The following results hold for any fixed $p,q\in \N\cup\{0\}$ with $p+q\ge 3$, uniformly in $z=E+\ii\eta\in\mathcal{D}$.

\emph{(i)} Both for $\beta=1$ and $\beta=2$ it holds that
\begin{equation}
\kappa_{p,q}\left(X\right) = \frac{\eta^{p+q}}{N^{(p+q-2)/2}}K_{p,q}(z)\kappa_{p+q}^{\mathrm{d}} + \frac{N^\xi}{N\eta} \mathcal{O}\left( \frac{1}{N^{(p+q-3)/2}}+ \frac{1}{(N\eta)^{p+q-3}}\right),
\label{eq:kappapq_new}
\end{equation}
for any fixed $\xi>0$. Here $\kappa^{\dif}_{p+q}:=\kappa_{p+q}(\chi_{\rm{d}})$ is the cumulant of a normalized diagonal entry $w_{aa}$ of $W$, and the function $K_{p,q}$ is explicitly given by
\begin{equation}
K_{p,q}(z) := \left( - \frac{m^2}{1-m^2}\right)^p \left(  - \frac{\overline{m}^2}{1-\overline{m}^2}\right)^q, 
\qquad m= m(z).
\label{eq:def_K}
\end{equation}

\emph{(ii)} In the complex case ($\beta=2$) we further have
\begin{equation}
\left\vert\kappa_{p,q}\left(X\right)\right\vert \lesssim \frac{1}{\sqrt{N}}\bm{1}_{p+q=3} + \frac{1}{N},
\label{eq:main_pq}
\end{equation} 
while in the real case ($\beta=1$) it holds that
\begin{equation}
\kappa_{2,2}\left(X\right) = \frac{5\ii}{2} \left(-\frac{m}{1-m^2}\right)^3 \frac{1}{(N\eta)^3} + \mathcal{O}\left(\frac{1}{(N\eta)^4}+\eta\right).
\label{eq:kappa22_real}
\end{equation}
\end{theorem}

Note that part (i) of the theorem is mainly relevant for relatively large $\eta$’s where it explicitly identifies the leading order non-Gaussianity of the cumulants for both symmetry classes. Part (ii) is more interesting in the small $\eta$ regime, where the explicit term in~\eqref{eq:kappapq_new} becomes smaller than the error terms. In particular, for $\beta=1$ we identify a new leading term of order $(N\eta)^{-3}$ in this regime. For $\beta=2$ the analogous terms of order $(N\eta)^{-k}$ all vanish giving rise to a much faster convergence to the Gaussian limit for $\beta=2$ than $\beta=1$.

The situation is much simpler for one-sided cumulants, $\kappa_{p,0} = \bar{\kappa}_{0,p}$, i.e. when one of the parameters $p,q$ equals to zero. In this case there is no difference between the real and complex case. Moreover, no powers of $N\eta$ appear in the estimates and the global mode corresponding to the first term in the rhs. of \eqref{eq:kappapq_new} becomes the leading term in the entire mesoscopic regime. Crudely speaking, these improvements are due to the fact that one-sided cumulants are analytic functions of $z$.

\begin{theorem}[One-sided cumulants]\label{theo:main1} Assume the set-up and conditions of Theorem~\ref{theo:main}. For $q=0$ and any fixed $p\ge 3$ it holds that
\begin{equation}
\kappa_{p,0}\left(X\right) = \frac{\eta^p}{N^{(p-2)/2}}\left(K_{p,0}(z)\kappa_p^{\mathrm{d}} + \mathcal{O}\left(\frac{1}{N^{1/2}}\right)\right),
\label{eq:main_p0}
\end{equation}
uniformly in $z\in\mathcal{D}$.
\end{theorem}

For simplicity we stated our results for a single $X=X(z)$, but 
we can easily extend them to the joint cumulants of a collection of $X(z_j)$’s with different spectral parameters:

\begin{remark}[Different spectral parameters]\label{rem:ext} The results of Theorem~\ref{theo:main}(i) and Theorem~\ref{theo:main1} also hold  for the cumulants of resolvent traces evaluated at different spectral parameters. For example, setting $X_j:=N\eta_j\langle G(z_j)-m(z_j)\rangle$, \eqref{eq:kappapq_new} becomes
\begin{equation}
\begin{split}
\kappa\left(\big\lbrace X_j\big\rbrace_{j=1}^r\right) =& \frac{\prod_{j=1}^r \eta_j}{N^{(r-2)/2}}\left(K\left(\{z_j\}_{j=1}^r\right)\kappa_{r}^{\mathrm{d}} + \mathcal{O}\left(\frac{1}{N^{1/2}}\right)\right)\\
 + &\frac{N^\xi}{N\eta_*}\frac{\prod_{j=1}^r \eta_j}{\eta_*^r}\mathcal{O}\left(\frac{1}{N^{(r-3)/2}}+\frac{1}{(N\eta_*)^{r-3}}\right),
\end{split}
\label{eq:different_z}
\end{equation}
for any fixed $r\in \N$, $r\ge 3$, uniformly in $z_j=E_j\pm \ii\eta_j\in \mathcal{D}\cup \overline{\mathcal{D}}$. Here $\overline{\mathcal{D}}:=\left\lbrace\overline{z}:z\in\mathcal{D}\right\rbrace$, $\eta_*:=\min_{j\in [r]} \eta_j$ and $K$ is given by
\begin{equation}
K\left(\{z_j\}_{j=1}^r\right):=\prod_{j=1}^r \left(-\frac{m_j^2}{1-m_j^2}\right),\qquad m_j:=m(z_j).
\end{equation}
In particular $K_{p,q}(z)$ coincides with $K$ evaluated at the set of $p+q$ arguments, $p$ of those equal to $z$ and $q$ to $\overline{z}$. The proofs of~\eqref{eq:different_z} and the analogue of~\eqref{eq:main_p0} are identical to those of \eqref{eq:kappapq_new} and \eqref{eq:main_p0}, respectively, hence we omit them.
\end{remark}

Finally, we discuss the implications of Theorems~\ref{theo:main} and \ref{theo:main1} for the real valued random variable~$\Im X$.

\begin{remark}[Cumulants of $\Im X$] Assume the set-up and conditions of Theorem~\ref{theo:main}. By \eqref{eq:kappapq_new} and multi-linearity of cumulants it holds that 
\begin{equation}
\begin{split}
\kappa_p\left(\Im X\right) = &\frac{\eta^{p+q}}{N^{(p+q-2)/2}}K_p^\Im (z)\kappa_{p}^{\mathrm{d}} + \frac{N^\xi}{N\eta} \mathcal{O}\left( \frac{1}{N^{(p-3)/2}}+ \frac{1}{(N\eta)^{p-3}}\right),\\
K_p^\Im(z):=&\left(-\Im \frac{m^2}{1-m^2}\right)^p,\quad m=m(z),
\end{split}
\label{eq:kappapq_Im}
\end{equation}
for any fixed integer $p\ge 3$, uniformly in $z\in\mathcal{D}$. Since $K_p^\Im (z)\neq 0$ for $\Re z\neq 0$, \eqref{eq:kappapq_Im} identifies the leading order of $\kappa_p(\Im X)$ in the global regime away from the imaginary axis. We also note that there is no analogue of \eqref{eq:main_p0} for $\Im X$ due to the fact that $\Im X$ is not an analytic function of $z$.
\end{remark}

\section{Proof strategy}\label{sec:strategy}

In this section we present the key ingredients of the proofs of our main results. Using these inputs, we then prove the statements \eqref{eq:kappapq_new}, \eqref{eq:main_pq} and \eqref{eq:main_p0} of Theorems~\ref{theo:main} and~\ref{theo:main1} at the end of Section~\ref{sec:strategy_Wigner}. Meanwhile, the proofs of the ingredients are postponed to later sections. Since \eqref{eq:kappa22_real} from Theorem~\ref{theo:main} allows for a more simplified argument, we omit its discussion here and defer it to Appendix~\ref{app:kappa22_real}.

The proof strategies of \eqref{eq:kappapq_new} and \eqref{eq:main_p0} are closely related, the more restrictive setting of \eqref{eq:main_p0} merely allows for sharper estimates. Though the proof of \eqref{eq:main_pq} involves additional new ideas, all three statements follow the same overall scheme consisting of three steps, with the principal differences arising in the second step. The first step concerns the global regime, where $\eta:=|\Im z|\sim 1$. Specifically, we prove Theorems~\ref{theo:main} and~\ref{theo:main1} for $z\in\mathcal{D}_g$ with
\begin{equation}
\mathcal{D}_g=\mathcal{D}_g(\varepsilon,C):=\{z\in\C\,:\, \Im z\ge \varepsilon,\, |z|\le C\},
\label{eq:def_D_global}
\end{equation} 
for any fixed $\varepsilon,C>0$. Note that in the global regime there is no need to restrict $\Re z$ to the bulk spectrum. In the second step we extend these results down to the real axis inside the bulk at the cost of adding a Gaussian component to the random matrix. Finally, the third step removes this Gaussian component via a Green function comparison theorem (GFT) based on multiple moment matching. We discuss these steps in detail in Sections~\ref{sec:strategy_Gauss}, \ref{sec:strategy_Gaus_div} and~\ref{sec:strategy_Wigner} below.

Arguments both in Sections~\ref{sec:strategy_Gauss} and~\ref{sec:strategy_Gaus_div} rely on hierarchies of identities for cumulants. Though these hierarchies are obtained in different ways (via static and dynamic expansions, respectively), to make them closed one needs to extend the family of cumulants in the lhs. of \eqref{eq:kappapq_new} by including traces of powers of $G$ into consideration.  To this end, define
\begin{equation}\label{def:mk}
m_k=m_k(z):=\frac{1}{(k-1)!}\partial_z^{k-1} m(z),\quad k\in\N,
\end{equation}
where $m(z)$ is given by \eqref{eq:def_m}. In particular, $m_1=m$. Throughout the paper we frequently use that
\begin{equation}
|m_k(z)|\lesssim 1\quad\text{uniformly in}\,\, z\in\mathcal{D}\cup\mathcal{D}_g, 
\label{eq:m_bound}
\end{equation}
for any fixed $k\in\N$, since $m(z)$ has bounded derivatives in the bulk regime. Let $I\in\C^{N\times N}$ be the identity matrix. We call $m_k= m_k\cdot I$ the \emph{deterministic approximation} to $G^k$, as the following averaged and isotropic local laws from \cite[Theorem~2.5]{Multi_res_llaws} hold 
\begin{align}
\left\vert\langle G^k(z)-m_k(z)\rangle\right\vert & \prec\frac{1}{N\eta^k},\label{eq:kG_av}\\
\left\vert\left\langle \bm{x},\big( G^k(z)-m_k(z)\big)\bm{y}\right\rangle\right\vert &\prec \frac{1}{\sqrt{N\eta}\eta^{k-1}},\label{eq:kG_iso}
\end{align}
for any fixed $k\in \N$, uniformly in deterministic vectors $\bm{x},\bm{y}\in\C^N$ with $\|\bm{x}\|=\|\bm{y}\|=1$, and in $z\in\mathcal{D}\cup\mathcal{D}_g$. For $n\in\N$ and a vector $\bm{k}=(k_1,\ldots,k_n)\in\N^n$, denote
\begin{equation}
\Delta(\bm{k}):=\prod_{i=1}^n \mathrm{Tr}\left[ G^{k_i}-m_{k_i}\right],\quad \Delta(\emptyset):=1,\quad \mathfrak{n}(\bm{k}):=n,\quad |\bm{k}|:=k_1+\ldots+k_n.
\label{eq:def_Delta}
\end{equation}
We call $\bm{k}$ a \emph{multi-index}, infusing this name with the meaning that all entries of $\bm{k}$ are positive integers. By \eqref{eq:kG_av} the following \emph{size bound} holds: 
\begin{equation}
\eta^{|\bm{k}|} \left\vert \Delta(\bm{k})\right\vert \prec 1,
\label{eq:Delta_a_priori}
\end{equation}
for any fixed multi-index $\bm{k}$.

We extend the family of the cumulants in the rhs. of~\eqref{eq:kappapq_new} by considering
\begin{equation}
\kappa\left(\left\lbrace \eta^{|\bm{k}_j|+|\bm{l}_j|}\Delta(\bm{k}_j)\overline{\Delta(\bm{l}_j)}\right\rbrace_{j=1}^r\right),
\label{eq:general_av_kappa}
\end{equation}
where $\bm{k}_j,\bm{l}_j$ are (possibly empty) multi-indices for $j\in[r]$. Here the multivariate cumulant is defined by
\begin{equation}
\kappa\left(\lbrace X_j\rbrace_{j=1}^r\right)=\kappa\left(X_1,\ldots,X_r\right):=\partial_{x_1}\cdots\partial_{x_r} \log \E \exp\left\lbrace x_1X_1+\ldots + x_rX_r\right\rbrace \Big|_{x_1=\ldots=x_r=0},
\label{eq:def_cum}
\end{equation}
for complex-valued random variables $X_1,\ldots,X_r$ defined on the same probability space, see also Footnote~\ref{ftn:kappa}. We recover $\kappa_{p,q}\left(\eta\Delta(1)\right)$ from \eqref{eq:general_av_kappa} in the special case
\begin{equation}
r=p+q,\qquad \bm{k}_j=\begin{cases}
1,& j\le p\\
\emptyset,& j\ge p+1
\end{cases},\qquad
\bm{l}_j=\begin{cases}
\emptyset,& j\le p,\\
1,& j\ge p+1.
\end{cases} 
\end{equation}

\subsection{Static step}\label{sec:strategy_Gauss} Throughout this section the spectral parameter $z\in \C$ remains fixed, i.e. we follow the static approach. First, in Section~\ref{sec:global} we prove Theorem~\ref{theo:main}(i) in the global regime that we restate here explicitly as follows:

\begin{proposition}[Global regime]\label{prop:global} Let $W$ be an $N\times N$ Wigner matrix satisfying Assumption~\ref{ass:momass}. Fix a (small) $\varepsilon>0$ and a (large) $C>0$. Then \eqref{eq:kappapq_new} holds uniformly in $z\in\mathcal{D}_g(\varepsilon,C)$, where $\mathcal{D}_g$ is defined in~\eqref{eq:def_D_global}.
\end{proposition}

Note that for $z\in\mathcal{D}_g$, Theorem~\ref{theo:main1} is the special case of Theorem~\ref{theo:main}(i) and does not require a separate proof. Next, in Section~\ref{sec:Gauss} we independently prove the following result for GUE, which is the key ingredient in the proof of~\eqref{eq:main_pq}.

\begin{proposition}\label{prop:E_Gauss} Let $W$ be an $N\times N$ GUE matrix. For any $n_1,n_2\in\N\cup\{0\}$ and multi-indices $\bm{k}\in\N^{n_1},\bm{l}\in\N^{n_2}$ it holds that    
\begin{equation}
\E(\bm{k},\bm{l}):=\E \left[\eta^{|\bm{k}|+|\bm{l}|}\Delta(\bm{k}) \overline{\Delta(\bm{l})}\right] = \Phi_0(\bm{k},\bm{l}) + \mathcal{O}(\eta),
\label{eq:E_Gauss}
\end{equation}
uniformly in $z\in \mathcal{D}$. Here $\Phi_0$ is an explicit function of $\bm{k}$ and $\bm{l}$ given later in \eqref{eq:Phi_0_expl} which does not depend on $N$ and $z$. Similarly, the implicit constant in $\mathcal{O}(\eta)$ does not depend on $N$.
\end{proposition}

The precise value of $\Phi_0$ is used only twice in this paper: internally in the proof of Proposition~\ref{prop:E_Gauss} and later in the proof of~\eqref{eq:main_pq} for $p+q=3$. We stress that~\eqref{eq:E_Gauss} does not hold for a GOE matrix $W$ if $\eta\sim N^{-1+\epsilon}$ for a sufficiently small fixed $\epsilon>0$. In fact, typically in this case the correction to the order one $\Phi_0$ term (which slightly differs in the real case) in the rhs. of \eqref{eq:E_Gauss}  is of order $(N\eta)^{-1}\sim N^{-\epsilon}$. For more details see Appendix~\ref{app:kappa22_real}.

While Proposition~\ref{prop:E_Gauss} concerns not only the global regime $z\in\mathcal{D}_g$, but directly deals with any $z\in\mathcal{D}$, it shares the static proof methods with Proposition~\ref{prop:global} and thus is stated in this section. Since we prove Propositions~\ref{prop:global} and~\ref{prop:E_Gauss} independently and the proof of the latter one is technically simpler, here we only present the proof ideas only behind Proposition~\ref{prop:E_Gauss}. 

Our main tool is the \emph{second order (Gaussian) renormalization} originally introduced in \cite[Eq.(4.2)]{ETH_Wigner} (for earlier works, which did not operate with this formalism but followed a similar approach\footnote{We warn the reader that in \cite{Knowles20} the underline notation is used for a different purpose, to denote the centering of a random variable. While no special notation was introduced for  the second order renormalization, the idea
is present in the last term in the rhs. of \cite[Eq.(2.5)]{Knowles20}.} see e.g. \cite[Eq.(2.5)]{Knowles20}). For any differentiable function~$g(W)$ we denote
\begin{equation}
\underline{Wg(W)}:=Wg(W) - \widetilde{\E}\left[ \widetilde{W}\left(\partial_{\widetilde{W}}g\right)(W)\right],
\label{eq:def_under_1}
\end{equation} 
where $\widetilde{W}$ is an independent copy of $W$, the expectation $\widetilde{\E}$ is taken with respect to $\widetilde{W}$ and $\partial_{\widetilde{W}}$ is the derivative in the direction of $\widetilde{W}$. When the entries of $W$ are Gaussian, we have $\E \underline{Wg(W)}=0$.

Consider $\bm{k}=1$ and $\bm{l}=(1,1)$, which is the simplest case demonstrating the proof of Proposition~\ref{prop:E_Gauss}. We have from \eqref{eq:E_Gauss} that
\begin{equation}
\E \left[ \eta^3 \Delta(1)\overline{\Delta(1)}^2\right] = \Phi_0(1,(1,1)) +\mathcal{O}(\eta)=\mathcal{O}(\eta),
\label{eq:cancel_ex0}
\end{equation}
using, additionally, that $\Phi_0(1,(1,1))=0$ by \eqref{eq:Phi_0_expl} below. To establish \eqref{eq:cancel_ex0}, we use the basic identity
\begin{equation}
G-m=\frac{1}{N}m\Tr[G-m]G - m\underline{WG},
\label{eq:underline_basic_id}
\end{equation}
and obtain that
\begin{equation}
\E \left[ \eta^3 \Delta(1)\overline{\Delta(1)}^2\right] =\frac{m}{1-m^2}\frac{1}{N\eta}\left( \E \left[ \eta^4 |\Delta(1)|^4\right] + 2\E \left[ \eta^4\Tr\left[G(G^*)^2\right] \overline{\Delta(1)}\right]\right).
\label{eq:cancel_ex}
\end{equation}
Each of the two terms in the rhs. of~\eqref{eq:cancel_ex} (without the prefactor $(N\eta)^{-1}$) has an upper bound of order $N^\xi$ by the averaged local law~\eqref{eq:kG_iso}. In fact, this overestimates the size of these terms only by factor $N^\xi$, as it a posteriori follows from Proposition~\ref{prop:E_Gauss} and the explicit calculation of the corresponding $\Phi_0$ terms. Thus, naively we get an upper bound of order $(N\eta)^{-1}N^\xi$ on the lhs. of \eqref{eq:cancel_ex0} instead of $\eta$. This does not suffice for the proof of \eqref{eq:cancel_ex0} in the regime $\eta\sim N^{-1+\epsilon}$, so we need to identify a major cancellation between the two terms in \eqref{eq:cancel_ex}. Expanding iteratively the terms in the rhs. of \eqref{eq:cancel_ex} using \eqref{eq:underline_basic_id}, we further identify cancellations for all orders of $(N\eta)^{-1}$, for more details see Section~\ref{sec:Gauss}. We call this iterative expansion procedure the \emph{static chaos expansion}, borrowing the terminology from \cite{hyperuniformity}, where a similar expansion was used for Hermitization of an i.i.d. matrix.

The proof of Proposition~\ref{prop:global} also relies on the static chaos expansion, but for cumulants instead of expectations. It is designed to preserve the cumulant structure, and thus automatically captures cancellations between different terms in the cumulant-moment relations (see e.g.~\eqref{eq:cumulant_moment} later). Instead, the difficulty arises from the analysis of the expectation of fully underlined terms, which cannot be neglected for general Wigner matrices and thus are included into the hierarchy and are subjected to further expansions. For more details see Section~\ref{sec:global}.  

\subsection{Dynamic step}\label{sec:strategy_Gaus_div}

In the arguments presented in this section the spectral parameter is no longer fixed, but it is flowing towards the real axis. This approach allows us to prove \eqref{eq:kappapq_new}, \eqref{eq:main_pq} and~\eqref{eq:main_p0} for $W$ with a Gaussian component uniformly in $z\in\mathcal{D}$. Specifically, we consider $W$ of the form
\begin{equation}
W=\sqrt{1-\alpha} W_0 + \sqrt{\alpha} W_G,
\label{eq:G_component_model}
\end{equation}
where $\alpha\in(0,1)$ does not depend on $N$, $W_0$ is an $N\times N$ Wigner matrix satisfying Assumption~\ref{ass:momass}, and $W_G$ is an independent GOE/GUE matrix of the same symmetry type. We add the Gaussian component in~\eqref{eq:G_component_model} dynamically by embedding $W_0$ into the matrix-valued Ornstein-Uhlenbeck flow:
\begin{equation}
\dif W_t =-\frac{1}{2}W_t +\frac{1}{\sqrt{N}}\dif B_t,\quad t\ge 0,
\label{eq:OU_flow}
\end{equation}
Here $B_t$ is a real symmetric ($\beta=1$) or complex Hermitian ($\beta=2$) matrix-valued Brownian motion with entries having variance equal to $t$ times those of GOE/GUE. For any $t\ge 0$ we have 
\begin{equation}
W_t\stackrel{d}{=} \ee^{-t/2} W_0 +\sqrt{1-\ee^{-t}} W_G.
\end{equation}
In particular, $W_T\stackrel{d}{=}W$ for $T:=|\log(1-\alpha)|$.

We further fix $z\in \mathcal{D}$ and consider the evolution of the spectral parameter along the \emph{characteristic flow}
\begin{equation}
\frac{\dif}{\dif t} z_t = -\frac{1}{2} z_t -m(z_t),\quad t\in[0,T],\,\, \,z_T:=z.
\label{eq:char_flow}
\end{equation}
Note that $z$ is chosen to be the target of this flow at the terminal time $T$,  not as the initial condition. Denote
\begin{equation}
G_t:=(W_t-z_t)^{-1},\quad m_{k,t}:=m_k(z_t)\,\,\text{for}\,\, k\in\N,\quad \eta_t:=|\Im z_t|.
\end{equation}
An important property of \eqref{eq:char_flow} is that $\eta_t$ monotonically decreases and 
\begin{equation}
\eta_t\sim \eta_T+|T-t||\Im m_T|\sim \eta_T + |T-t|,
\end{equation}
where in the last step we used that $|\Im m_T|\sim 1$ due to $z_T\in\mathcal{D}$. In particular, $\eta_0\sim 1$, i.e. $z_0$ is in the global regime. We will also use that $m_t= 
e^{t/2} m_0$ for $t\in [0,T]$. The flow \eqref{eq:char_flow} is chosen in such a way that $G_t$ remains almost unchanged under the time evolution~\cite{pastur1972spectrum, HL_rig, von2019random, Bou_extreme}.

The flows \eqref{eq:OU_flow} and~\eqref{eq:char_flow} are frequently used in tandem as a part of the zigzag strategy in the proof of local laws \cite{cipolloni2024out, Cipolloni_meso, cipolloni2023eigenstate}. Now we adopt this approach for  two new purposes. First, we use it to compute cumulants of the tracial quantities of the form $\Delta(\bm{k})\overline{\Delta(\bm{l})}$ with precision exceeding the one given by the corresponding local laws~\eqref{eq:Delta_a_priori}. This allows us to establish Theorems~\ref{theo:main}(i) and~\ref{theo:main1} for $z\in\mathcal{D}$ and $W$ with a Gaussian component, which we now state.

\begin{proposition}[Gauss divisible case]\label{prop:kappa_local} Assume the set-up \eqref{eq:OU_flow}--\eqref{eq:char_flow} and denote $X_t:=N\eta_t\langle G_t-m_t\rangle$ for $t\in [0,T]$. For any fixed $p,q\in\N\cup\{0\}$ with $p+q\ge 3$ it holds that
\begin{equation}
\kappa_{p,q}\left(X_t\right) = \frac{\eta_t^{p+q}}{N^{(p+q-2)/2}}K_{p,q}(z_t)\kappa_{p+q}^{\mathrm{d}}(t) + \frac{N^\xi}{N\eta_t} \mathcal{O}\left( \frac{1}{N^{(p+q-3)/2}}+ \frac{1}{(N\eta_t)^{p+q-3}}\right),
\label{eq:kappapq_local}
\end{equation}
for any fixed $\xi>0$, uniformly in $t\in [0,T]$ and $z=z_T\in\mathcal{D}$. Here $\kappa^\dif_{p+q}(t):=\kappa_{p+q}\left(N^{1/2} (W_t)_{11}\right)$ and $K_{p,q}$ is defined in \eqref{eq:def_K}. We further have
\begin{equation}
\kappa_{p,0}\left(N\eta_t\langle G_t-m_t\rangle\right) = \frac{\eta_t^p}{N^{(p-2)/2}}\left(K_{p,0}(z_t)\kappa_p^{\mathrm{d}}(t) + \mathcal{O}\left(\frac{1}{N^{1/2}}\right)\right),
\label{eq:main_p0_local}
\end{equation}
for any fixed $p\in\N$, $p\ge 3$.  
\end{proposition}

The proof of Proposition~\ref{prop:kappa_local} is presented in Section~\ref{sec:p0}. It relies on Proposition~\ref{prop:global} as  the initial condition for $t=0$. 

Our second application of the flows \eqref{eq:OU_flow}, \eqref{eq:char_flow} contributes an essential ingredient for the proof of \eqref{eq:main_pq}. For any $n\in\N$ and a multi-index $\bm{k}\in\N^n$ introduce the time-dependent version of $\Delta(\bm{k})$ from \eqref{eq:def_Delta}:
\begin{equation}
\Delta_t(\bm{k}):=\prod_{i=1}^n \mathrm{Tr}\left[ G^{k_i}_t-m_{k_i,t}\right].
\end{equation}
Let $\E^{(W)}$ be the expectation wrt. $W_t$ solving \eqref{eq:OU_flow}, and $\E^{(G)}$ the expectation wrt. a GOE/GUE matrix of the same symmetry type. Equivalently, $\E^{(G)}$ is the expectation wrt. $W_t$ with $W_0$ chosen to be GOE/GUE matrix. We compare the expectations $\E^{(W)}$ and $\E^{(G)}$ of products of resolvent traces with precision matching the one of Proposition~\ref{prop:E_Gauss}. This is the content of the following statement.

\begin{proposition}\label{prop:E_local} Assume the set-up \eqref{eq:OU_flow}--\eqref{eq:char_flow}. The following results hold uniformly in $t\in [0,T]$ and $z=z_T\in\mathcal{D}$.

\noindent \emph{(i)} For any fixed $n_1,n_2\in\N$ and multi-indices $\bm{k}\in \N^{n_1},\bm{l}\in\N^{n_2}$ we have
\begin{equation}
\left(\E^{(W)}-\E^{(G)}\right)\left[\eta_t^{|\bm{k}|+|\bm{l}|}\Delta_t(\bm{k}) \overline{\Delta_t(\bm{l})}\right] =\mathcal{O}(\eta_t).
\label{eq:E_local}
\end{equation}

\noindent \emph{(ii)} For $\beta=2$ and any fixed $p,q\in\N\cup\{0\}$ with $p+q\ge 3$ it further holds that
\begin{equation}
\left\vert \kappa_{p,q}\left(N\eta_t\langle G_t-m_t\rangle\right)\right\vert \lesssim  \frac{1}{\sqrt{N}}\bm{1}_{p+q=3} + \frac{1}{N}.
\label{eq:main_pq_local}
\end{equation}

\end{proposition}

In essence, \eqref{eq:E_local} is a GFT transferring \eqref{eq:E_Gauss} from GUE to the Gaussian divisible case, however it is proven by the means of \eqref{eq:OU_flow} and \eqref{eq:char_flow}. We also note that \eqref{eq:E_local} is stated both for $\beta=1$ and $\beta=2$, since it will be later used in the proof of~\eqref{eq:kappa22_real} in the real case.

We first prove Proposition~\ref{prop:E_local}(i) independently of Proposition~\ref{prop:kappa_local} in Section~\ref{sec:E_local} and then derive Proposition~\ref{prop:E_local}(ii) from Propositions~\ref{prop:E_Gauss}, \ref{prop:kappa_local} and~\ref{prop:E_local}(i) in Section~\ref{sec:kappa_local}. The proofs of  Propositions~\ref{prop:kappa_local} and~\ref{prop:E_local} follow the same approach, so we explain it in the simplest setting, which is that of Proposition~\ref{prop:E_local}(i). The main idea behind the proof of Proposition~\ref{prop:E_local}(i) is the \emph{dynamic chaos expansion}, which is the dynamic version of the static procedure used in Section~\ref{sec:Gauss}. We express $\E \Delta_t(\bm{k})\overline{\Delta_t(\bm{l})}$ in terms of a time-integral of quantities of the same type. Iterating these identities sufficiently many times, we improve the a priori local law bounds~\eqref{eq:Delta_a_priori} to \eqref{eq:E_local}. Though \eqref{eq:Delta_a_priori} is simply a size bound which does not capture the cancellation between $\E^{(W)}$ and $\E^{(G)}$, it suffices as an input for the iterative procedure. We use the dynamic approach instead of static expansions similar to~\eqref{eq:cancel_ex} in order to avoid dealing with the third and higher order terms in the cumulant expansion in the local regime $\eta\ll 1$. We also note that unlike in \eqref{eq:E_Gauss} we do not need to follow any cancellations, as we follow $\E^{(W)}-\E^{(G)}$ instead of $\E^{(W)}$, which would naturally have the same cancellations as identified along the proof of  \eqref{eq:E_Gauss} in Section~\ref{sec:Gauss}.

\subsection{GFT step: multiple moment matching}\label{sec:strategy_Wigner}

In this section we remove the Gaussian component added to $W$ in Section~\ref{sec:strategy_Gaus_div} via GFT and conclude the proof of the statements \eqref{eq:kappapq_new}, \eqref{eq:main_pq} and \eqref{eq:main_p0} of Theorems~\ref{theo:main} an~\ref{theo:main1}, postponing technical details to later sections. The GFT is based on the following multiple moment matching lemma, stating that under some mild regularity conditions, any fixed number of moments of a given random variable can be matched with those of a Gauss-divisible random variable.

\begin{lemma}[Multiple moment matching]\label{lem:moment_match} Let $\chi$ be either real or complex random variable such that $\E |\chi|^p<\infty$ for any $p\in\N$. Assume that $\chi$ has non-degenerate support in the sense of Definition~\ref{def:support}. Then for any fixed $\Gamma\in\N$ and for sufficiently small, possibly $\Gamma$-dependent $\alpha>0$, there exists a random variable $\chi_\alpha$ with compactly supported distribution, such that
\begin{equation}
\E \left(\sqrt{1-\alpha}\,\chi_\alpha +\sqrt{\alpha} \chi^G\right)^{\gamma_1}\left(\sqrt{1-\alpha}\,\overline{\chi}_\alpha +\sqrt{\alpha} \overline{\chi}^G\right)^{\gamma_2} = \E \chi^{\gamma_1}\overline{\chi}^{\gamma_2},\quad \gamma_1+\gamma_2\le \Gamma,\, \gamma_1,\gamma_2\in\N\cup\{0\}.
\label{eq:moment_match}
\end{equation}
Here $\chi^G$ is a standard Gaussian random variable (real if $\chi$ is real and complex otherwise) which is independent of~$\chi_\alpha$.
\end{lemma}

The proof of Lemma~\ref{lem:moment_match} is presented in Section~\ref{sec:moment_match}. Next, we show how to compare $\E \Delta(\bm{k})\overline{\Delta(\bm{l})}$ for two Wigner matrices with several matching moments.

\begin{proposition}[GFT]\label{prop:GFT} Let $W^{(1)}=\big(w^{(1)}_{ab}\big)_{a,b=1}^N$ and $W^{(2)}=\big(w^{(2)}_{ab}\big)_{a,b=1}^N$ be two Wigner matrices satisfying Assumption~\ref{ass:momass}. Fix $\Gamma\in\N$ and assume that 
\begin{equation}
\E \big(w^{(1)}_{ab}\big)^{\gamma_1}\big(w^{(1)}_{ba}\big)^{\gamma_2} = \E \big(w^{(2)}_{ab}\big)^{\gamma_1}\big(w^{(2)}_{ba}\big)^{\gamma_2},\quad \gamma_1+\gamma_2\le \Gamma,\, \gamma_1,\gamma_2\in\N\cup\{0\},
\label{eq:GFT_condition}
\end{equation}
Let $\Delta^{(j)}$ be defined as in \eqref{eq:def_Delta} but with $W$ replaced by $W^{(j)}$, $j=1,2$. Then for any fixed $n_1,n_2\in\N\cup\{0\}$ and $\bm{k}\in\N^{n_1},\bm{l}\in\N^{n_2}$ it holds that
\begin{equation}
\E \left[\eta^{|\bm{k}|+|\bm{l}|}\Delta^{(1)} (\bm{k})\overline{\Delta^{(1)}(\bm{l})}\right]=\E \left[\eta^{|\bm{k}|+|\bm{l}|}\Delta^{(2)} (\bm{k})\overline{\Delta^{(2)}(\bm{l})}\right] + \mathcal{O}\left(N^{-(\Gamma-3)/2+\xi}\right),
\label{eq:E_match}
\end{equation}
for any fixed $\xi>0$, uniformly in $z\in\mathcal{D}$.
\end{proposition}

The proof of Proposition~\ref{prop:GFT} is based on the Lindeberg replacement strategy and is presented in Appendix~\ref{app:GFT}. Finally, we prove the main results of this paper by combining the technical inputs listed in Section~\ref{sec:strategy}. We also remind the reader that the proof of \eqref{eq:kappa22_real} from Theorem~\ref{theo:main} is presented in Appendix~\ref{app:kappa22_real}.   

\begin{proof}[Proof of Theorems~\ref{theo:main} and~\ref{theo:main1}] We start with the proof of \eqref{eq:kappapq_new}. Fix an $N$-independent $\Gamma\ge p+q+2$. Since $W$ satisfies Assumption~\ref{ass:support}, Lemma~\ref{lem:moment_match} implies that there exists an $N$-independent $\alpha\in (0,1)$ and random variables $\chi_{\alpha, \dif}, \chi_{\alpha,{\rm od}}$ with compact support such that \eqref{eq:moment_match} holds for $\chi_\dif$ and $\chi_{\rm od}$. Denote
\begin{equation}
\widetilde{\chi}_\dif := \sqrt{1-\alpha}\chi_\dif + \sqrt{\alpha} \chi^G_\dif,\quad \widetilde{\chi}_\dif := \sqrt{1-\alpha}\chi_{\rm od} + \sqrt{\alpha} \chi^G_{\rm od},
\end{equation}
where $\chi^G_\dif$ is a real Gaussian random variable with variance $2/\beta$, and $\chi^G_{\rm od}$ is standard real Gaussian for $\beta=1$ and complex for $\beta=2$. Let $\widetilde{W}=(\widetilde{w}_{ab})_{a,b=1}^N$ be a Wigner matrix with
\begin{equation*}
\widetilde{w}_{aa}\stackrel{\rd}{=}N^{-1/2}\widetilde{\chi}_{\rm{d}},\quad \widetilde{w}_{ab}\stackrel{\rd}{=}N^{-1/2}\widetilde{\chi}_{\rm{od}},\,\, a>b.
\end{equation*}
By construction, $\widetilde{W}=\sqrt{1-\alpha} W_0+\sqrt{\alpha}W^G$, where $W_0$ is a Wigner matrix satisfying Assumption~\ref{ass:momass}, and $W^G$ is an independent GOE/GUE matrix. Proposition~\ref{prop:kappa_local} applied to $W_T=\widetilde{W}$ implies that \eqref{eq:kappapq_new} holds for~$\widetilde{W}$.

Denote the $\widetilde{W}$-counterpart of $\Delta(\bm{k})$ by $\widetilde{\Delta}(\bm{k})$. We express both cumulants $\kappa_{p,q}(\eta\Delta(1))$ and $\kappa_{p,q}(\eta\widetilde{\Delta}(1))$ in terms of the corresponding moments using the cumulant-moment relations, see \eqref{eq:cumulant_moment} below. Using \eqref{eq:Delta_a_priori} along with \eqref{eq:E_match}, we obtain
\begin{equation}
\left\vert\kappa_{p,q}(\eta\Delta(1))-\kappa_{p,q}(\eta\widetilde{\Delta}(1))\right\vert \lesssim N^{-(\Gamma-3)/2+\xi},
\label{eq:kappa_match_main}
\end{equation}
for any fixed $\xi>0$. Notice additionally that
\begin{equation}
\kappa_{p+q}(N^{1/2}\widetilde{w}_{aa}):=\widetilde{\kappa}^\dif_{p+q}=\kappa^\dif_{p+q},
\label{eq:kappa_match}
\end{equation}
since $\Gamma>p+q$. Combining \eqref{eq:kappa_match_main}, \eqref{eq:kappa_match} and \eqref{eq:kappapq_new} for $\widetilde{W}$, be obtain that \eqref{eq:kappapq_new} holds for $W$ with an additional error term of order $N^{-(\Gamma-3)/2+\xi}$. Since $\Gamma\ge p+q+2$, it can be absorbed into the error term in the rhs. of \eqref{eq:kappapq_new}. This finishes the proof of~\eqref{eq:kappapq_new}.
 
The proofs of \eqref{eq:main_pq} and \eqref{eq:main_p0} are analogous to the one of \eqref{eq:kappapq_new}, so we omit  the details. We only mention that \eqref{eq:main_p0} relies on \eqref{eq:main_p0_local} instead of \eqref{eq:kappapq_local}, and \eqref{eq:main_pq} relies on \eqref{eq:main_pq_local}.
\end{proof}

\section{ Static chaos expansion in the global regime}\label{sec:global}

In this section we prove Proposition~\ref{prop:global} via the static chaos expansion. As it was already mentioned in the end of Section~\ref{sec:strategy_Gauss}, the family of cumulants  in \eqref{eq:general_av_kappa} is not closed under this expansion procedure due to the expectation of fully underlined terms which need further expansion that generates new types of terms. To this end, we now introduce a larger family of cumulants that will be closed. For $n\in\N$, $\bm{k}=(k_i)_{i=1}^n\in \N^n$ and a set of index pairs $\bm{\phi}:=(\phi_i)_{i=1}^n\subset ([N]\times[N])^n$ denote
\begin{equation}
\mathrm{Iso}(\bm{k},\bm{\phi}):=N^{\mathrm{od}(\bm{\phi})/2}\prod_{i=1}^n (G^{k_i})_{\phi_i},\qquad \mathrm{Iso}(\emptyset,\emptyset):=1.
\label{eq:def_Iso}
\end{equation} 
Here $\mathrm{od}(\bm{\phi})$ is the number of off-diagonal index pairs in $\bm{\phi}$ (i.e. the ones with two distinct indices). We call \eqref{eq:def_Iso} an \emph{isotropic quantity}, in contrast to $\Delta(\bm{k})$ which is referred as a \emph{tracial quantity}. The normalization in \eqref{eq:def_Iso} is chosen in such a way that $\left\vert\mathrm{Iso}(\bm{k},\bm{\phi})\right\vert\prec 1$ in the global regime $z\in\mathcal{D}_g$ by the isotropic local law \eqref{eq:kG_iso}. The family of cumulants under consideration is as follows:
\begin{equation}
\kappa\left(\mathrm{Iso}(\bm{k}_0,\bm{\phi})\Delta(\bm{k}_1)\overline{\mathrm{Iso}(\bm{l}_0,\bm{\phi}')\Delta(\bm{l}_1)}, \left\lbrace \Delta(\bm{k}_j)\overline{\Delta(\bm{l}_j)}\right\rbrace_{j=2}^p\right),
\label{eq:extended_family}
\end{equation}
where $\{\bm{k}_j\}_{j=0}^p$, $\{\bm{l}_j\}_{j=0}^p$ are (possibly empty) multi-indices and $\bm{\phi}, \bm{\phi}'$ are sets of index pairs of the same cardinality as $\bm{k}_0$ and $\bm{l}_0$, respectively.  Note that \eqref{eq:extended_family} differs from \eqref{eq:general_av_kappa} by the isotropic quantities included into the first argument of the cumulant. There is no need to extend the hierarchy further by introducing isotropic quantities in the remaining arguments, since \eqref{eq:extended_family} turns out to be  already closed under expansions.

For any $p\in\N$ with $p\ge 2$ and $n,S\in\N$ denote
\begin{equation}
\Psi^{(p)}(n,S):=\max\bigg\lbrace \left\vert \eqref{eq:extended_family}\right\vert : \sum_{j=1}^p \left(\mathfrak{n}(\bm{k}_j)+\mathfrak{n}(\bm{l}_j)\right) \le n,\, \sum_{j=0}^p \left(|\bm{k}_j|+|\bm{l}_j|\right)+\mathrm{od}(\bm{\phi})+\mathrm{od}(\bm{\phi}')\le S\bigg\rbrace.
\label{eq:def_Psi}
\end{equation}
Here $n$ is a threshold on the number of traces in \eqref{eq:extended_family} and $S$ is a threshold on the number of resolvents in traces and isotropic quantities in  \eqref{eq:extended_family} plus the number of off-diagonal entries in isotropic quantities. In the case when the set of cumulants in \eqref{eq:def_Psi} is empty (e.g. when $n\le p-2$), we set $\Psi^{(p)}(n,S):=0$. We also set $\Psi^{(p)}(n,S):=0$ for $p\le 1$. In the sequel, the estimates on $\Psi^{(p)}(n,S)$ will depend only on $p$ and $n$, while $S$ will be used as a counter. In particular, the precise definition  of  $S$ will play a crucial role twice in the proof of Proposition~\ref{prop:global}, see \eqref{eq:case2_ext_terms} and \eqref{eq:cum_intermediate_case2_5} below.

The local laws \eqref{eq:kG_av}--\eqref{eq:kG_iso} imply the following a priori bound on \eqref{eq:def_Psi}:
\begin{equation}
\Psi^{(p)}(n,S)\lesssim N^\xi,
\label{eq:Psi_a_priori}
\end{equation}
for any fixed $\xi>0$ and $p,n,S\in\N$. We stress that the implicit constant in \eqref{eq:Psi_a_priori} may depend on $p,n,S$ but not on $N$. This also concerns all estimates in this section, even if  not stated explicitly. To prove Proposition~\ref{prop:global} we need the following improved bound on $\Psi^{(p)}(n,S)$:

\begin{proposition}\label{prop:global_Psi} Let $W$ be a Wigner matrix satisfying Assumption~\ref{ass:momass}. Fix a (small) $\varepsilon>0$ and a (large) $C>0$. For any fixed $p,n,S\in\N$ it holds that 
\begin{equation}
\Psi^{(p)}(n,S) \lesssim  N^{-(p-2)/2+\xi} \left( N^{-1} + \bm{1}_{n\ge p} N^{(n-p)/2}\right),
\label{eq:Psi_improved}
\end{equation}
for any fixed $\xi>0$, uniformly in $z\in\mathcal{D}_g(\varepsilon,C)$, which is defined in~\eqref{eq:def_D_global}.
\end{proposition}
We remark that the regime $n<p$ is non-trivial in \eqref{eq:Psi_improved} only for $n=p-1$, since for smaller $n$ we set $\Psi$ equal to zero. The case $n=p-1$ corresponds to the situation when there are only isotropic quantities in the first argument of the cumulant \eqref{eq:extended_family}, and each of the remaining arguments contains exactly one trace. If $n\ge 2p-1$ then \eqref{eq:Psi_improved} is even weaker than~\eqref{eq:Psi_a_priori}, but the sharpness of the estimate in this regime is not relevant in the sequel.

Note that Proposition~\ref{prop:global_Psi}  implies the optimal upper bound (up to the factor $N^\xi$) on the lhs. of \eqref{eq:kappapq_new} in the global regime, $\eta\sim 1$:
\begin{equation}
\left\vert\kappa_{p,q}\left(N\eta\langle G(z)-m(z)\rangle\right) \right\vert \lesssim \Psi^{(p+q)}(p+q,p+q) \lesssim N^{-(p+q-2)/2+\xi},
\end{equation}
(here $p+q$ plays the role of $p$ in  Proposition~\ref{prop:global_Psi} and we consider the case $n=p+q$). However, it does not identify the leading term in the rhs. of \eqref{eq:kappapq_new}, so this will be done separately later in this section.

The remainder of this section is structured as follows. First, in Section~\ref{sec:proof_global_Psi} we state the hierarchy of inequalities satisfied by the control parameters~\eqref{eq:def_Psi}, this is the content of Proposition~\ref{prop:global_master}. Using this hierarchy to gradually improve the a priori bound~\eqref{eq:Psi_a_priori}, we prove Proposition~\ref{prop:global_Psi}. Afterwards, in Sections~\ref{sec:Psi_preliminaries}--\ref{sec:diag_expansion} we prove Proposition~\ref{prop:global_master} via the static chaos expansion. Finally, in Section~\ref{sec:proof_global} we use Proposition~\ref{prop:global_Psi} together with the first step of the chaos expansion to identify the leading term in the rhs. of \eqref{eq:kappapq_new} in the global regime and complete the proof of Proposition~\ref{prop:global}.

\subsection{Proof of Proposition~\ref{prop:global_Psi}}\label{sec:proof_global_Psi}

The main technical step in the proof of Proposition~\ref{prop:global_Psi} is as follows.

\begin{proposition}[Master inequalities]\label{prop:global_master} Assume the set-up and conditions of Proposition~\ref{prop:global}. For any fixed $p\ge 2$, $n,S\in\N$ and $\xi> 0$ it holds that
\begin{equation}
\begin{split}
\Psi^{(p)}(n,S)\lesssim &\Psi^{(p)}(n,S-1) + \Psi^{(p)}(n-1,S+2L)+\frac{1}{N^{1/2}}\Psi^{(p)}(n,S+2L)+ \frac{1}{N}\Psi^{(p)}(n+1,S+1)\\
+&\sum_{r=1}^{p-2}\left(\frac{1}{N^{r-1/2}}+\frac{1}{N^{(r+1)/2}}\right)\Psi^{(p-r)}(n-r,S+2L)\\
+&\sum_{r=1}^{p-2} \frac{1}{N^{(r-1)/2}} \Psi^{(p-r)}(n-r-1,S+2L) + N^{-(p-2)/2+\xi} \left( N^{-1} + \bm{1}_{n\ge p}\right),
\end{split}
\label{eq:global_master}
\end{equation}
where $L=L^{(p)}(S):=p+S+4$.
\end{proposition}

We refer to the inequalities in Proposition \ref{prop:global_master} as \emph{master inequalities}, adopting the terminology from \cite{Multi_res_llaws}, where a similar hierarchy \cite[Eq. (3.20a), (3.20b)]{Multi_res_llaws} was employed to prove \eqref{eq:kG_av}, \eqref{eq:kG_iso} as well as more general multi-resolvent local laws for Wigner matrices. The key difference is that in \cite{Multi_res_llaws} the hierarchy was used to control the magnitude of fluctuations of a product of resolvents, while now we control cumulants, that is specific combinations of expected values. Note that cumulants are typically much smaller than fluctuations, hence much higher precision is needed for our new hierarchy.

Now we prove Proposition~\ref{prop:global_Psi} by gradually improving the a priori bound \eqref{eq:Psi_a_priori} using the master inequalities \eqref{eq:global_master}. First, we show that
\begin{equation}
\Psi^{(p)}(n,S)\lesssim N^\xi\left( N^{-d/2} + N^{(n-p)/2}N^{-(p-2)/2}\right),
\label{eq:Psi_improved_iterative}
\end{equation}
 with the implicit constant which does not depend on $N$,
but may depend on $p,d,n,S\in\N\cup\{0\}$. \nc Taking $d$ in~\eqref{eq:Psi_improved_iterative} sufficiently large, we obtain \eqref{eq:Psi_improved} for $n\ge p$ and a slightly weaker version of~\eqref{eq:Psi_improved} for $n= p-1$. In the end of this proof we will show how to gain the missing factor $N^{-1/2}$ for $n=p-1$.

We prove \eqref{eq:Psi_improved_iterative} by induction on $(p,d,n,S)\in (\N\cup\{0\})^4$ using \eqref{eq:global_master}
for  the inductive step
and some straightforward base cases. 
We need two ingredients: (i) a linear order $
\lhd$ on $(\N\cup\{0\})^4$, and (ii) a finite subset $\mathscr{S}\subset (\N\cup\{0\})^4$ on which the induction is carried out to ensure that $p,d,n,S$ remain $N$-independent throughout the proof. First, 
for any two elements of $(\N\cup\{0\})^4$ we set
\begin{equation}
(p_1,d_1,n_1,S_1) \lhd (p_2,d_2,n_2,S_2) \Longleftrightarrow
\begin{cases}
p_1<p_2,&\text{or}\\
p_1=p_2,\, d_1<d_2,&\text{or}\\
p_1=p_2,\, d_1=d_2,\,n_1<n_2,&\text{or}\\
p_1=p_2,\, d_1=d_2,\,n_1=n_2,\, S_1<S_2.
\end{cases}
\label{eq:def_lexicographic}
\end{equation}
This defines a linear lexicographic order on $(\mathbf{N}\cup\{0\})^4$. Similarly, we define the lexicographic order on $(\N\cup\{0\})^3$ with coordinates $(p,d,n)$ and denote it by the same symbol $\lhd$.  Next, fix some $\bm{s}_0:=(p_0,d_0,n_0,S_0)\in(\N\cup\{0\})^4$ and think of this as the quartet of parameters where we want to prove \eqref{eq:Psi_improved_iterative}. One may assume that $n_0\le 2p_0-3$, since in the complementary case \eqref{eq:Psi_improved_iterative} immediately follows from the a priori bound \eqref{eq:Psi_a_priori}. Similarly, on may assume that $d_0\le p_0$: otherwise $d_0$ can be replaced by $p_0$ and the rhs. of \eqref{eq:Psi_improved_iterative} will remain of the same order due to the second term in it which is larger than $N^{-d_0/2}$. \nc To prove \eqref{eq:Psi_improved_iterative} for $\bm{s}_0$, introduce an  auxiliary (finite) set
\begin{equation}
\mathscr{S}_0=\mathscr{S}_0(p_0)\nc:=\left\lbrace (p,d,n)\in(\N\cup\{0\})^3\,:\, p\le p_0,\, n\le 2p_0-3,\, d \le p_0\right\rbrace.
\end{equation}
Write the elements of $(\mathscr{S}_0,\lhd)$ in the decreasing order as $\bm{a}_1\rhd\ldots\rhd\bm{a}_R$ with $R:=|\mathscr{S}_0|$. Define the function $S_{\max} = S_{\max}^{(p_0,S_0)}\nc:\mathscr{S}_0\to\N$  recursively:
\begin{equation}
S_{\max}^{(p_0,S_0)}(\bm{a}_1):=S_0,\qquad S_{\max}^{(p_0,S_0)}(\bm{a}_j):=3S_{\max}^{(p_0,S_0)}(\bm{a}_{j-1})+2(p_0+4),\,\,\text{for}\,\, 2\le j\le R.
\end{equation}
Finally, we set  
\begin{equation}
\mathscr{S}=\mathscr{S}(p_0,S_0)\nc:=\left\lbrace (p,d,n,S)\in (\N\cup\{0\})^4\,:\, (p,d,n)\in\mathscr{S}_0,\, 
S\le S_{\max}^{(p_0,S_0)}(p,d,n)\right\rbrace.
\end{equation}
By construction, $\bm{s}_0\in \mathscr{S}$.\nc

We argue by induction on $(\mathscr{S},\lhd)$. The base cases are given by the elements of $\mathscr{S}$ such that $pdnS=0$. Indeed, if at least one of the parameters $p,n,S$ equals to zero, then $\Psi^{(p)}(n,S)=0$ by definition and \eqref{eq:Psi_improved_iterative} holds. For $d=0$, \eqref{eq:Psi_improved_iterative} immediately follows from the a priori bound \eqref{eq:Psi_a_priori}. Our inductive step is
\begin{equation}
\eqref{eq:Psi_improved_iterative} \text{ holds for all } \bm{s}'\in\mathscr{S}\text{ with }\bm{s}'\lhd \bm{s}\Longrightarrow \eqref{eq:Psi_improved_iterative} \text{ holds for } \bm{s}=(p,d,n,S)\in\mathscr{S}.  
\label{eq:induction}
\end{equation}
We obtain \eqref{eq:induction} from \eqref{eq:global_master} applied to $\Psi^{(p)}(n,S)$. The inhomogeneous term in the rhs. of \eqref{eq:global_master} is estimated from above by $1$. The first two terms in the rhs. of \eqref{eq:global_master} are estimated by the inductive assumption \eqref{eq:Psi_improved_iterative} for $(p,d,n,S-1)$ and $(p,d,n-1,S+2L)$. These elements may not be in $\mathscr{S}$ only for the trivial reasons $S-1<0$ and $n-1<0$, respectively, where the corresponding $\Psi$'s equal to zero by definition. We omit the discussion of similar degeneracies for the remaining terms $\Psi^{(p')}(n',S')$ in the rhs. of \eqref{eq:global_master}. Estimating these terms by \eqref{eq:Psi_improved_iterative} for $(p',d-1,n',S')$, we finish the proof of~\eqref{eq:induction}.

To cover the case $n=p-1$ in \eqref{eq:Psi_improved}, we also prove a closely related to \eqref{eq:Psi_improved_iterative} second bound, namely that
\begin{equation}
\Psi^{(p)}(n,S)\lesssim N^\xi\left( N^{-d/2} + N^{n-p}N^{-(p-2)/2}\right)
\label{eq:Psi_improved_iterative1}
\end{equation}
holds for any $p, d, n, S$. In fact, the proof of~\eqref{eq:Psi_improved_iterative1} is identical to that of~\eqref{eq:Psi_improved_iterative}, except that the inhomogeneous term $\left( N^{-1} + \bm{1}_{n\ge p}\right)$ in \eqref{eq:global_master} is overestimated by $N^{n-p}$ (recall that one may assume  $n\ge p-1$ in \eqref{eq:global_master}, otherwise $\Psi^{(p)}(n, S)=0$ by definition and no need to use  \eqref{eq:global_master}). Now the proof of~\eqref{eq:Psi_improved} for $n=p-1$ directly follows from~\eqref{eq:Psi_improved_iterative1}. This finishes the proof of Proposition~\ref{prop:global_Psi}.\hfill\qed

\subsection{Proof of Proposition~\ref{prop:global_master}: preliminaries}\label{sec:Psi_preliminaries}

In this section we collect some basic notations and facts needed for the proof of Proposition~\ref{prop:global_master}. First, for a multi-index $\bm{k}=(k_i)_{i=1}^n\in\N^n$ and a set of indices $I\subset [n]$, we define
\begin{equation}
\bm{k}^{[I]}:=(k_i)_{i\in [n]\setminus I}.
\end{equation}
For a finite set $S=\{s_1,\ldots,s_n\}$ and $p\in\N$ denote the set of partitions of $S$ into $p$ subsets by
\begin{equation}
P_p(S):= \left\lbrace \{S_j\}_{j=1}^p: S_j\subset S,\, S=\cup_{j=1}^p S_j,\, S_i\cap S_j=\emptyset,\,\forall i\neq j\right\rbrace.
\end{equation}
Here by disjointness of $S_i$ and $S_j$ for $i\neq j$ we mean that there is no element $s_k$ simultaneously belonging to $S_i$ and $S_j$. However, we do not exclude the possibility that an element of $S_i$ is equal to some element of $S_j$, though these elements have different indices (i.e. the partition is really on the index set $1, 2,\ldots, n$ of the elements $s_1, \ldots, s_n$, but we ignore this subtle difference for simplicity). We also note that some of the subsets $S_j$ may be empty. 

For a set of index pairs $\bm{\alpha}=(\alpha_1,\ldots,\alpha_k)\in([N]^2)^k$ denote the normalized cumulant of the corresponding elements of $W$ by
\begin{equation*}
\kappa(\bm{\alpha})=\kappa(\alpha_1,\ldots,\alpha_k):=\kappa(\sqrt{N}w_{\alpha_1},\ldots,\sqrt{N}w_{\alpha_k}).
\end{equation*}
Since the entries of $W$ are independent up to the symmetry constraint, $\kappa(\bm{\alpha})$ does not vanish only when $\alpha_1,\ldots,\alpha_k\in\{ab,ba\}$ for some $a,b\in[N]$. We also have from the moment assumption \eqref{eq:momass} that $|\kappa(\bm{\alpha})|\le C_k’$ for some constants depending only on $k$. Let $\partial_{\alpha}$ be the derivative with respect to $w_{\alpha}$. Denote for short 
\begin{equation*}
\partial_{\bm{\alpha}}=\partial_{\alpha_1}\cdots \partial_{\alpha_k}\quad \text{for}\quad \bm{\alpha}=(\alpha_1,\ldots,\alpha_k).
\end{equation*}

The following lemma shows how to perform static expansions of cumulants. 

\begin{lemma}\label{lem:underline_exp_general} Assume the set-up and conditions of Proposition~\ref{prop:global}. Fix $p\in\N$ with $p\ge 2$, let $A\in\C^{N\times N}$  be a deterministic matrix, and consider
\begin{equation}
F:=G^kA\,\mathrm{Iso}(\bm{k}_0,\bm{\phi})\Delta(\bm{k}_1)\overline{\mathrm{Iso}(\bm{l}_0,\bm{\phi}')\Delta(\bm{l}_1)},\qquad f_j=\Delta(\bm{k}_j)\overline{\Delta(\bm{l}_j)}, \, \, j\in [2,p],
\label{eq:def_Ff}
\end{equation}
with fixed $k\in \N$, $\bm{k}_j\in\N^{n_{1,j}}$, $\bm{l}_j\in \N^{n_{2,j}}$ with $n_{1,j},n_{2,j}\in\N\cup\{0\}$ and sets of index pairs $\bm{\phi},\bm{\phi}'$. Then uniformly in $z\in\mathcal{D}_g$ and $A$ with $\|A\|\lesssim 1$ it holds that    
\begin{equation}
\begin{split}
&\kappa\left(\mathrm{Tr}\left[\underline{WF}\right] ,f_2,\ldots,f_p\right) =\sum_{j=2}^p \kappa\left(\widetilde{\E}\left[\mathrm{Tr}\big[\widetilde{W}F\big]\partial_{\widetilde{W}}[f_j]\right],\{f_i\}_{i\neq j}\right)\\
&\quad -\sum_{\ell=3}^L \sum_{a,b=1}^N \sum_{\bm{\alpha}\in\{ab,ba\}^{\ell-1}}\frac{\kappa(ab,\bm{\alpha})}{(\ell-1)!N^{\ell/2}}\sum_{P_p(\bm{\alpha})}\kappa\left(\Bigg(\partial_{\bm{\alpha}_1}F_{ba} \prod_{\bm{\alpha_j}\neq\emptyset} \partial_{\bm{\alpha}_j}f_j\Bigg),\left\lbrace f_j\right\rbrace_{\bm{\alpha}_j=\emptyset}\right) + \mathcal{O}\left(N^{-D}\right),
\end{split}
\label{eq:underline_exp_general}
\end{equation}
for any fixed $D>0$. The positive integer $L$ does not depend on $N$ and $A$, but may depend on $D$, parameters from \eqref{eq:def_Ff} and the parameters of the domain $\mathcal{D}_g$ defined in \eqref{eq:def_D_global}. In \eqref{eq:underline_exp_general}, $\widetilde{W}$ is an independent copy of $W$ and the expectation $\widetilde{\E}$ is taken with respect to $\widetilde{W}$. 
\end{lemma}

Notably, all terms in the rhs. of \eqref{eq:underline_exp_general} apart from the negligible error term are again cumulants of some random variables. However, some of them may degenerate into cumulants of one variable, that is expectations. We will treat these terms separately. We stated Lemma~\ref{lem:underline_exp_general} for general deterministic matrix $A$, however, in the proof of Proposition~\ref{prop:global_master} we will encounter only two cases: (i) $A$ is the identity matrix and (ii) $A$ has only one non-zero entry. The proof of Lemma~\ref{lem:underline_exp_general} is based on \eqref{eq:def_cum}, \eqref{eq:def_under_1} and the cumulant expansion, for more details see Appendix~\ref{app:underline}. Finally, we note that the special form of the arguments~\eqref{eq:def_Ff} is used only in the estimate on the error term in \eqref{eq:underline_exp_general}, the structure of \eqref{eq:underline_exp_general} is not specific to these functions.

Throughout the proof of Proposition~\ref{prop:global_master}  for simplicity we consider  only the cumulants of the form \eqref{eq:extended_family} where there are no conjugates:
\begin{equation}
\kappa\left(\mathrm{Iso}(\bm{k}_0,\bm{\phi})\Delta(\bm{k}_1), \left\lbrace \Delta(\bm{k}_j)\right\rbrace_{j=2}^p\right).
\label{eq:extended_family_l=0}
\end{equation}
This is not restrictive in the global regime. In the general case the only additional technicality is that the traces and matrix entries of mixed powers of the form $G^k(G^*)^l$ may arise. To deal with them, we apply the resolvent identity several times to this product, and write it as a sum of pure powers of $G$ and $G^*$ (see e.g. Lemma~\ref{lem:resolvent_id} later). This decomposition is done at the cost of some positive powers of $\eta$ appearing in denominator, but in the global regime $\eta\sim 1$, so these factors are harmless. Thus, the general case does not carry any additional difficulties compared to~\eqref{eq:extended_family_l=0} which we consider in the proof.

In the proof of Proposition~\ref{prop:global_Psi} we consider three types of cumulants in \eqref{eq:extended_family_l=0}:
\begin{equation}
(1)\,\, \bm{k}_0=\emptyset,\qquad (2)\,\, \mathrm{od}(\bm{\phi})>0,\qquad (3)\,\, \mathrm{od}(\bm{\phi})=0.
\end{equation}
In other words, in the first case all arguments of the cumulant in \eqref{eq:extended_family_l=0} are products of traces and no isotropic quantities arise. In the second case, the first argument of the cumulant contains a product of resolvent entries and some of them are off-diagonal. Finally, in the third case each of these entries is diagonal. This division into cases is not the only possible one. Its role  is to separate the first case, where the estimates are technically simpler, though capture all the ideas needed for the proof in general case. We discuss this case in detail in Section~\ref{sec:tracial_expansion} and then explain the adjustments required in the second and third cases in Sections~\ref{sec:off_diag_expansion} and~\ref{sec:diag_expansion}, respectively.

\subsection{Proof of Proposition~\ref{prop:global_master}: expansion of purely tracial quantities.}\label{sec:tracial_expansion}

For $\bm{k}_0=\emptyset$, \eqref{eq:extended_family_l=0} is of the form
\begin{equation}
\kappa\left(\Delta(\bm{k}_1),\ldots,\Delta(\bm{k}_p)\right),\quad\text{where}\quad p\ge 3,\,\,\, \mathfrak{n}(\bm{k}_j)=n_j,\,\,\, n_1+\ldots+n_p\le n,\,\,\, |\bm{k}_1|+\ldots+|\bm{k}_p|\le S.
\label{eq:kappa_under_consideration}
\end{equation}
Since we consider the case when there are no isotropic quantities in \eqref{eq:extended_family}, $S$ is simply an upper bound on the number of resolvents in \eqref{eq:extended_family}. We proceed in several steps. First, in Section~\ref{sec:expansion_1arg} we derive  that (recall the definition of $m_k$ 
from~\eqref{def:mk})
\begin{equation}
\begin{split}
\Delta(\bm{k}_1) =& \frac{m}{1-m^2} \frac{2}{\beta}\sum_{i=2}^{n_1} \left(m_{k_{11}+k_{1i}+1}\Delta\big(\bm{k}_1^{[1,i]}\big) +\frac{1}{N}\Delta\big(k_{11}+k_{1i}+1,\bm{k}_1^{[1,i]}\big) \right)\\
+& \frac{m}{1-m^2}\left(\bm{1}_{\beta=1}k_{11}m_{k_{11}+1}\Delta\big(\bm{k}_1^{[1]}\big) + \Delta\big(k_{11}-1, \bm{k}_1^{[1]}\big) + 2\sum_{l=1}^{k_{11}-1} m_{k_{11}+1-l} \Delta\big(l,\bm{k}^{[1]}\big)\right)\\
 +&\frac{m}{1-m^2}\left(\frac{1}{N}\sum_{l=1}^{k_{11}}\Delta\big(l,k_{11}+1-l,\bm{k}_1^{[1]}\big) + \bm{1}_{\beta=1} \frac{k_{11}}{N}\Delta\big(k_{11}+1,\bm{k}_1^{[1]}\big) - \underline{\mathrm{Tr}\left[WG^{k_{11}}\right]\Delta\big(\bm{k}^{[1]}_1\big)}\right).
\end{split}
\label{eq:Delta_bmk_expansion}
\end{equation}
We plug this identity into \eqref{eq:kappa_under_consideration} and bound all terms using \eqref{eq:def_Psi} apart from the ones containing the  underline term. Next, in Section~\ref{sec:hit_two_args} we apply Lemma~\ref{lem:underline_exp_general} to handle the underline and estimate all terms in the first line of \eqref{eq:underline_exp_general} using \eqref{eq:def_Psi}. In Section~\ref{sec:cum_exp_last_term} we further estimate the terms in the second line of \eqref{eq:underline_exp_general}, where all arguments of the cumulant in \eqref{eq:kappa_under_consideration} are differentiated out, in which case the cumulant is reduced to the expectation. These estimates do not involve the control parameters \eqref{eq:def_Psi} and rely purely on the local laws \eqref{eq:kG_av}--\eqref{eq:kG_iso}. Finally, in Section~\ref{sec:cum_exp_terms} we estimate the remaining terms in the second line of \eqref{eq:underline_exp_general} again using \eqref{eq:def_Psi}.

\subsubsection{Expansion within one argument of the cumulant}\label{sec:expansion_1arg} The proof of \eqref{eq:Delta_bmk_expansion} is based on the identity
\begin{equation}
G = m+ m\langle G-m\rangle G +\bm{1}_{\beta=1} \frac{m}{N}G^2 -m\underline{WG},
\label{eq:G_underline}
\end{equation}
which immediately follows from \eqref{eq:def_under_1}, $WG- zG =I$ and the relation $m^2 +mz+1=0$. This identity is used for the first $G$ in $G^k=G\cdot G^{k-1}$ with $k=k_{11}$, and then the underline is extended onto the entire power $G^{k}$ via \eqref{eq:def_under_1}:
\begin{equation}
\begin{split}
\Delta(k) =& \frac{m}{1-m^2}\left(\bm{1}_{\beta=1}km_{k+1} + \Delta(k-1) + 2\sum_{l=1}^{k-1} m_{k+1-l} \Delta(l)\right)\\
 +&\frac{m}{1-m^2}\left(\frac{1}{N}\sum_{l=1}^k\Delta(l,k+1-l) + \bm{1}_{\beta=1} \frac{k}{N}\Delta(k+1) - \mathrm{Tr}\left[\underline{WG^k}\right]\right).
\end{split}
\label{eq:Delta_k}
\end{equation}
To expand the renormalization further on the factors $\Delta(k_{1j})$ with $j\in [2,n_1]$, we observe that
\begin{equation}
\E\left[\mathrm{Tr}[WA]\mathrm{Tr}[WB]\right] = \frac{1}{N}\left(\mathrm{Tr}[AB] +\bm{1}_{\beta=1}\mathrm{Tr}[AB^\mathfrak{t}]\right),
\label{eq:EWW}
\end{equation}
for any $A,B\in\C^{N\times N}$ independent of $W$. Therefore,
\begin{equation}
\mathrm{Tr}\left[\underline{WG^k}\right] \Delta(l) = \underline{\mathrm{Tr}[WG^k]\Delta(l)}  -\frac{2}{\beta N} \mathrm{Tr}\left[G^{k+l+1}\right],
\label{eq:under_extension_basic}
\end{equation}
for any $k,l\in \N$. The remaining details of the proof of \eqref{eq:Delta_bmk_expansion} are analogous to \cite[Eq.(7.44)-(7.47)]{hyperuniformity} and thus are omitted.

Next, we substitute \eqref{eq:Delta_bmk_expansion} into \eqref{eq:kappa_under_consideration} and estimate the sum of all terms apart from the one containing the underline from above by 
\begin{equation}
\Psi^{(p)}(n,S-1)+\frac{1}{N}\Psi^{(p)}(n+1,S+1).
\label{eq:err_1_1}
\end{equation}
Here we additionally used that
\begin{equation}
| m_k|\lesssim 1\quad\text{and}\quad \left\vert 1-m^2\right\vert^{-1}\lesssim 1,
\end{equation}
in the global regime $\eta\sim 1$, for any fixed $k\in\N$.

\subsubsection{Expansion of the renormalization to the entire cumulant}\label{sec:hit_two_args}
Now we are left with the term
\begin{equation}
\kappa\left(\underline{\mathrm{Tr}\left[WG^{k_{11}}\right]\Delta\big(\bm{k}^{[1]}_1\big)}, \Delta\big(\bm{k}_2\big),\ldots, \Delta\big(\bm{k}_p\big)\right),
\end{equation}
which we expand further using Lemma~\ref{lem:underline_exp_general}. Due to \eqref{eq:EWW}, the terms corresponding to the first line of \eqref{eq:underline_exp_general} are of the form
\begin{equation}
\frac{1}{N}\kappa\left(\Delta\big(k_{11}+k_{ji}+1,\bm{k}_1^{[1]},\bm{k}_j^{[i]}\big),\{\Delta\big(\bm{k}_l\big)\}_{l\neq 1,j}\right)+ m^{[k_{11}+k_{ji}+1]}\kappa\left(\Delta\big(\bm{k}_1^{[1]},\bm{k}_j^{[i]}\big),\{\Delta\big(\bm{k}_l\big)\}_{l\neq 1,j}\right),
\end{equation}
for $j\in [2,p]$ and $i\in [n_j]$. Therefore, the sum of these terms has an upper bound of order
\begin{equation}
\frac{1}{N}\Psi^{(p-1)}(n-1,S+1)+\Psi^{(p-1)}(n-2,S-2).
\label{eq:err_1_2}
\end{equation}

\subsubsection{Analysis of the expectation in the second line of \eqref{eq:underline_exp_general}}\label{sec:cum_exp_last_term}
Consider the terms in the second line of \eqref{eq:underline_exp_general} with $\bm{\alpha}_j\neq\emptyset$ for all $j\in [2,p]$. These are the terms where derivatives hit every argument of the cumulant~\eqref{eq:kappa_under_consideration}, so the cumulant structure is erased and we are left with expectation:
\begin{equation}
\sum_{a,b=1}^N\frac{\kappa(ab,\bm{\alpha})}{(\ell-1)!N^{\ell/2}}\E\bigg[\partial_{\bm{\alpha}_1}\Big[\left(G^{k_{11}}\right)_{ba}\Delta\big(\bm{k}_1^{[1]}\big)\Big] \prod_{j=2}^p \partial_{\bm{\alpha}_j}\Delta\big(\bm{k}_j\big)\bigg].
\label{eq:cum_exp_error_term}
\end{equation}
We do not estimate~\eqref{eq:cum_exp_error_term} in terms of the control parameters~\eqref{eq:def_Psi}, but instead bound it directly by
\begin{align}
|\eqref{eq:cum_exp_error_term},a=b|&\lesssim N^{-(\ell-2)/2+\xi}\le N^{-(p-2)/2+\xi},\label{eq:cum_exp_error_term_bound_d}\\
|\eqref{eq:cum_exp_error_term},a\neq b|&\lesssim N^{-(p-2)+\xi}.
\label{eq:cum_exp_error_term_bound_od}
\end{align}
Here by $a=b$ and $a\neq b$ we indicated which terms are considered in the summation in \eqref{eq:cum_exp_error_term}.

To verify~\eqref{eq:cum_exp_error_term_bound_d}, we bound all derivatives in \eqref{eq:cum_exp_error_term} by $N^\xi$ by the means of~\eqref{eq:kG_av}--\eqref{eq:kG_iso}, and take into account that there are $N$ pairs of indices $(a,b)$ with $a=b$. This implies\footnote{Strictly speaking we obtain some $N^{C\xi}$ factor instead of $N^\xi$, but since
$\xi>0$ is arbitrary when we estimated the derivatives by $N^\xi$, 
this constant $C$ is irrelevant as it can be removed by reducing the small exponent $\xi>0$.
We will indicate all such factors generically by $N^\xi$.}  the first bound in \eqref{eq:cum_exp_error_term_bound_d}. To obtain the second one, it suffices to notice that $\ell\ge p$, since $\mathfrak{n}(\bm{\alpha}_j)\ge 1$ for all $j\in [2,p]$.

The proof of \eqref{eq:cum_exp_error_term_bound_od} relies on the following observation: for any $s\in \N$, $\bm{\alpha}\in\{ab,ba\}^s$ with $a\neq b$ and a multi-index $\bm{k}$ it holds that
\begin{equation}
N^{-s/2} \left\vert \partial_{\bm{\alpha}} \Delta(\bm{k})\right\vert \lesssim N^{-1+\xi}.
\label{eq:final_term_trick1} 
\end{equation}
For $s\ge 2$, \eqref{eq:final_term_trick1} directly follows from \eqref{eq:kG_av}--\eqref{eq:kG_iso}, while for $s=1$ on top of $N^{-s/2}$ we gain an additional factor $N^{-1/2}$ from an off-diagonal entry due to~\eqref{eq:kG_iso}. Similarly to \eqref{eq:final_term_trick1}, we get that
\begin{equation}
N^{-\mathfrak{n}(\bm{\alpha}_1)/2}\left\vert\partial_{\bm{\alpha}_1}\Big[\left(G^{k_{11}}\right)_{ba}\Delta\big(\bm{k}_1^{[1]}\big)\Big]\right\vert \lesssim N^{-1/2+\xi}.
\label{eq:alpha1_bound}
\end{equation}
Combining this bound with \eqref{eq:final_term_trick1} for $\bm{\alpha}=\bm{\alpha}_j$, $\bm{k}=\bm{k}_j$, $j\in [2,n]$, and recalling that $\ell= 1+\mathfrak{n}(\bm{\alpha}_1)+\ldots+\mathfrak{n}(\bm{\alpha}_p)$, we obtain~\eqref{eq:cum_exp_error_term_bound_od}.

\subsubsection{Analysis of the second and higher order cumulants in the second line of \eqref{eq:underline_exp_general}}\label{sec:cum_exp_terms}

Now we consider the terms in the second line of \eqref{eq:underline_exp_general} where there is at least one $j\in [2,p]$ such that $\bm{\alpha}_j=\emptyset$. We bound each of these terms by its absolute value, estimate $|\kappa(ab,\bm{\alpha})|\lesssim 1$, and arrive to the sum
\begin{equation}
N^{-\ell/2}\sum_{a,b=1}^N\Bigg|\kappa\Bigg(\bigg(\partial_{\bm{\alpha}_1}\Big[\left(G^{k_{11}}\right)_{ba}\Delta\big(\bm{k}_1^{[1]}\big)\Big] \prod_{\bm{\alpha}_j\neq \emptyset} \partial_{\bm{\alpha}_j}\Delta\big(\bm{k}_j\big)\bigg),\left\lbrace \Delta\big(\bm{k}_j\big)\right\rbrace_{\bm{\alpha}_j=\emptyset}\Bigg)\Bigg|.
\label{eq:cum_intermediate}
\end{equation}
Here $j$ takes values in $[2,p]$. 

First, we estimate the terms in \eqref{eq:cum_intermediate} with $\ell\ge L^{(p)}(S)=p+S+4$. These estimates rely purely on the local laws \eqref{eq:kG_av}--\eqref{eq:kG_iso} and their immediate corollaries, and do not involve control parameters \eqref{eq:def_Psi}. By \eqref{eq:Delta_a_priori} it holds that $|\Delta(\bm{k}_j)|\lesssim N^\xi$ with very high probability, which we use for $\bm{\alpha}_j=\emptyset$. Similarly to \eqref{eq:final_term_trick1} and \eqref{eq:alpha1_bound} we get that each of the factors in \eqref{eq:cum_intermediate} containing a derivative has an upper bound of order $N^\xi$. Therefore, \ref{eq:cum_intermediate} is bounded by
\begin{equation}
N^{-\ell/2} N^2 N^{C \xi} \lesssim N^{-p/2}\quad \text{for}\quad \ell\ge L^{(p)}(S).
\label{eq:high_order_term}
\end{equation}
Note that \eqref{eq:high_order_term} even holds under the weaker constraint $\ell\ge p+4$. The term $S$ was included into the definition of $L^{(p)}(S)$ only to unify the estimates with the ones in Sections~\ref{sec:off_diag_expansion} and~\ref{sec:diag_expansion} below. 

The terms in \eqref{eq:cum_intermediate} with $\ell<L^{(p)}(S)=:L$ require a more delicate treatment. \nc We will estimate each of the cumulants in \eqref{eq:cum_intermediate} by the appropriate control parameter from \eqref{eq:def_Psi} (or a sum of several $\Psi$'s). We perform differentiation in $\bm{\alpha}_j$ in \eqref{eq:cum_intermediate} for all $j\in [1,p]$, and obtain an $N$-independent number of terms for each $a,b\in [N]$.  Each of these terms is associated to a certain allocation of derivatives over resolvents according to the Leibniz rule. Since the number of these allocations does not depend on $N$, it suffices to show that for each of them the sum over $a,b$ in \eqref{eq:cum_intermediate} is bounded by the rhs. of \eqref{eq:global_master}. We further treat the allocation of derivatives as fixed, even if it is not mentioned explicitly.

Denote $r:=|\{j\in [2,p]:\bm{\alpha}_j\neq \emptyset\}|$, so that the cumulants in \eqref{eq:cum_intermediate} have $p-r$ arguments. The number of traces remaining after the differentiation is at most $n-1-r$ (recall that one trace is replaced by $(G^{k_{11}})_{ba}$). Let $s\in\N\cup\{0\}$ be such that this number equals to $n-1-r-s$. Next, the total number of derivatives in \eqref{eq:cum_intermediate} equals to $\ell-1$, we have that $t:=\ell-1-r-s\ge 0$. Thus,
\begin{equation}
\ell=r+s+t+1.
\label{eq:ell_rst}
\end{equation}
Since \eqref{eq:cum_intermediate} does not depend on the order in which derivatives are taken, it is convenient to differentiate in the following order. First, we choose $r+s$ traces in the first argument of the cumulant in \eqref{eq:cum_intermediate} in such a way that at least one trace is chosen from $\Delta(\bm{k}_j)$ for each $j\in [2,p]$. We differentiate each of these $r+s$ traces once and $n-r-s-1$ traces remain. Afterwards, we perform the remaining $t$ differentiations by hitting only resolvent entries, but not traces.

Now we estimate \eqref{eq:cum_intermediate} in terms of \eqref{eq:def_Psi} and distinguish between the following cases. 

\smallskip

\noindent\underline{$a\neq b$, $r\ge 1$.} We claim that the corresponding sum in \eqref{eq:cum_intermediate} has an upper bound of order \nc
\begin{equation}
\sum_{r=1}^{p-2}\frac{1}{N^{r-1}} \Psi^{(p-r)}(n-r-1,S+2L).
\label{eq:cum_intermediate1}
\end{equation}
 Here the factor $N^{-(r-1)}$ emerges from a combination of three effects. First, the number of terms in the summation in \eqref{eq:cum_intermediate} is of order $N^2$. Next, we take into account the factor $N^{-\ell/2}$. Finally, and most importantly, we gain $N^{-1/2}$ from each off-diagonal entry arising after differentiation due to the normalization in \eqref{eq:def_Iso} (but not due to the local law~\eqref{eq:kG_iso}). Since each differentiation of a trace creates an off-diagonal entry, and each differentiation of a resolvent entry destroys at most one, the number of off-diagonal entries is bounded from below by  $r+s+1-t$ (this number may be even negative). We bound $s\ge 0$, combine the three effects and recall the definition of $\Psi$ from \eqref{eq:def_Psi}. To complete the proof of \eqref{eq:cum_intermediate1} it remains to observe that each derivative increases the number of resolvents by one and the number of diagonal entries in the product at most by one. This explains why $\Psi$ arises in \eqref{eq:cum_intermediate1} with the second argument equal to $S+2L$.

\smallskip

\noindent\underline{$a\neq b$, $r=0$.} Though the bound \eqref{eq:cum_intermediate1} is valid also for $r=0$, it is not sufficiently precise for the proof of Proposition~\ref{prop:global_master}. Thus, we argue differently. If $\ell\ge 4$, we do not need to gain from off-diagonal entries, and simply bound the sum of the corresponding terms in \eqref{eq:cum_intermediate} by
\begin{equation}
\Psi^{(p)}(n-1,S+2L).
\label{eq:cum_intermediate2}
\end{equation}
If $\ell=3$, then there is at least one off-diagonal entry in \eqref{eq:cum_intermediate}, since index $a$ appears exactly three times in the entries obtained after differentiation (and similarly for $b$). Therefore, we gain at least one factor $N^{-1/2}$ and again obtain the upper bound of order \eqref{eq:cum_intermediate2}.

\smallskip

\noindent\underline{$a=b$, $r\ge 1$.} Only diagonal entries arise after differentiation in $w_{aa}$, thus we do not gain from off-diagonality. We simply estimate $\ell\ge r+1$ and take into account that the number of pairs $(a,b)$ with $a=b$ equals to $N$, arriving at the bound
\begin{equation}
\sum_{r=1}^{p-2}\frac{1}{N^{(r-1)/2}} \Psi^{(p-r)}(n-r-1,S+2L).
\label{eq:cum_intermediate3}
\end{equation}

\smallskip

\noindent\underline{$a=b$, $r=0$.} Estimating $\ell\ge 3$, we get the bound of order \eqref{eq:cum_intermediate2} with an additional small factor $N^{-1/2}$.

\smallskip

Collecting the estimates proved in Section~\ref{sec:tracial_expansion}, we conclude that the absolute value of \eqref{eq:kappa_under_consideration} is bounded by the rhs. of \eqref{eq:global_master} up to some $N$-independent multiplicative constant.\nc

\subsection{Proof of Proposition~\ref{prop:global_master}: expansion with respect to an off-diagonal entry}\label{sec:off_diag_expansion}

In this section we prove an upper bound on \eqref{eq:extended_family_l=0} in the case when $\bm{k}_0\neq\emptyset$ and $\mathrm{od}(\bm{\phi})>0$. We write $\bm{k}_0=(k_{0i})_{i=1}^{n_0}$ and $\bm{\phi}=(\phi_i)_{i=1}^{n_0}$ for $n_0:=\mathfrak{n}(\bm{k}_0)$. Without loss of generality one may assume that $\phi_1$ is off-diagonal. Instead of \eqref{eq:Delta_k}, we start with the expansion of $(G^k)_{\phi}$ for $k=k_{01}$ and $\phi=\phi_1$ within the $\mathrm{Iso}(\bm{k}_0,\bm{\phi})$ factor in~\eqref{eq:extended_family_l=0}:
\begin{equation}
\begin{split}
(G^k)_\phi =& m\left( (G^{k-1})_\phi + \sum_{l=1}^{k-1} m_{k+1-l} (G^{l})_\phi\right)\\
 +&\frac{m}{N} \left( \bm{1}_{\beta=1} k (G^{k+1})_\phi + \sum_{l=1}^{k-1} \Delta(k+1-l) (G^{l})_\phi\right) - m\left(\underline{WG^k}\right)_\phi.
\end{split}
\label{eq:iso_G_expansion}
\end{equation}
This identity is a direct consequence of \eqref{eq:G_underline} and \eqref{eq:def_under_1}. From \eqref{eq:EWW} we can extend the underline as follows:
\begin{align}
(\underline{WG^k})_{ab}\Delta(l)= &\frac{2l}{\beta N} (G^{k+l+1})_{ab} - \underline{(WG^k)_{ab}\Delta(l)},\label{eq:case2_ext1}\\
(\underline{WG^k})_{ab} (G^l)_{cd} = & \sum_{r=1}^l \left((G^{l+1-s})_{ad}(G^{k+s})_{cb} + \bm{1}_{\beta=1}(G^{l+1-s})_{ac}(G^{k+s})_{db}\right) - \underline{(WG^k)_{ab}(G^l)_{cd}},\label{eq:case2_ext2}
\end{align}
for any $k,l\in \mathbf{N}$ and $a,b,c,d\in [N]$. Using these inputs, we get that the analogues of the bounds in \eqref{eq:err_1_1} and \eqref{eq:err_1_2} are
\begin{equation}
\frac{1}{\sqrt{N}}\Psi^{(p)}(n,S) + \Psi^{(p)}(n,S-1) + \frac{1}{N}\Psi^{(p)}(n+1,S+1)\quad \text{and}\quad \frac{1}{N}\Psi^{(p-1)}(n-1,S+1).
\label{eq:case2_ext_terms}
\end{equation}
The control parameter $N^{-1/2}\Psi^{(p)}(n,S)$ which is new compared to \eqref{eq:err_1_1} emerges as an upper bound on the term where the underline is extended to an off-diagonal entry in the first argument of \eqref{eq:extended_family} and a diagonal entry arises from \eqref{eq:case2_ext2}. For example, this could happen for $d=a$, $c\neq b$. If both resulting entries are diagonal, then the number of off-diagonal entries is decreased by 2 and the corresponding cumulant is bounded by $\Psi^{(p)}(n,S-1)$.

\subsubsection{Analysis of the expectation in the second line of \eqref{eq:underline_exp_general}}\label{sec:cum_exp_last_term2}

Now we adapt the analysis performed in Section~\ref{sec:cum_exp_last_term} to the current case. Let us write $\phi_1$ as $ab$, $a\neq b$. Instead of \eqref{eq:cum_exp_error_term} we have the following sum:
\begin{equation}
\sum_{c=1}^N\frac{\kappa(ac,\bm{\alpha})}{(\ell-1)!N^{\ell/2}}\E\bigg[\partial_{\bm{\alpha}_1}\Big[N^{1/2}\left(G^{k_{01}}\right)_{cb}\mathrm{Iso}\big(\bm{k}_0^{[1]},\bm{\phi}^{[1]}\big)\Delta\big(\bm{k}_1\big)\Big] \prod_{j=2}^p \partial_{\bm{\alpha}_j}\Delta\big(\bm{k}_j\big)\bigg],
\label{eq:cum_exp_error_term_case2}
\end{equation}
where $\bm{\alpha}\in\{ac,ca\}^{\ell-1}$. We claim that
\begin{equation}
\left\vert \eqref{eq:cum_exp_error_term_case2}\right\vert \lesssim  N^{-(p-2)/2+\xi}\left(N^{-1} + \bm{1}_{n\ge p}\right).
\label{eq:cum_exp_error_term_case2_bound}
\end{equation}

To prove~\eqref{eq:cum_exp_error_term_case2_bound}, we show in addition to~\eqref{eq:final_term_trick1} how to estimate derivatives of diagonal and off-diagonal resolvent entries. For any $s\in \N\cup\{0\}$,  $a,c,d,e\in[N]$ with $d\neq e$, $\bm{\alpha}\in\{ac,ca\}^s$ and $k\in\N$ it holds that
\begin{equation}
\begin{split}
N^{-s/2} \left\vert \partial_{\bm{\alpha}}(G^k)_{dd}\right\vert &\lesssim N^\xi\left( \bm{1}_{s=0}+N^{-1/2}\right),\\
N^{-s/2} \left\vert \partial_{\bm{\alpha}}N^{1/2}(G^k)_{de}\right\vert &\lesssim N^\xi\left( \bm{1}_{s=0}+\bm{1}_{a,c\in\{d,e\}}+N^{-1/2}\right).
\end{split}
\label{eq:final_term_trick2}
\end{equation} 
The proof of~\eqref{eq:final_term_trick2} is analogous to that of~\eqref{eq:final_term_trick1} and thus is omitted. Now 
for the proof of~\eqref{eq:cum_exp_error_term_case2_bound} we consider several cases depending on the value of $c$ in~\eqref{eq:cum_exp_error_term_case2}.

\smallskip

\noindent\underline{$c=b$.} Since $a\neq b$, we get that $a\neq c$. Applying~\eqref{eq:final_term_trick1} to the derivatives of $\Delta(\bm{k}_j)$, $j\in [2,p]$, we gain the factor $N^{-(p-1)}$. Next, we use~\eqref{eq:final_term_trick1} and the worst scenario $N^\xi$ term in the rhs. of~\eqref{eq:final_term_trick2} to bound the derivative in~$\bm{\alpha}_1$:
\begin{equation}
N^{-\bm{\alpha_1}/2}\left\vert\partial_{\bm{\alpha}_1}\Big[N^{1/2}\left(G^{k_{01}}\right)_{cb}\mathrm{Iso}\big(\bm{k}_0^{[1]},\bm{\phi}^{[1]}\big)\Delta\big(\bm{k}_1\big)\Big]\right\vert \lesssim N^{1/2+\xi}.
\end{equation}
Here the factor $N^{1/2}$ arises from the fact that the entry $\left(G^{k_{01}}\right)_{cb}$ is now diagonal and does not require the $N^{1/2}$ normalization factor. Collecting these inputs, we obtain
\begin{equation}
\left\vert\text{term in }\eqref{eq:cum_exp_error_term_case2}\text{ with }c=b\right\vert \lesssim N^{-(p-1)+\xi}\le N^{-1}N^{-(p-2)/2+\xi}\quad\text{for}\quad p\ge 2.
\label{eq:c=b_exp}
\end{equation}

\smallskip

\noindent\underline{$c=a$.} Similarly to~\eqref{eq:final_term_trick1} we have that
\begin{equation}
N^{-s}\left\vert\partial_{aa}^s \Delta(k)\right\vert \lesssim N^{-s},\quad \forall s\in\N\cup\{0\},\, k\in \N.
\end{equation}
Combining this bound with \eqref{eq:final_term_trick2}, where the rhs. is again trivially estimated by $N^\xi$, we get
\begin{equation}
\left\vert\text{term in }\eqref{eq:cum_exp_error_term_case2}\text{ with }c=a\right\vert \lesssim N^{-p/2+\xi}= N^{-1}N^{-(p-2)/2+\xi}.
\end{equation}

\smallskip

\noindent\underline{$c\neq a,b$ and $c$ appears in $\bm{\phi}$.} By the latter condition we meant that $c$ is one of the two indices in $\phi_i$ for some $\phi_i$ in $\bm{\phi}$. Using \eqref{eq:final_term_trick1}, \eqref{eq:final_term_trick2} with the rhs. estimated by $N^\xi$ and observing that the number of indices $c$ in $\bm{\phi}$ does not depend on $N$, we obtain the same bound as in~\eqref{eq:c=b_exp} (actually, we gain an additional factor $N^{-1/2}$, but this is not needed for our purposes). 

\smallskip

\noindent\underline{$c$ does not appear in $\bm{\phi}$.} In this case we eventually use~\eqref{eq:final_term_trick2} with $N^{-1/2+\xi}$ in the rhs. for $s>0$. Together with \eqref{eq:final_term_trick1}, this gives
\begin{equation}
|\eqref{eq:cum_exp_error_term_case2}, \text{sum over } c\text{ not in } \bm{\phi}|\lesssim  N \cdot N^{-1/2} N^{-\mathfrak{n}(\bm{\alpha}_1)/2} N^{-(p-1)+\xi}.
\label{eq:final_term_trick2_appl}
\end{equation}
where the first factor $N$ comes from the number of indices of summation, $N^{-1/2}$ is gained using that $\ell-\mathfrak{n}(\bm{\alpha}_1)-\ldots-\mathfrak{n}(\bm{\alpha}_p)=1$, $N^{-\mathfrak{n}(\bm{\alpha}_1)/2}$ arises from \eqref{eq:final_term_trick1}, \eqref{eq:final_term_trick2}, and $N^{-(p-1)}$ from \eqref{eq:final_term_trick1}. In most of the cases, \eqref{eq:final_term_trick2_appl} implies~\eqref{eq:cum_exp_error_term_case2_bound}. Specifically, it remains to consider the case when $\bm{\alpha}_1=\emptyset$ (i.e. there is no gain from $\mathfrak{n}(\bm{\alpha}_1)$ in \eqref{eq:final_term_trick2_appl}), $p=2$ and $n=1$. In this case \eqref{eq:cum_exp_error_term_case2} simplifies to
\begin{equation}
\sum_{c=1}^N\frac{\kappa(ac,\bm{\alpha})}{(\ell-1)!N^{\ell/2}}\E\bigg[N^{1/2}\left(G^{k_{01}}\right)_{cb}\mathrm{Iso}\big(\bm{k}_0^{[1]},\bm{\phi}^{[1]}\big) \partial_{\bm{\alpha}}\Delta(k)\bigg],\qquad \bm{\alpha}\in \{ac,ca\}^{\ell-1}.
\label{eq:before_iso_resum}
\end{equation}
In particular, $\Delta(k)$ is a single trace but not a product of several ones. The target bound~\eqref{eq:cum_exp_error_term_case2_bound} is equivalent to the upper bound of order $N^{-1+\xi}$ on the absolute value of~\eqref{eq:before_iso_resum}.

We consider only the terms with $c\neq a$, since the complementary case is covered above. For $\ell\ge 4$ we estimate the expectation in \eqref{eq:before_iso_resum} by $N^\xi$ by the means of \eqref{eq:kG_av}--\eqref{eq:kG_iso} and obtain the desired bound. The case $\ell=3$ is more subtle. Since $\mathfrak{n}(\bm{\alpha})=\ell-1=2$, $\partial_{\bm{\alpha}}\Delta(k)$ is a sum of products of two resolvent entries which are simultaneously either off-diagonal or diagonal. If they are off-diagonal, we gain a factor $N^{-1/2}$ from each of them due to \eqref{eq:kG_iso}, which is sufficient for our purpose. When both of these entries are diagonal, \eqref{eq:before_iso_resum} becomes of the form
\begin{equation}
\sum_{c\in [N]\setminus\{a\}}\frac{\kappa_{ac}}{2N^{3/2}}\E\bigg[N^{1/2}\left(G^{k_{01}}\right)_{cb}\mathrm{Iso}\big(\bm{k}_0^{[1]},\bm{\phi}^{[1]}\big) (G^u)_{cc}(G^{k+2-u})_{aa}\bigg],
\label{eq:before_iso_resum1}
\end{equation}
where $u\in [k+1]$ and $\kappa_{ac}$ equals either to $\kappa(ac,ac,ac)$ or to $\kappa(ac,ca,ca)$. Decompose $(G^u)_{cc}$ into the deterministic approximation and the fluctuation around it: $(G^u)_{cc}= m_u + (G^u-m_u)_{cc}$. In the terms with fluctuation we gain $N^{-1/2}$ from \eqref{eq:kG_iso} and obtain the upper bound of order $N^{-1+\xi}$ on the absolute value of~\eqref{eq:before_iso_resum1}. Finally, in the terms with $m_u$ we perform summation over $c$:
\begin{equation}
\left\vert \sum_{c\in [N]\setminus\{a\}} \kappa_{ac} \left(G^{k_{01}}\right)_{cb}\right\vert =\left\vert \left(G^{k_{01}}\right)_{\bm{v}b}\right\vert \prec N^{-1/2}\|\bm{v}\|\lesssim 1.
\label{eq:iso_resum}
\end{equation}
Here $\bm{v}_c =\delta_{ac}\kappa_{ac}$ for $c\in [N]$, and in the last but one bound we used~\eqref{eq:kG_iso}. This finishes the proof of~\eqref{eq:cum_exp_error_term_case2_bound}.

\subsubsection{Analysis of the second and higher order cumulants in the second line of \eqref{eq:underline_exp_general}}\label{sec:cum_exp_terms2}
Similarly to Section~\ref{sec:cum_exp_terms} we consider the terms
\begin{equation}
N^{-\ell/2}\sum_{c=1}^N\Bigg|\kappa\Bigg(\bigg(\partial_{\bm{\alpha}_1}\Big[N^{1/2}\left(G^{k_{01}}\right)_{cb}\mathrm{Iso}\big(\bm{k}_0^{[1]},\bm{\phi}^{[1]}\big)\Delta\big(\bm{k}_1\big)\Big] \prod_{\bm{\alpha}_j\neq \emptyset} \partial_{\bm{\alpha}_j}\Delta\big(\bm{k}_j\big)\bigg),\left\lbrace \Delta\big(\bm{k}_j\big)\right\rbrace_{\bm{\alpha}_j=\emptyset}\Bigg)\Bigg|,
\label{eq:cum_intermediate_case2}
\end{equation}
where $\bm{\alpha}_j=\emptyset$ for at least one $j\in [2,p]$.  The terms with $\ell\ge L^{(p)}(S)$ are estimated similarly to the argument around \eqref{eq:high_order_term}, so we further focus on the complementary case $\ell<L^{(p)}(S)$ and abbreviate $L:=L^{(p)}(S)$. \nc Recall the notations $r,s$ and $t$ introduced in Section~\ref{sec:cum_exp_terms}. We note that after the differentiation in \eqref{eq:cum_intermediate_case2} is performed, the number of remaining traces equals to $n-r-s$, but not to $n-r-s-1$ as in Section~\ref{sec:cum_exp_terms}. Consider the following cases.

\smallskip

\noindent\underline{$c\neq a,b$; $r\ge 1$.} Similarly to the argument around \eqref{eq:cum_intermediate1}, we estimate the difference between the number of off-diagonal entries after differentiation and originally in $\mathrm{Iso}(\bm{k}_0,\bm{\phi})$ from below by $r+s-t$, and instead of \eqref{eq:cum_intermediate1} obtain 
\begin{equation}
\sum_{r=1}^{p-2}\frac{1}{N^{r-1/2}} \Psi^{(p-r)}(n-r,S+2L).
\label{eq:cum_intermediate_case2_1}
\end{equation}

\noindent\underline{$c\neq a,b$; $r=0$; $s\ge 1$.} The argument presented in the previous case gives the bound of order \eqref{eq:cum_intermediate2} times the small factor $N^{-1/2}$.

\smallskip

\noindent\underline{$c$ does not appear in $\bm{\phi}$; $r=s=0$.} Observe that the number of off-diagonal entries does not decrease. Indeed, for any $(e,f)\in\bm{\phi}$, $e\neq f$, $k\in\N$ and $\bm{\alpha}\in\{ac,ca\}^l$, $l\in\N$, we have that $\partial_{\bm{\alpha}} (G^k)_{ef}$ contains at least one off-diagonal entry, since $e$ and $f$ cannot simultaneously equal to $a$. Estimating additionally $\ell\ge 3$, we obtain the bound
\begin{equation}
N^{-1/2}\Psi^{(p)}(n,S+2L).
\label{eq:cum_intermediate_case2_2}
\end{equation}

\smallskip

\noindent\underline{$c$ appears in $\bm{\phi}$; $c\neq a,b$; $r=s=0$.} The number of off-diagonal entries decreases at most by $t=\ell-1$. Since the number of $c$'s under consideration does not depend on $N$, we again get the bound of order  \eqref{eq:cum_intermediate_case2_2}.

\smallskip

\noindent\underline{$c=a$; $r\ge 0$.} Since all derivatives are in the diagonal entry $w_{aa}$, the number of off-diagonal entries does not decrease. Estimating $\ell\ge r+1$, we get
\begin{equation}
\sum_{r=0}^{p-2} \frac{1}{N^{(r+1)/2}} \Psi^{(p-r)}(n-r,S+2L).
\label{eq:cum_intermediate_case2_3}
\end{equation}

\smallskip

\noindent\underline{$c=b$; $r\ge 1$.} The number of off-diagonal entries increases at least by $r+s-t-1$, where $1$ is subtracted because of $(G^{k_{10}})_{cb}=(G^{k_{10}})_{bb}$. Together with \eqref{eq:ell_rst} this gives
\begin{equation}
\sum_{r=1}^{p-1} \frac{1}{N^r} \Psi^{(p-r)}(n-r,S+2L).
\label{eq:cum_intermediate_case2_4}
\end{equation}

\smallskip

\noindent\underline{$c=b$; $r=0$; $s\ge 1$.} The argument presented in the previous case gives the bound of order \eqref{eq:cum_intermediate2} times the small factor $N^{-1/2}$.

\smallskip

\noindent\underline{$c=b$; $r=s=0$.} We claim that the sum of corresponding terms in \eqref{eq:cum_intermediate_case2} has an upper bound of order
\begin{equation}
N^{-1/2}\Psi^{(p)}(n,S+2L) + \Psi^{(p)}(n,S-1).
\label{eq:cum_intermediate_case2_5}
\end{equation}
Indeed, the number of off-diagonal entries can decrease at most by $\ell$. If it decreased by at most $\ell-1$, then the associated terms are bounded by the first term in \eqref{eq:cum_intermediate_case2_5}. Otherwise, by the second, since the total number of resolvents increases at most by $\ell-1$. 

\smallskip

 Collecting the estimates proved in Section~\ref{sec:off_diag_expansion}, we conclude that the absolute value of \eqref{eq:extended_family_l=0} with ${\mathrm{od}}(\bm{\phi})>0$ is bounded by the rhs. of \eqref{eq:global_master} up to some $N$-independent multiplicative constant.\nc

\subsection{Proof of Proposition~\ref{prop:global_master}: expansion with respect to a diagonal entry}\label{sec:diag_expansion}

Now we consider the remaining case where $\bm{k}_0\neq\emptyset$, ${\rm{od}}(\bm{\phi})=0$ and perform the underline expansion with respect to a diagonal entry. The analysis is almost identical to the one in Section~\ref{sec:off_diag_expansion}, so we only outline the required adjustments. We start with the identity~\eqref{eq:iso_G_expansion}, which holds without any changes for diagonal $\phi$. Next, in \eqref{eq:cum_exp_error_term_case2} index $b$ is equal to $a$ and the factor $N^{1/2}$ by $(G^{k_{01}})_{ca}$ is absent. We claim that the bound \eqref{eq:cum_exp_error_term_case2_bound} is unchanged. Additional attention is required only in the case $c=a$, where we simply estimate the analogue of the expectation in \eqref{eq:cum_exp_error_term_case2} by $N^\xi$ and bound $N^{-\ell/2}\le N^{-p/2}$. For the second and higher order cumulants in the second line of \eqref{eq:underline_exp_general} one needs to take care of $c=a$, but then the argument is identical to the one in the case $c=a$ in Section~\ref{sec:cum_exp_terms2}.

Combining the estimates from Sections~\ref{sec:tracial_expansion}--\ref{sec:diag_expansion}, we finish the proof of Proposition~\ref{prop:global_master}.\hfill\qed

\subsection{Proof of Proposition~\ref{prop:global}}\label{sec:proof_global} We may assume that $p\ge 1$, since in the complementary case one may use that $\kappa_{0,q}(\Delta(1)) = \overline{\kappa_{q,0}(\Delta(1))}$. Along the proof we suppress the argument $S$ in $\Psi^{(p)}(n,S)$, since none of the estimates depends on this parameter. We revisit the static expansion presented in Section~\ref{sec:tracial_expansion} and apply it for $\kappa_{p,q}(\Delta(1))$. In this case, each of the parameters $p,n$ and $S$ in Section~\ref{sec:tracial_expansion} should be replaced by $p+q$. As we will now see, most of the terms in the expansion can be absorbed into the error term in~\eqref{eq:kappapq_new} in the global regime. We start with the terms discussed in Sections~\ref{sec:expansion_1arg}--\ref{sec:hit_two_args}. By \eqref{eq:err_1_1} and \eqref{eq:err_1_2} they have an upper bound of order
\begin{equation}
N^{-1}\Psi^{(p+q)}(p+q+1) + N^{-1}\Psi^{(p+q-1)}(p+q-1) \lesssim N^{-1/2+\xi} N^{-(p+q-2)/2},
\end{equation}
where in the last estimate we used Proposition~\ref{prop:global_Psi}. Additionally, we removed the first term from~\eqref{eq:err_1_1}, since it originated from the terms in the first line of~\eqref{eq:Delta_k}. The second and the third of these terms are not present for $\kappa_{p,q}(\Delta(1))$, while the first one does not contribute in general to the expansion. Similarly, we removed the second term in~\eqref{eq:err_1_2}. 

Next, the terms in the expansion discussed in Section~\ref{sec:cum_exp_terms} are bounded by
\begin{equation}
\sum_{r=1}^{p+q-2}N^{-(r-1)/2} \Psi^{(p+q-r)}(p+q-r-1) + \Psi^{(p+q)}(p+q-1) \lesssim N^{-1/2+\xi} N^{-(p+q-2)/2},
\label{eq:global_pq_error_term}
\end{equation}
as it follows from \eqref{eq:cum_intermediate1}, \eqref{eq:cum_intermediate2} and \eqref{eq:cum_intermediate3}. Here we used Proposition~\ref{prop:global_Psi} in the last estimate and observed that the characteristic function in the rhs. of~\eqref{eq:Psi_improved} vanishes for each of the terms in~\eqref{eq:global_pq_error_term}. We are left with the terms~\eqref{eq:cum_exp_error_term} discussed in Section~\ref{sec:cum_exp_last_term}. By~\eqref{eq:cum_exp_error_term_bound_d}--\eqref{eq:cum_exp_error_term_bound_od} all of them can be absorbed into the error term in the rhs. of~\eqref{eq:kappapq_new} apart from the ones with $a=b$ and $\ell=p+q$. Thus, we get
\begin{equation}
\kappa_{p,q}(\Delta(1)) = \frac{1}{N^{(p+q-2)/2}}\left(-\frac{m}{1-m^2}\sum_{a=1}^N\frac{\kappa_{p+q}^{\dif}}{N}\E\left[G_{aa} \left(\partial_{aa}\Delta(1)\right)^{p-1}\left(\partial_{aa}\overline{\Delta(1)}\right)^{q}\right] + \mathcal{O}\left(N^{-1/2+\xi}\right)\!\right)\!.
\label{eq:global_leading_term_prefin}
\end{equation}
The factor $(\ell-1)!$ in \eqref{eq:cum_exp_error_term} canceled out with the number of splittings of $\bm{\alpha}=\{aa\}^{\ell-1}$ into $\ell-1$ labeled blocks of unit size. Using~\eqref{eq:kG_iso}, we compute
\begin{equation}
\partial_{aa}\Delta(1) = - (G^2)_{aa} = -m_2 -(G^2-m_2)_{aa} = -\frac{m^2}{1-m^2} + \mathcal{O}\left(N^{-1/2+\xi}\right),
\label{eq:Delta_dif_aa}
\end{equation}
with the bound on the error term valid with very high probability. Combining \eqref{eq:global_leading_term_prefin} with \eqref{eq:Delta_dif_aa} and performing the same argument as in \eqref{eq:Delta_dif_aa} for $G_{aa}$, we complete the proof of Proposition~\ref{prop:global}.\hfill$\qed$

\section{Static chaos expansion: cancellation for GUE}\label{sec:Gauss}

In this section we consider $W$ with Gaussian entries and prove Proposition~\ref{prop:E_Gauss} via static method. First, in Section~\ref{sec:E_expansion} we prove the existence of expansion of $\E(\bm{k},\bm{l})$ in the negative powers of $N\eta$ both for GUE and GOE, and derive recursive relations for the coefficients of this expansion. Next, in Section~\ref{sec:E_cancel} we focus on the GUE case and identify the major cancellation in the expansion by showing that all orders starting from the first one vanish.

\subsection{Full $(N\eta)$-expansion of expectation}\label{sec:E_expansion} For $z=E+\ii\eta$ with $\eta>0$ denote for short
\begin{equation}
\mathfrak{a}:=m/(1-m^2)\quad \text{and}\quad \mathfrak{a}_0:=m_0/(1-m^2_0),\,\,\, m_0:=m(E+\ii 0).
\label{eq:def_a} 
\end{equation}
We show that $\E(\bm{k},\bm{l})$ admits a full expansion in negative powers of $N\eta$, up to the error term of order $\eta$.
\begin{lemma}\label{lem:E_full_expansion} Let $W$ be an $N\times N$ GOE/GUE matrix. For any fixed $D\in\N$, $n_1,n_2\in\N\cup\{0\}$ and multi-indices $\bm{k}\in \N^{n_1},\bm{l}\in\N^{n_2}$, it holds that
\begin{equation}
\E\left(\bm{k},\bm{l}\right) = \sum_{d=0}^{D-1} \frac{1}{(N\eta)^d}\Phi_d\left(\bm{k},\bm{l}\right) + \mathcal{O}\left(\frac{1}{(N\eta)^D}+\eta\right),
\label{eq:E_full_expansion}
\end{equation}
where the implicit constant in $\mathcal{O}(\cdot)$ is uniform in $z\in\mathcal{D}$. For $d\in \N\cup\{0\}$, the coefficients $\Phi_d$ depend only on $\Re z$, $\bm{k}$, $\bm{l}$ and $\beta$, and satisfy the following recurrence relations for $\bm{k}\neq\emptyset$:
\begin{equation}
\begin{split}
\Phi_d(\bm{k},\bm{l})=&\frac{2\ma_0}{\beta}\Im m_0 \sum_{j=1}^{n_2} \frac{\ii^{k_1-l_j-1}}{2^{k_1+l_j-1}} l_j\binom{k_1+l_j-1}{k_1-1} \Phi_d\left(\bm{k}^{[1]},\bm{l}^{[j]}\right)+\bm{1}_{\beta=1}\ma_0 k_1\Phi_{d-1}\left(\big(k_1+1,\bm{k}^{[1]}\big),\bm{l}\right)\\
+& \ma_0 \sum_{u=1}^{k_1} \Phi_{d-1}\left(\big(u,k_1+1-u,\bm{k}^{[1]}\big),\bm{l}\right)+\frac{2\ma_0}{\beta} \sum_{i=2}^{n_1} k_i \Phi_{d-1}\left(\big(k_1+k_i+1,\bm{k}^{[1,i]}\big),\bm{l}\right)\\
+&\frac{2\ma_0}{\beta} \sum_{j=1}^{n_2} \sum_{u=1}^{k_1} l_j \frac{(-1)^{k_1-u}}{(2\ii)^{k_1+l_j+1-u}} \binom{k_1+l_j-u}{l_j}\Phi_{d-1}\left( \big(u, \bm{k}^{[1]}\big), \bm{l}^{[j]}\right)\\
+&\frac{2\ma_0}{\beta} \sum_{j=1}^{n_2} \sum_{v=1}^{l_j+1} l_j\frac{(-1)^{k_1}}{(2\ii)^{k_1+l_j+1-v}}\binom{k_1+l_j-v}{k_1-1}\Phi_{d-1}\left(\bm{k}^{[1]},\big(v,\bm{l}^{[j]}\big)\right).
\end{split}
\label{eq:Phi_recursion}
\end{equation}
Here we additionally set $\Phi_{d'}:=0$ for $d'<0$ and $\Phi_0(\emptyset,\emptyset):=1$.
\end{lemma}

Now we prove Lemma~\ref{lem:E_full_expansion}. 
First, similarly to the static expansion presented in Section~\ref{sec:expansion_1arg} we get
\begin{equation}
\begin{split}
&\ma^{-1}\E(\bm{k},\bm{l}) = \eta \E\left((k_1-1,\bm{k}^{[1]}), \bm{l}\right) + 2\sum_{u=1}^{k_1-1} m_{k_1+1-u} \eta^{k_1-u} \E\left((u,\bm{k}^{[1]}),\bm{l}\right)\\
&\quad +\frac{1}{N\eta}\sum_{u=1}^{k_1}\E\left((u,k_1+1-u,\bm{k}^{[1]}),\bm{l}\right)+\bm{1}_{\beta=1} k_1\left( m_{k_1+1}\eta^{k_1}\E\left(\bm{k}^{[1]},\bm{l}\right) + \frac{1}{N\eta}\E\left((k_1+1,\bm{k}^{[1]}),\bm{l}\right)\right)\\
&\quad  +\sum_{i=2}^{n_1}\frac{2k_i}{\beta} \left(m_{k_1+k_i+1}\eta^{k_1+k_i}\E\left(\bm{k}^{[1,i]},\bm{l}\right)+\frac{1}{N\eta}\E\left((k_1+k_i+1,\bm{k}^{[1,i]}),\bm{l}\right)\right)\\
&\quad  +\frac{2}{\beta}\sum_{j=1}^{n_2}\sum_{u=1}^{k_1}\frac{(-1)^{k_1-u}}{(2\ii)^{k_1+l_j+1-u}} l_j\binom{k_1+l_j-u}{l_j}\left(\frac{1}{N\eta}\E\left((u,\bm{k}^{[1]}),\bm{l}^{[j]}\right) + \eta^{u-1}m_u \E \left(\bm{k}^{[1]},\bm{l}^{[j]}\right)\right)\\
&\quad  +\frac{2}{\beta}\sum_{j=1}^{n_2}\sum_{v=1}^{l_j+1}\frac{(-1)^{k_1}}{(2\ii)^{k_1+l_j+1-v}}l_j\binom{k_1+l_j-v}{k_1-1}\left(\frac{1}{N\eta} \E\left(\bm{k}^{[1]},(v,\bm{l}^{[j]})\right) + \eta^{v-1}\overline{m_v}\E\left(\bm{k}^{[1]},\bm{l}^{[j]}\right)\right).
\end{split}
\label{eq:E_expansion}
\end{equation}
The terms in the first two lines of~\eqref{eq:E_expansion} come from the expansion of $\Delta(k_1)$, see~\eqref{eq:Delta_k}. The third line corresponds to the expansion of underline to $\Delta(k_i)$ via~\eqref{eq:under_extension_basic}, and the third line to the similar expansion to $\Delta(l_j)$. The expectation removes the full underline term since 
we are in the Gaussian case. The sums over $u$ and $v$ in the last two lines of~\eqref{eq:E_expansion} arise from the following generalized resolvent identity.

\begin{lemma}[Generalized resolvent identity]\label{lem:resolvent_id} For $H=H^*\in\C^{N\times N}$ and $z_j\in\C\setminus\R$ denote $G_j:=(H-z_j)^{-1}$, $j=1,2$. For any $k,l\in\N$ it holds that
\begin{equation}
G_1^kG_2^l = \sum_{u=1}^k (-1)^{k-u} \binom{l+k-u-1}{l-1} \frac{G_1^u}{(z_1-z_2)^{l+k-u}} + \sum_{v=1}^l (-1)^{l-v} \binom{l+k-v-1}{k-1} \frac{G_2^v}{(z_2-z_1)^{l+k-v}}.
\label{eq:resolvent_id}
\end{equation}
\end{lemma}

\begin{proof}[Proof of Lemma~\ref{lem:resolvent_id}] We write the lhs. of \eqref{eq:resolvent_id} as follows:
\begin{equation}
\begin{split}
G_1^kG_2^l =& \frac{1}{(k-1)!(l-1)!} \partial_{z_1}^{k-1}\partial_{z_2}^{l-1} G_1G_2 = \frac{1}{(k-1)!(l-1)!} \partial_{z_1}^{k-1}\partial_{z_2}^{l-1} \frac{G_1-G_2}{z_1-z_2}\\
=& \frac{1}{(k-1)!} \partial_{z_1}^{k-1} \frac{G_1}{(z_1-z_2)^l} + \frac{1}{(l-1)!} \partial_{z_2}^{l-1} \frac{G_2}{(z_2-z_1)^k}.
\end{split}
\label{eq:resolvent_id_proof}
\end{equation}
Performing the differentiation in $z_1,z_2$ in the rhs. of \eqref{eq:resolvent_id_proof}, we complete the proof of \eqref{eq:resolvent_id}.
\end{proof}

Having~\eqref{eq:E_expansion} in hand, we first show that
\begin{equation}
\left\vert \E(\bm{k},\bm{l})\right\vert\lesssim 1.
\label{eq:E_1bound}
\end{equation}
for any fixed multi-indices $\bm{k}$ and $\bm{l}$. Note that~\eqref{eq:Delta_a_priori} implies a slightly weaker bound of order $N^\xi$ for any fixed $\xi>0$ on the rhs. of~\eqref{eq:E_1bound}, and the meaning of~\eqref{eq:E_1bound} is that this small power of $N$ can be removed. First we prove \eqref{eq:E_1bound} for $\bm{l}=\emptyset$ by induction on $|\bm{k}|$. In this special case only the first three lines of \eqref{eq:E_expansion} remain. We immediately neglect the terms containing the factor $(N\eta)^{-1}$, since they are smaller than $N^{-\epsilon +\xi}$ by~\eqref{eq:Delta_a_priori}, and observe that only the terms with strictly smaller $|\bm{k}|$ are left, which makes the induction work. Next, by the symmetry between $G$ and $G^*$ we also obtain~\eqref{eq:E_1bound} for $\bm{k}=\emptyset$ and general $\bm{l}$. Using \eqref{eq:E_expansion} and the case when at least one of the multi-indices $\bm{k}, \bm{l}$ is empty as a base, we perform induction on $|\bm{k}|+|\bm{l}|$ and conclude \eqref{eq:E_1bound} in full generality.

Now we prove~\eqref{eq:E_full_expansion}. By \eqref{eq:E_1bound} one can absorb all terms containing at least one factor $\eta$ into the error term in the rhs. of \eqref{eq:E_full_expansion}. In particular, in the last but one line of \eqref{eq:E_expansion} the last term survives only for $u=1$, and similarly for the last line of \eqref{eq:E_full_expansion}. We multiply \eqref{eq:E_expansion} by $\ma$ and replace $\ma$ by $\ma_0$, since $\ma=\ma_0+\mathcal{O}(\eta)$. Additionally, we replace $m$ by $m_0$ for the same reason. The proof of~\eqref{eq:E_full_expansion} proceeds by induction on $D\in \N\cup\{0\}$ starting from the base case $D=0$, which follows from~\eqref{eq:E_1bound}. To make a step from $d$ to $d+1$ we induct on $|\bm{k}|+|\bm{l}|$. The remaining details are elementary and thus are omitted. Finally, the recursive relation~\eqref{eq:Phi_recursion} immediately follows from \eqref{eq:E_full_expansion}. This finishes the proof of Lemma~\ref{lem:E_full_expansion}.\hfill$\qed$

\subsection{Cancellation for $\beta=2$}\label{sec:E_cancel} In this section we prove a major cancellation among the explicit terms in the expansion \eqref{eq:E_full_expansion} for GUE and complete the proof of Proposition~\ref{prop:E_Gauss}.

\begin{lemma}\label{lem:Phi_expl} Let $W$ be an $N\times N$ GUE matrix. For any $n_1,n_2\in \N\cup\{0\}$ and $\bm{k}\in \N^{n_1}$, $\bm{l}\in \N^{n_2}$ it holds that
\begin{equation}
\Phi_0(\bm{k},\bm{l})=\delta_{n_1,n_2}\frac{\ii^{|\bm{k}|-|\bm{l}|}}{2^{|\bm{k}|+|\bm{l}|}} \prod_{j=1}^n \frac{1}{(k_j-1)!(l_j-1)!}\sum_{\sigma\in S_{n}} \prod_{j=1}^n (k_j+l_{\sigma(j)}-1)!. 
\label{eq:Phi_0_expl}
\end{equation}
Here in the case $n_1=n_2$ we denoted $n:=n_1$, and $S_n$ is the set of permutations of $[n]$. For any $d\ge 1$ we further have
\begin{equation}
\Phi_d(\bm{k},\bm{l})=0.
\label{eq:Phi_vanish}
\end{equation}
\end{lemma}

Lemma~\ref{lem:E_full_expansion} together with Lemma~\ref{lem:Phi_expl} immediately imply Proposition~\ref{prop:E_Gauss}. 

\begin{proof}[Proof of Lemma~\ref{lem:Phi_expl}] Since the case $\bm{k}=\emptyset$ is not explicitly covered by \eqref{eq:Phi_recursion}, we first eliminate it by showing that
\begin{equation}
\Phi_d(\emptyset, \bm{l}) = \Phi_d(\bm{k},\emptyset)=0,
\label{eq:Phi_vanish_trivial}
\end{equation}
for any $d\in \N\cup\{0\}$ and $\bm{k},\bm{l}\neq \emptyset$. Due to the symmetry between $G$ and $G^*$, it suffices to prove the second identity in~\eqref{eq:Phi_vanish_trivial}. Applying \eqref{eq:Phi_recursion} to $\Phi_d(\bm{k},\emptyset)$, we obtain the sum of $\Phi_{d-1}$, i.e. $\Phi_d$ disappears. Using that $\Phi_d=0$ for $d<0$ and any multi-indices, and arguing by induction on $d$, we complete the proof of \eqref{eq:Phi_vanish_trivial}. 

In the remainder of the proof we assume that $\bm{k},\bm{l}\neq \emptyset$. Now we prove \eqref{eq:Phi_0_expl}. For $d=0$ the recursive relation \eqref{eq:Phi_recursion} greatly simplifies: only the sum over $j$ in the first line of \eqref{eq:Phi_recursion} remains. Observe additionally that
\begin{equation}
\ma_0 \Im m_0 = \frac{m_0\Im m_0}{1-m_0^2}=\frac{\ii}{2},
\label{eq:a_id}
\end{equation}
as it follows from~\eqref{eq:def_m} by a straightforward calculation. This identity is the core reason for the main cancellation. Thus, we have
\begin{equation}
\Phi_0(\bm{k},\bm{l})=\sum_{j=1}^{n_2} \frac{\ii^{k_1-l_j}}{2^{k_1+l_j}} \frac{(k_1+l_j-1)!}{(k_1-1)!(l_j-1)!}  \Phi_0\left(\bm{k}^{[1]},\bm{l}^{[j]}\right).
\label{eq:Phi0_recursion}
\end{equation}
Together with the initial condition  $\Phi_0(\emptyset,\emptyset)=1$ this immediately implies that $\Phi_0(\bm{k},\bm{l})=0$ for $n_1\neq n_2$. In the case when $n_1=n_2$ we apply \eqref{eq:Phi0_recursion} to $\Phi_0\left(\bm{k}^{[1]},\bm{l}^{[j]}\right)$ and proceed iteratively. After $n=n_1$ expansions (including the first one) we arrive to the sum of $n!$ terms naturaly labeled by $S_n$, and obtain \eqref{eq:Phi_0_expl}.

By \eqref{eq:Phi_recursion}, it suffices to verify \eqref{eq:Phi_vanish} for $d=1$. The proof of \eqref{eq:Phi_vanish} for larger $d$ proceeds then by induction on $d$ and uses that all terms in \eqref{eq:Phi_recursion} apart from the ones in the first line vanish by inductive hypothesis. We further focus on the case $d=1$. It is sufficient to show that the sum of all terms in the second  to the fourth lines of \eqref{eq:Phi_recursion} vanishes. For $n_1\neq n_2\pm 1$ this holds since each of the terms vanishes by \eqref{eq:Phi_0_expl}. However, for $n_2=n_2\pm 1$ we need to identify the cancellation between different lines. We further denote $n:=n_2$ and consider separately the cases $n_1=n-1$ and $n_1=n+1$.

\smallskip

\noindent\underline{$n_1=n-1$.} By \eqref{eq:Phi_0_expl}, all terms in the fourth line of \eqref{eq:Phi_recursion} vanish, as well as all terms in the second sum in the second line. Thus, we need to show that
\begin{equation}
\sum_{u=1}^{k_1} \left[\Phi_0\left(\big(u,k_1+1-u,\bm{k}^{[1]}\big),\bm{l}\right) + \sum_{j=1}^n \frac{(-1)^{k_1-u}}{(2\ii)^{k_1+l_j+1-u}} l_i\binom{k_1+l_j-u}{l_i}\Phi_{0}\left( \big(u, \bm{k}^{[1]}\big), \bm{l}^{[j]}\right)\right]=0.
\label{eq:cancel1}
\end{equation}
We will show that the big bracket in the lhs. of \eqref{eq:cancel1} vanishes for every $u\in [k_1]$. First, by \eqref{eq:Phi_0_expl} we have that the first term in \eqref{eq:cancel1} equals to
\begin{equation}
\frac{\ii^{|\bm{k}|-|\bm{l}|+1}}{2^{|\bm{k}|+|\bm{l}|+1}} \frac{1}{(u-1)!(k_1-u)!}\prod_{a=2}^{n-1}\frac{1}{(k_a-1)!}\prod_{b=1}^n \frac{1}{(l_b-1)!}\sum_{j=1}^n (k_1-u+l_j)!\sum_{\sigma\in \mathcal{M}_j} \prod_{c\in [2,n]} (k_c+l_{\sigma(c)}-1)!.
\label{eq:cancel1_1}
\end{equation}
Here we separated the element of $\bm{l}$ paired with $k_1+1-u$ in \eqref{eq:Phi_0_expl}, denoted $k_n:=u$ and employed $\mathcal{M}_j$ to denote the set of all bijective maps from  $[2,n]$ to $[n]\setminus\{j\}$. We claim that for every $j\in [n]$ the sum over $\sigma$ in \eqref{eq:cancel1_1} exactly cancels out with the term corresponding to the same index $j$ in the second sum in \eqref{eq:cancel1}. To this end, we apply~\eqref{eq:Phi_0_expl} to the latter term. The product of the coefficient of $\Phi_{0}\left( \big(u, \bm{k}^{[1]}\big), \bm{l}^{[j]}\right)$ in \eqref{eq:cancel1} and the one coming from\footnote{Here and further in this proof we say that $\ii^{|\bm{k}|-|\bm{l}|}/2^{|\bm{k}|+|\bm{l}|}$ is the coefficient in \eqref{eq:Phi_0_expl}.} \eqref{eq:Phi_0_expl} applied to this $\Phi_0$ equals to
\begin{equation}
\frac{(-1)^{k_1-u}}{(2\ii)^{l_i+k_1+1-u}} \frac{\ii^{|\bm{k}|-k_1+u-|\bm{l}|+l_i}}{2^{|\bm{k}|-k_1+u+|\bm{l}|-l_i}} = \frac{\ii^{|\bm{k}|-k_1+u-|\bm{l}|+l_i +2k_1-2u -l_i-k_1-1+u}}{2^{|\bm{k}|-k_1+u+|\bm{l}|-l_i +l_i+k_1+1-u}}=\frac{\ii^{|\bm{k}|-|\bm{l}|-1}}{2^{|\bm{k}|+|\bm{l}|+1}} = -\frac{\ii^{|\bm{k}|-|\bm{l}|+1}}{2^{|\bm{k}|+|\bm{l}|+1}},
\label{eq:2_power}
\end{equation}
which coincides with the corresponding coefficient in \eqref{eq:cancel1_1} up to the minus sign which gives the cancellation. The sum over permutations arising in \eqref{eq:Phi_0_expl} for $\Phi_{0}\left( \big(u, \bm{k}^{[1]}\big), \bm{l}^{[j]}\right)$ is naturally indexed by $\mathcal{M}_j$ and coincides with the sum in \eqref{eq:cancel1_1}. The factorials also exactly match after relabelling. This finishes the proof of~\eqref{eq:cancel1}.

\smallskip

\noindent\underline{$n_1=n+1$.} In this case all terms in the third line of \eqref{eq:Phi_recursion} vanish due to \eqref{eq:Phi_0_expl}, as well as the first group of terms in the second line. Thus, we need to show that
\begin{equation}
\sum_{i=2}^{n+1} k_i \Phi_0\left(\big(k_1+k_i+1,\bm{k}^{[1,i]}\big),\bm{l}\right) + \sum_{j=1}^n \sum_{v=1}^{l_j+1} l_j\frac{(-1)^{k_1}}{(2\ii)^{k_1+l_j+1-v}}\binom{l_j+k_1-v}{k_1-1}\Phi_0\left(\bm{k}^{[1]},\big(v,\bm{l}^{[j]}\big)\right) = 0.
\label{eq:cancel2}
\end{equation}
We apply \eqref{eq:Phi_0_expl} to all $\Phi_0$'s in \eqref{eq:cancel2}. A calculation similar to \eqref{eq:2_power} shows that the powers of $\ii$ and $2$ in the two sums in \eqref{eq:cancel2} are given by
\begin{equation}
\frac{\ii^{|\bm{k}|-|\bm{l}|+1}}{2^{|\bm{k}|+|\bm{l}|+1}}\quad\text{and}\quad -\frac{\ii^{|\bm{k}|-|\bm{l}|+1}}{2^{|\bm{k}|+|\bm{l}|+1}},
\end{equation}
respectively. Fix indices $i\in [n+1]\setminus\{1\}$ and $j\in [n]$. Consider the term in the first $\Phi_0$ in \eqref{eq:cancel2} where $k_1+k_i+1$ is paired with $l_j$ via \eqref{eq:Phi_0_expl}, and the term in the second $\Phi_0$ in \eqref{eq:cancel2} where $k_i$ is paired with $v$. Fix further any bijection $\sigma: [n+1]\setminus\{1,i\} \to [n]\setminus\{j\}$. In this
way, we keep only summation over $v$ in \eqref{eq:cancel2}, the rest of the parameters are fixed. It remains to show that
\begin{equation}
k_i\frac{(k_1+k_i+l_j)!}{(k_1+k_i)!(l_j-1)!}=\sum_{v=1}^{l_j+1}l_j \frac{(l_j+k_1-v)!(k_i+v-1)!}{(k_1-1)!(l_j-v+1)!(k_i-1)!(v-1)!}
\end{equation}
Denoting $u:=v-1$, 
this identity can be equivalently written as
\begin{equation}
\sum_{u=0}^{l_j} \binom{l_j+k_1-1-u}{k_1-1}\binom{k_i+u}{k_i} = \binom{k_1+k_i+l_j}{l_j}.
\label{eq:binom_id}
\end{equation}
which is a standard fact for binomial coefficients. This finishes the proof of~\eqref{eq:cancel2} and of Lemma~\ref{lem:Phi_expl}.
\end{proof}

\section{Dynamic chaos expansion: Propagation from global to the local regime}\label{sec:dynamic_expansion}

All implicit constants in the bounds in this section are uniform in $z\in\mathcal{D}$, $t\in [0,T]$ and $N$ even if it is not stated explicitly. In contrast, these constants typically depend on $p,q$ in \eqref{eq:kappapq_local}, \eqref{eq:main_p0_local}, \eqref{eq:main_pq_local}, and on multi-indices $\bm{k},\bm{l}$ in \eqref{eq:E_local} through $|\bm{k}|+|\bm{l}|$ and $\mathfrak{n}(\bm{k})+\mathfrak{n}(\bm{l})$.  

\subsection{Proof of Proposition~\ref{prop:E_local}(i): flow for expectations}\label{sec:E_local}

Introduce the family of time-dependent control parameters
\begin{equation}
\Upsilon_t(n,S):=\max\left\lbrace \left\vert\left(\E^{(W)}-\E^{(G)}\right)\left[\Delta_t(\bm{k}) \overline{\Delta_t(\bm{l})}\right]\right\vert\,:\, \mathfrak{n}(\bm{k})+\mathfrak{n}(\bm{l})\le n, \, |\bm{k}|+|\bm{l}|\le S \right\rbrace,
\label{eq:def_Upsilon}
\end{equation}
for $n,S\in\N \cup \{ 0\} $ and $t\in [0,T]$. Since $\Delta(\emptyset)=1$ by the definition \eqref{eq:def_Delta}, it holds that
\begin{equation}
\Upsilon_t(0,S)=\Upsilon_t(n,0)=0,\qquad \forall n,S\in\N\cup\{0\},\, t\in [0,T].
\label{eq:Upsilon_init}
\end{equation}
We further extend \eqref{eq:def_Upsilon} to $n,S\in\mathbf{Z}$ by setting $\Upsilon_t(n,S):=0$ for $\min\{n,S\}<0$. \nc
Note that the expectation in~\eqref{eq:def_Upsilon} is not normalized by $\eta_t^{|\bm{k}|+|\bm{l}|}$ to simplify the time-dependence of $\Upsilon_t$. Thus, a priori $\Upsilon_t$ is not an order one quantity, but instead \eqref{eq:Delta_a_priori} implies that
\begin{equation}
\Upsilon_t(n,S)\lesssim N^\xi \eta_t^{-S},
\label{eq:Upsilon_a_priori}
\end{equation}
uniformly in $t\in [0,T]$, for any fixed $n,S\in \N$. In turn, the statement of Proposition~\ref{prop:E_local}(i) is equivalent to
\begin{equation}
\Upsilon_t(n,S) \lesssim \eta_t^{-(S-1)},
\label{eq:Upsilon_goal}
\end{equation}
so essentially we need to gain an extra factor $\eta_t$ from the difference of the expectations. The remainder of this section is dedicated to the proof of~\eqref{eq:Upsilon_goal}.

We gradually improve~\eqref{eq:Upsilon_a_priori} to~\eqref{eq:Upsilon_goal} using the following system of dynamic master inequalities.

\begin{lemma}[Dynamic master inequalities for expectations]\label{lem:dynamic_master} For any fixed $n,S\in\N$ it holds that
\begin{equation}
\begin{split}
\Upsilon_t(n,S)\lesssim \Upsilon_0(n,S) + &\int_0^t \left(\Upsilon_s(n,S-1) + \sum_{S'\le S-2} \frac{1}{\eta^{S-S'+1}_s}\Upsilon_s(n-2,S')\right)\dif s\\
+&\frac{1}{N}\int_0^t \left(\Upsilon_s(n+1,S+2) + \sum_{S'\le S}\frac{1}{\eta^{S-S'+2}_s}\Upsilon_s(n-1,S')\right)\dif s,
\end{split}
\label{eq:dynamic_master}
\end{equation}
uniformly in $t\in [0,T]$. The implicit constant in~\eqref{eq:dynamic_master} does not depend on $N$ and $z=z_T\in\mathcal{D}$, but it may depend on $n$ and $S$.
\end{lemma}

The proof of Lemma~\ref{lem:dynamic_master} is postponed to the end of Section~\ref{sec:E_local}. One more ingredient required for the proof of Proposition~\ref{prop:E_local} is the control on the initial condition $\Upsilon_0$.

\begin{lemma}\label{lem:Upsilon_0} For any fixed $n,S\in\N$ it holds that
\begin{equation}
\Upsilon_0(n,S)\lesssim 1.
\label{eq:Upsilon_0}
\end{equation}
\end{lemma}

This statement shows that the factor $N^\xi$ can be removed from \eqref{eq:Upsilon_a_priori} for $t=0$ (recall from Section~\ref{sec:strategy_Gaus_div} that $\eta_0\sim 1$). The proof of Lemma~\ref{lem:Upsilon_0} is presented in Appendix~\ref{app:Upsilon_0} and proceeds by a minor adjustment of arguments presented in Sections~\ref{sec:global} and~\ref{sec:Gauss}).

We now prove Proposition~\ref{prop:E_local}(i) by iterating~\eqref{eq:dynamic_master}, gaining a $1/N\eta$ factor in each step.  Specifically, we show by induction on $(d,n,S)\in(\N\cup\{0\})^3$ that 
\begin{equation}
\Upsilon_s(n,S) \lesssim \frac{1}{\eta_s^{S-1}} + \frac{1}{(N\eta_s)^d}\frac{N^\xi}{\eta_s^S},\quad s\in [0,T],
\label{eq:Upsilon_iteration}
\end{equation}
where the implicit constant in \eqref{eq:Upsilon_iteration} does not depend on $s$ and $N$, but may depend on $d,n$ and $S$. Recalling from \eqref{eq:def_D} that $N\eta_T\ge N^{-1+\epsilon}$, using that $\eta_s\ge \eta_T$ for all $s\in [0,T]$ and taking $d>\epsilon^{-1}$ in \eqref{eq:Upsilon_iteration}, we immediately obtain~\eqref{eq:Upsilon_goal}.

As the base case for the induction we use $(d,n,S)$ with $dnS=0$. For $d=0$, \eqref{eq:Upsilon_iteration} follows from the a priori bound \eqref{eq:Upsilon_a_priori}, and for $nS=0$ from \eqref{eq:Upsilon_init}. The set-up of the inductive argument is similar to the one in Section~\ref{sec:proof_global_Psi}: it consists of an appropriate finite set $\widehat{\mathscr{S}}\subset(\N\cup\{0\})^3$ and the linear order on it given by $\lhd$ defined around \eqref{eq:def_lexicographic} without the $p$-coordinate. Omitting further details, we focus on the inductive step, and prove \eqref{eq:Upsilon_iteration} for any triple $\bm{s}_0=(d_0,n_0,S_0)\in \widehat{\mathscr{S}}$ with $d_0,n_0,S_0\ge 1$ assuming that \eqref{eq:Upsilon_iteration} holds for all $\bm{s}\lhd \bm{s}_0$. We use \eqref{eq:dynamic_master} for $\Upsilon_t(n_0,S_0)$ and apply the inductive assumption to each $\Upsilon_s$ in the rhs. of \eqref{eq:dynamic_master}, arriving at
\begin{equation}
\Upsilon_t(n_0,S_0)\lesssim \Upsilon_0(n_0,S_0) +\int_0^t\!\!\left( \!\bm{1}_{S_0\ge 2}\!\left(\frac{1}{\eta_s^{S_0}} + \frac{1}{(N\eta_s)^{d_0}}\frac{N^\xi}{\eta_s^{S_0+1}}\right) + \frac{1}{N}\!\!\left(\frac{1}{\eta_s^{S_0+1}} + \frac{1}{(N\eta_s)^{d_0-1}} \frac{N^\xi}{\eta_s^{S_0+2}}\right)\!\!\right)\!\!\dif s.
\label{eq:Upsilon_step}
\end{equation}
Here the cut-off $\bm{1}_{S_0\ge 2}$ arises due to \eqref{eq:Upsilon_init}. Using \eqref{eq:char_flow} together with $\Im m_s\sim 1$, we obtain
\begin{equation}
\int_0^t \frac{\dif s}{\eta_s^k}\lesssim \frac{1}{\eta_t^{k-1}},
\label{eq:int_eta}
\end{equation}
for any real $k>1$. The implicit constant in $\lesssim$ does not depend on $t\in[0,T]$, but may depend on $k$. Combining \eqref{eq:Upsilon_0}, \eqref{eq:Upsilon_step} and \eqref{eq:int_eta}, we complete the proof of \eqref{eq:Upsilon_iteration} for $(d_0,n_0,S_0)$ and thereby of Proposition~\ref{prop:E_local}(i).\hfill $\square$ \nc

\begin{proof}[Proof of Lemma~\ref{lem:dynamic_master}] We start with differentiation of a single trace along the flow \eqref{eq:OU_flow}, \eqref{eq:char_flow}. For any $k \in \N$ It\^{o}'s formula gives
\begin{equation}
\begin{split}
\dif \Delta_t(k) =& k\left(\frac{1}{2}+m_{2,t}\right) \Delta_t(k) \dif t + \mathcal{F}_t(k)\dif t + \dif \mathcal{E}_t(k),\\
\mathcal{F}_t(k):= &\sum_{u=1}^{k-1} km_{k+2-u,t} \Delta_t(u)+ \frac{k}{2N} \sum_{u=1}^{k+1} \Delta_t (u,k+2-u)+\frac{k(k+1)}{2}\bm{1}_{\beta=1}\left(\!m_{k+2,t}+\frac{1}{N}\Delta_t(k+2)\!\right)\!,\\
\dif \mathcal{E}_t(k) = &-\frac{k}{\sqrt{N}} \mathrm{Tr}[G_t^{k+1}\dif B_t].
\end{split}
\label{eq:Tr_dif}
\end{equation}
To differentiate a product of several traces, we again apply the It\^{o}'s formula and get
\begin{equation}
\dif \prod_{i=1}^n X_{i,t} = \sum_{j=1}^n \dif X_{j,t} \prod_{i\in[n]\setminus\{j\}} X_{i,t} + \sum_{1\le j<h\le n}[X_i,X_j]_t \prod_{i\in[n]\setminus\{j,h\}} X_{i,t},
\end{equation}
for any collection of diffusion processes $\{X_{i,t}\}_{i=1}^n$. The covariation process of two resolvent traces is given by
\begin{equation}
[\Delta(k_1),\Delta (k_2)]_t=\frac{k_1k_2}{N}\sum_{a,b,c,d\in [N]}(G^{k_1+1})_{ba}(G^{k_2+1})_{dc}[B_{ab},B_{cd}]_t = \frac{2k_1k_2}{\beta N}\mathrm{Tr}[G_t^{k_1+k_2+2}]\dif t,
\label{eq:cov_process_traces}
\end{equation}
for any $k_1,k_2\in \N$. Here in the last step we used that
\begin{equation}
N^{-1}[B_{ab},B_{cd}]_t = \E [w_{ab}w_{cd}]\dif t,\quad \forall a,b,c,d\in[N],
\end{equation}
by the definition of $B_t$. Similarly to \eqref{eq:cov_process_traces} we have
\begin{equation}
\left[\Delta(k_1),\overline{\Delta (k_2)}\right]_t=\frac{2k_1k_2}{\beta N}\mathrm{Tr}[G_t^{k_1+1}(G_t^*)^{k_2+1}]\dif t.
\label{eq:cov_process_traces_star}
\end{equation}
Finally, combining \eqref{eq:Tr_dif}--\eqref{eq:cov_process_traces_star}, we obtain for any multi-indices $\bm{k}\in\N^{n_1}$, $\bm{l}\in\N^{n_2}$ with $n_1,n_2\in\N\cup\{0\}$ that
\begin{equation}
\begin{split}
\dif \left[\Delta_t(\bm{k})\overline{\Delta_t(\bm{l})}\right] = &f_t(\bm{k},\bm{l}) \Delta_t(\bm{k})\overline{\Delta_t(\bm{l})}\dif t + \mathcal{F}_t(\bm{k},\bm{l})\dif t + \dif \mathcal{E}_t(\bm{k},\bm{l}),\\
f_t(\bm{k},\bm{l}):= &\frac{|\bm{k}|+|\bm{l}|}{2} +|\bm{k}|m_{2,t} + |\bm{l}|\overline{m}_{2,t}.
\end{split}
\label{eq:dif_Deltakl}
\end{equation}
Here the forcing term $\mathcal{F}_t(\bm{k},\bm{l})$ equals to
\begin{equation}
\begin{split}
&\mathcal{F}_t(\bm{k},\bm{l}) = \sum_{i=1}^{n_1}\mathcal{F}_t(k_i)\Delta_t\big(\bm{k}^{[i]}\big) \overline{\Delta_t(\bm{l})} + \sum_{i=1}^{n_2}\overline{\mathcal{F}_t(l_i)}\Delta_t(\bm{k}) \overline{\Delta_t\big(\bm{l}^{[i]}\big)}\\
&\quad+ \frac{2}{\beta N}\!\!\sum_{1\le i<j \le n_1}\!\! k_ik_j \Tr \left[ G_t^{k_i+k_j+2}\right] \Delta_t\big(\bm{k}^{[i,j]}\big)\overline{\Delta_t(\bm{l})} + \frac{2}{\beta N}\!\!\sum_{1\le i<j \le n_2}\!\! l_il_j \overline{\Tr \left[ G_t^{l_i+l_j+2}\right]} \Delta_t(\bm{k})\overline{\Delta_t\big(\bm{l}^{[i,j]}\big)}\\
&\quad+ \frac{2}{\beta N} \sum_{i\in [n_1]}\sum_{j\in [n_2]} k_il_j \Tr\left[ G_t^{k_i+1}(G^*_t)^{l_j+1}\right] \Delta_t\big(\bm{k}^{[i]}\big)\overline{\Delta_t\big(\bm{l}^{[j]}\big)},
\end{split}
\label{eq:F_kl}
\end{equation}
with $\mathcal{F}_t(k)$ defined in~\eqref{eq:Tr_dif}. The exact form of the martingale term $\dif \mathcal{E}_t(\bm{k},\bm{l})$ does not matter since it disappears after taking expectation of~\eqref{eq:dif_Deltakl}.

We take the expectation of \eqref{eq:dif_Deltakl} with respect to $W_t$ with Gaussian and Wigner initial conditions, and subtract these equations from each other:
\begin{equation}
\frac{\dif}{\dif t}\left(\E^{(W)}\!\!-\!\!\E^{(G)}\right)\left[ \Delta_t(\bm{k})\overline{\Delta_t(\bm{l})}\right] = f_t(\bm{k},\bm{l}) \left(\E^{(W)}\!\!-\!\!\E^{(G)}\right)\left[ \Delta_t(\bm{k})\overline{\Delta_t(\bm{l})}\right] + \left(\E^{(W)}\!\!-\!\!\E^{(G)}\right)\mathcal{F}_t(\bm{k},\bm{l}).
\label{eq:E_dif_dif}
\end{equation}
Applying Lemma~\ref{lem:resolvent_id} to the last line of \eqref{eq:F_kl} and decomposing each $\Tr[G^u_t]$, $u\in \N$, in \eqref{eq:F_kl} into $\Delta_t(u)$ and $Nm_{u,t}$, we obtain
\begin{equation}
\begin{split}
\left\vert \left(\E^{(W)}\!\!-\!\!\E^{(G)}\right)\mathcal{F}_t(\bm{k},\bm{l})\right\vert \lesssim &\Upsilon_t(n,S-1) + \sum_{S'\le S-2} \frac{1}{\eta^{S-S'+1}_t}\Upsilon_t(n-2,S')\\
 + &\frac{1}{N}\left(\Upsilon_t(n+1,S+2) + \sum_{S'\le S}\frac{1}{\eta^{S-S'+2}_t}\Upsilon_t(n-1,S')\right).
\end{split}
\label{eq:F_difE_bound}
\end{equation}
Here the first sum over $S'$ arises from the deterministic approximations to the $GG^*$ terms in the last line of \eqref{eq:F_kl}, and the second sum from the $\Delta$-counterparts of these terms.

Solving~\eqref{eq:E_dif_dif}, we get
\begin{equation}
\left(\E^{(W)}\!\!-\!\!\E^{(G)}\right)\left[ \Delta_t(\bm{k})\overline{\Delta_t(\bm{l})}\right] = \mathcal{P}_{0,t}\left(\E^{(W)}\!\!-\!\!\E^{(G)}\right)\left[ \Delta_0(\bm{k})\overline{\Delta_0(\bm{l})}\right]  + \int_0^t \mathcal{P}_{s,t} \left(\E^{(W)}\!\!-\!\!\E^{(G)}\right)\mathcal{F}_s(\bm{k},\bm{l})\dif s,
\label{eq:E_dif_solved}
\end{equation}
where the $(\bm{k},\bm{l})$-dependent propagator $\mathcal{P}_{s,t}$ is given by
\begin{equation}
\mathcal{P}_{s,t}=\mathcal{P}_{s,t}(\bm{k},\bm{l}):= \exp\left\lbrace \int_s^t f_r(\bm{k},\bm{l})\dif r\right\rbrace,\quad \forall\, 0\le s\le t\le T.
\label{eq:E_dif_bound}
\end{equation}
Finally, observe that $|m_{2,r}|\lesssim 1$, since $m_2(z)=\partial_z m(z)$ and $m(z)$ has a bounded derivative in the bulk regime. Therefore, $|\mathcal{P}_{s,t}|\lesssim 1$. Combining \eqref{eq:E_dif_bound} with \eqref{eq:F_difE_bound}, we complete the proof of Lemma~\ref{lem:dynamic_master}.
\end{proof}

\subsection{Proof of Proposition~\ref{prop:kappa_local}: flow for cumulants}\label{sec:p0}

Introduce the two families of time-dependent control parameters to control the cumulants in \eqref{eq:kappapq_new} and \eqref{eq:main_p0}, respectively, along the flow \eqref{eq:OU_flow}, \eqref{eq:char_flow}:
\begin{equation}
\Xi_t^{(p)}(n,S):=\max\left\lbrace \Big|\kappa\Big(\left\lbrace \Delta_t(\bm{k}_j)\overline{\Delta_t(\bm{l}_j)}\right\rbrace_{j=1}^p\Big)\Big|\,:\, \sum_{j=1}^p (\mathfrak{n}(\bm{k}_j)+\mathfrak{n}(\bm{l}_j))\le n,\sum_{j=1}^p (|\bm{k}_j|+|\bm{l}_j|)\le S\right\rbrace,\label{eq:def_Xi}
\end{equation}
\begin{equation}
\Theta_t^{(p)}(n,S):=\max\left\lbrace \Big|\kappa\Big(\left\lbrace \Delta_t(\bm{k}_j)\right\rbrace_{j=1}^p\Big)\Big|\,:\, \sum_{j=1}^p \mathfrak{n}(\bm{k}_j)\le n,\sum_{j=1}^p |\bm{k}_j|\le S\right\rbrace,\label{eq:def_Theta}
\end{equation}
for $t\in [0,T]$ and $p,n,S\in\N$ with $p\ge 2$.  Since $\Delta(\emptyset)=1$ by \eqref{eq:def_Delta} and a cumulant of order $p\ge 2$ with one of the arguments equal to one (or any constant) vanishes, we have
\begin{equation}
\Xi_t^{(p)}(n,0)=\Xi_t^{(p)}(0,S)=0,\quad \forall p\ge 2, \, n,S\in\N\cup\{0\},\, t\in [0,T],
\end{equation}
and similarly for $\Theta_t^{(p)}$. \nc By \eqref{eq:Delta_a_priori}, $\Xi_t^{(p)}$ and $\Theta_t^{(p)}$ satisfy the following a priori bounds:
\begin{equation}
\Xi_t^{(p)}(n,S)\le \Theta_t^{(p)}(n,S)\lesssim N^\xi\eta_t^{-S},
\label{eq:Theta_a_priori}
\end{equation}
for any fixed $\xi>0$, uniformly in $t\in [0,T]$. We also have from Proposition~\ref{prop:global_Psi} the bounds on the initial conditions at $t=0$:
\begin{equation}
\Xi_0^{(p)}(n,S)\le \Theta_0^{(p)}(n,S) \lesssim \frac{N^{(n-p)/2+\xi}}{N^{(p-2)/2}},
\end{equation}
for any fixed $p\ge 2$.

We first establish Proposition~\ref{prop:kappa_local} in a weaker form. Instead of identifying the global mode, we give an upper bound on the cumulants matching the one which one would get from Proposition~\ref{prop:kappa_local}. The global mode will be recovered later in this section relying on the following size bound.

\begin{proposition}\label{prop:Theta_improved} Let $W$ be an $N\times N$ Wigner matrix satisfying Assumption~\ref{ass:momass}. For any fixed $p, n,S\in \N$ with $p\ge 3$ it holds that
\begin{align}
\Xi_t^{(p)}(n,S)\lesssim & \left(\frac{N^{(n-p)/2}}{N^{(p-2)/2}} + \frac{(N\eta_t)^{n-p}}{(N\eta_t)^{p-2}}\right)\frac{N^\xi}{\eta_t^S},\label{eq:Xi_improved}\\
\Theta_t^{(p)}(n,S)\lesssim & N^{(n-2p+2)/2+\xi},
\label{eq:Theta_improved}
\end{align}
for any fixed $\xi>0$, uniformly in $t\in [0,T]$ and $z=z_T\in\mathcal{D}$.
\end{proposition}

The proof of Proposition~\ref{prop:Theta_improved} relies on the dynamic master inequalities analogous to the ones for $\Upsilon_t$ in Lemma~\ref{lem:dynamic_master}.\nc

\begin{lemma}[Dynamic master inequalities for cumulants]\label{lem:master_Xi} Let $W$ be an $N\times N$ Wigner matrix satisfying Assumption~\ref{ass:momass}. For any fixed $p, n,S\in \N$ with $p\ge 3$ it holds that
\begin{equation}
\begin{split}
\Xi_t^{(p)}(n,S) \lesssim &\Xi_0^{(p)}(n,S) + \int_0^t \left(\Xi^{(p)}_s(n,S-1) +\frac{1}{N}\Xi^{(p)}_s(n+1,S+2)\right) \dif s\\
+&\int_0^s\left(\sum_{S'\le S} \frac{1}{N\eta^{S-S'+2}}\Xi^{(p)}_s(n-1,S') + \sum_{S'\le S-2} \frac{1}{\eta^{S-S'+1}}\Xi^{(p)}_s(n-2,S')\right)\dif s\\
+&\int_0^s\left(\sum_{S'\le S} \frac{1}{N\eta^{S-S'+2}}\Xi^{(p-1)}_s(n-1,S') + \sum_{S'\le S-2} \frac{1}{\eta^{S-S'+1}}\Xi^{(p-1)}_s(n-2,S')\right)\dif s,\\
\Theta_t^{(p)}(n,S) \lesssim &\Theta_0^{(p)}(n,S) + \int_0^t \left(\Theta^{(p)}_s(n,S-1) +\frac{1}{N}\Theta^{(p)}_s(n+1,S+2)\right) \dif s\\
+&\int_0^t \left(\frac{1}{N}\Theta_s^{(p-1)}(n-1,S) +\Theta_s^{(p-1)}(n-2,S-2)\right)\dif s,
\end{split}
\label{eq:Theta_master}
\end{equation}
for any fixed $\xi>0$, uniformly in $t\in [0,T]$ and $z=z_T\in\mathcal{D}$.
\end{lemma}

\begin{proof}[Proof of Proposition~\ref{prop:Theta_improved}] The proof of \eqref{eq:Xi_improved}--\eqref{eq:Theta_improved} relies on the gradual improvement of the a priori bounds \eqref{eq:Theta_a_priori} via the dynamic master inequalities \eqref{eq:Theta_master} in a way very similar to the argument \eqref{eq:Upsilon_iteration}--\eqref{eq:int_eta} for $\Upsilon_t$. We omit further details.
\end{proof}

The proof of Lemma~\ref{lem:master_Xi} proceeds by differentiation of cumulants along the flow \eqref{eq:OU_flow}, \eqref{eq:char_flow}.
The following simple statement shows how to perform this differentiation.

\begin{lemma}[It\^{o} formula for cumulants]\label{lem:Ito_cumulants} For $p\in\N$ and $T>0$, let $X_{1,t},\ldots,X_{p,t}$, $t\in[0,T]$, be complex-valued diffusion processes defined on the same probability space, i.e.
\begin{equation}
\dif X_{j,t}=\mu_{j,t}\dif t + (\bm{\sigma}_t\dif B_t)_j,
\end{equation}
where $\mu_{j,t}\in\C$, $\bm{\sigma}_t\in\C^{p\times d}$ for some $d\in\N$, and $\dif B_t$ is the standard Brownian motion in $\R^d$. Assume that $X_{j,t}$ has all finite moments for all $j\in[p]$ and $t\in [0,T]$. Then it holds that
\begin{equation}
\frac{\dif}{\dif t}\kappa\left(\lbrace X_{j,t}\rbrace_{j\in[p]}\right) = \sum_{i=1}^p \kappa\left(\mu_{i,t},\lbrace X_{j,t}\rbrace_{j\in[p]\setminus\{i\}}\right) + \sum_{1\le i<h\le p}\kappa\left(\frac{\dif}{\dif t}[X_i,X_h]_t,\lbrace X_{j,t}\rbrace_{j\in[p]\setminus\{i,h\}}\right), 
\label{eq:Ito_cumulants}
\end{equation} 
for any $t\in [0,T]$. 
\end{lemma}

The proof of Lemma~\ref{lem:Ito_cumulants} is presented in Appendix~\ref{app:Ito_cumulants}. In the proof of Lemma~\ref{lem:master_Xi} we take $X_{j,t}:=\Delta_t(\bm{k}_j)\overline{\Delta_t(\bm{l}_j)}$ for $j\in[p]$. The remaining calculations and estimates are the same as in the proof of Lemma~\ref{lem:dynamic_master}, we omit them as well.  This finishes the proof of Lemma~\ref{lem:master_Xi}.\nc

\smallskip

Now we turn to the proof of \eqref{eq:kappapq_local}. We fix $p,q\in\N\cup\{0\}$ with $p+q\ge 3$ and differentiate $\kappa_{p,q}(\Delta_t(1))$. For a random variable $X$ and $u,v\in\N\cup\{0\}$ it is convenient to denote by $\{X\}_{u,v}$ the set with $u$ elements $X$ and $v$ elements $\overline{X}$. Using this notation along with \eqref{eq:Tr_dif}, \eqref{eq:cov_process_traces}, \eqref{eq:cov_process_traces_star} and~\eqref{eq:Ito_cumulants}, we get
\begin{equation}
\begin{split}
\frac{\dif}{\dif t} \kappa_{p,q}\left(\Delta_t(1)\right) =& f_t(p,q)\kappa_{p,q}(\Delta_t(1)) + \widetilde{\mathcal{F}}_t(p,q),\qquad f_t(p,q):=\frac{p+q}{2} + p m_{2,t} + q\overline{m}_{2,t},\\
\widetilde{\mathcal{F}}_t(p,q):= &\frac{1}{N}\left(p\kappa\left(\Delta_t(1,2),\{\Delta_t(1)\}_{p-1,q}\right) + q \kappa\left(\overline{\Delta_t(1,2)},\{\Delta_t(1)\}_{p,q-1}\right)\right)\\
+&\frac{\bm{1}_{\beta=1}}{N}\left(p\kappa\left(\Delta_t(3),\{\Delta_t(1)\}_{p-1,q}\right)+q\kappa\left(\overline{\Delta_t(3)},\{\Delta_t(1)\}_{p,q-1}\right)\right)\\
+&\frac{2}{\beta N}\left( \frac{p(p-1)}{2}\kappa\left( \Delta_t(4),\{\Delta_t(1)\}_{p-2,q}\right) + \frac{q(q-1)}{2}\kappa\left( \overline{\Delta_t(4)},\{\Delta_t(1)\}_{p,q-2}\right)  \right)\\
+& \frac{2pq}{\beta N} \kappa\left( \Tr\left[(G_tG^*_t)^2\right],\{\Delta_t(1)\}_{p-1,q-1}\right).
\end{split}
\label{eq:kappa_dif}
\end{equation}
Denote $n:=p+q$. We first bound $\widetilde{F}_t(p,q)$ in terms of $\Xi_t$ and then apply~\eqref{eq:Xi_improved}:
\begin{equation}
\begin{split}
&\left\vert\widetilde{F}_t(p,q)\right\vert \lesssim \frac{1}{N}\Xi^{(n)}\!(n+1,n+2) + \frac{1}{N}\Xi^{(n-1)}\!(n-1,n+2)\\
&\quad +\frac{1}{N\eta^2_t}\Xi^{(n-1)}\!(n-1,n) + \frac{1}{N\eta^3_t}\Xi^{(n-1)}\!(n-1,n-1)\lesssim  \frac{N^\xi}{N\eta}\left(\frac{1}{N^{(n-3)/2}} + \frac{1}{(N\eta)^{n-3}}\right)\frac{1}{\eta_t^{n+1}}. 
\end{split}
\label{eq:F_cum_bound}
\end{equation}
Solving the equation in the first line of~\eqref{eq:kappa_dif} as it was done in~\eqref{eq:E_dif_solved}, and using~\eqref{eq:F_cum_bound} along with \eqref{eq:int_eta} and the bound $|f_t(p,q)|\lesssim 1$, we obtain
\begin{equation}
\kappa_{p,q}(\Delta_t(1)) = \exp\left\lbrace \int_0^t f_r(p,q)\dif r\right\rbrace \kappa_{p,q}(\Delta_0(1)) + \frac{N^\xi}{N\eta_t}\frac{1}{\eta_t^S}\mathcal{O}\left(\frac{1}{N^{(n-3)/2}} + \frac{1}{(N\eta)^{n-3}}\right).
\label{eq:kappa_pq_prefin_id}
\end{equation}
It remains to apply Proposition~\ref{prop:global} for $\kappa_{p,q}(\Delta_0(1))$ and observe that
\begin{equation}
\begin{split}
&\frac{\dif}{\dif t} K_{p,q}(z_t) \kappa_{p+q}^{\dif}(t) = f_t(p,q)\kappa_{p+q}^{\dif}(t), \quad\text{since}\\
&\kappa_{p+q}^{\dif}(t) = \ee^{-t(p+q)/2}\kappa_{p+q}^\dif(0)\quad\text{and}\quad \frac{\dif}{\dif t} m_{2,t}=\left(1+m_{2,t}\right)m_{2,t}.
\end{split}
\label{eq:dif_main}
\end{equation}
Solving the differential equation in the first line of \eqref{eq:dif_main} we get
\begin{equation}
\exp\left\lbrace \int_0^t f_r(p,q)\dif r\right\rbrace K_{p,q}(z_0) \kappa_{p+q}^\dif(0) = K_{p,q}(z_t) \kappa_{p+q}^\dif(t).
\end{equation}
This finishes the proof of~\eqref{eq:kappapq_local}. The proof of~\eqref{eq:main_p0_local} is identical to the one presented above, with the only difference that \eqref{eq:Theta_improved} is used instead of~\eqref{eq:Xi_improved}.
 This completes the proof of Proposition~\ref{prop:kappa_local}.\hfill\qed\nc

\subsection{Proof of Proposition~\ref{prop:E_local}(ii)}\label{sec:kappa_local} 

The proof of Proposition~\ref{prop:E_local}(ii) closely parallels the proof of  \eqref{eq:kappapq_local}
presented in the previous section, which is based on Proposition~\ref{prop:Theta_improved} and the explicit calculation \eqref{eq:kappa_dif}. We keep \eqref{eq:kappa_dif} unchanged and use the following size bound instead of Proposition~\ref{prop:Theta_improved}.

\begin{lemma}\label{lem:kappa_eta} Fix $p\ge 3$, $n_{1,j},n_{2,j}\in \N\cup\{0\}$ and multi-indices $\bm{k}_j\in\N^{n_{1,j}}$, $\bm{l}_j\in\N^{n_{2,j}}$ for $j\in [p]$. Denote $n:=\sum_{j=1}^p (n_{1,j}+n_{2,j})$ and assume that $n<2p-2$. Then for $\beta=2$ we have
\begin{equation}
\Big|\kappa\Big(\left\lbrace \eta_t^{|\bm{k}_j|+|\bm{l}_j|}\Delta_t(\bm{k}_j)\overline{\Delta_t(\bm{l}_j)}\right\rbrace_{j=1}^p\Big)\Big|\lesssim \eta_t,
\label{eq:kappa_eta}
\end{equation}
uniformly in $t\in [0,T]$ and $z=z_T\in\mathcal{D}$.
\end{lemma}

\begin{proof}[Proof of Lemma~\ref{lem:kappa_eta}] First, we combine Propositions~\ref{prop:E_Gauss} and~\ref{prop:E_local}(i) and get
\begin{equation}
\E \left[\eta_t^{|\bm{k}|+|\bm{l}|}\Delta_t(\bm{k}) \overline{\Delta_t(\bm{l})}\right] = \Phi_0(\bm{k},\bm{l}) + \mathcal{O}(\eta_t),
\label{eq:E_eta}
\end{equation}
for any fixed multi-indices $\bm{k}$ and $\bm{l}$. Here $\Phi_0$ is a constant given by~\eqref{eq:Phi_0_expl}, which does not depend on $N$ and $z_t$. Next, we use the well-known cumulant-moment relations
\begin{equation}
\kappa\left(\{Y_j\}_{j=1}^p\right) = \sum_{\mathcal{P}} (-1)^{\mathfrak{n}(\mathcal{P})-1}(\mathfrak{n}(\mathcal{P})-1)! \prod_{B\in \mathcal{P}} \E \prod_{j\in B} Y_j,
\label{eq:cumulant_moment}
\end{equation}
for $Y_j:=\eta_t^{|\bm{k}_j|+|\bm{l}_j|}\Delta_t(\bm{k}_j)$. In \eqref{eq:cumulant_moment}, $\mathcal{P}$ is the family of all partitions of the set $[p]$, and $\mathfrak{n}(\mathcal{P})$ is the number of blocks in $\mathcal{P}$. Together with \eqref{eq:E_eta} applied to each of the expectations in the rhs. of \eqref{eq:cumulant_moment}, this implies 
\begin{equation}
\kappa\Big(\left\lbrace \eta_t^{|\bm{k}_j|+|\bm{l}_j|}\Delta_t(\bm{k}_j)\overline{\Delta_t(\bm{l}_j)}\right\rbrace_{j=1}^p\Big) = \widehat{\Phi}_0\left(\{(\bm{k}_j,\bm{l}_j)\}_{j=1}^p\right) + \mathcal{O}(\eta_t),
\label{eq:kappa_Phi}
\end{equation}
for some constant $\widehat{\Phi}_0$ which does not depend on $N$ and $z_t$. 

 To prove Lemma~\ref{lem:kappa_eta} it remains to show that $\widehat{\Phi}_0=0$. Though $\widehat{\Phi}_0$ can be explicitly represented in terms of $\Phi_0$ from \eqref{eq:Phi_0_expl}, and thus  $\widehat{\Phi}_0=0$ may
 be verified directly via an algebraic cancellation, we proceed with a shorter indirect proof. 
We have from \eqref{eq:Xi_improved} that
\begin{equation}
\Big|\kappa\Big(\left\lbrace \eta_t^{|\bm{k}_j|+|\bm{l}_j|}\Delta_t(\bm{k}_j)\overline{\Delta_t(\bm{l}_j)}\right\rbrace_{j=1}^p\Big)\Big|\lesssim N^\xi\left(\frac{1}{N^{1/2}} + \frac{1}{N\eta_t}\right)^{2p-2-n},
\label{eq:kappa_Phi_aux}
\end{equation}
for any fixed $\xi>0$. One may further assume that $\eta_T \le N^{-\xi_0}$ for some fixed (small) $\xi_0>0$. If this is not the case, the trajectory of the characteristic flow \eqref{eq:char_flow} can be extended so that this bound holds. Combining \eqref{eq:kappa_Phi} with \eqref{eq:kappa_Phi_aux} and taking $t:=T$, we thus obtain
\begin{equation}
\left\vert \widehat{\Phi}_0\left(\{(\bm{k}_j,\bm{l}_j)\}_{j=1}^p\right)\right\vert \lesssim \eta_T + N^\xi\left(\frac{1}{N^{1/2}} + \frac{1}{N\eta_T}\right)^{2p-2-n}\lesssim N^{-\xi_0/2}.
\label{eq:Phi_hat_bound}
\end{equation}
In the last bound we used that $n-2p+2<0$ and took $\xi<\xi_0/2$. Since the lhs. of \eqref{eq:Phi_hat_bound} does not depend on $N$, it equals to zero. Together with \eqref{eq:kappa_Phi} this finishes the proof of Lemma~\ref{lem:kappa_eta}.
\end{proof}

Now we prove Proposition~\ref{prop:E_local}(ii) and distinguish between the two cases: $p+q\ge 4$ and $p+q=3$.

\smallskip

\noindent\underline{$p+q\ge 4$.} We apply the resolvent identity from Lemma~\ref{lem:resolvent_id} to the last line of~\eqref{eq:kappa_dif} and observe that the condition of Lemma~\ref{lem:kappa_eta} holds for all cumulants contributing to $\widetilde{\mathcal{F}}_t(p,q)$. Therefore, Lemma~\ref{lem:kappa_eta} implies that
\begin{equation}
\big|\widetilde{\mathcal{F}}_t(p,q)\big| \lesssim \left(N\eta_t^{p+q+1}\right)^{-1},
\label{eq:F_gain_eta}
\end{equation}
uniformly in $t\in [0,T]$ and $z=z_T\in\mathcal{D}$.  Arguing as in \eqref{eq:kappa_pq_prefin_id} and using \eqref{eq:F_gain_eta} instead of~\eqref{eq:F_cum_bound}, we get
\begin{equation}
\kappa_{p,q}(\Delta_t(1)) = \exp\left\lbrace \int_0^t f_r(p,q)\dif r\right\rbrace \kappa_{p,q}(\Delta_0(1)) + \mathcal{O}\left((N\eta_t^{p+q})^{-1}\right).
\label{eq:kappa_pq_prefin_N_id}
\end{equation}
The first term in the rhs. of \eqref{eq:kappa_pq_prefin_N_id} has an upper bound of order $N^{-1}$ by Proposition~\ref{prop:global} and \eqref{eq:m_bound}. This completes the proof of Proposition~\ref{prop:E_local}(ii) for $p+q\ge 4$.

\smallskip

\noindent\underline{$p+q=3$.} Without loss of generality we may assume that $p\ge q$. In the case $p=3,q=0$, Proposition~\ref{prop:E_local}(ii) immediately follows from \eqref{eq:main_p0_local}, so the only non-trivial case is $p=2,q=1$. We have from \eqref{eq:kappa_dif} that
\begin{equation}
\begin{split}
\widetilde{\mathcal{\mathcal{F}}}_t(2,1) = &\frac{1}{N}\left( 2\kappa\left(\Delta_t(1,2),\Delta_t(1),\overline{\Delta_t(1)}\right) + \kappa\left(\Delta_t(1),\Delta_t(1),\overline{\Delta_t(1,2)}\right)\right)\\
+&\frac{1}{N}\left( \kappa\left(\Delta_t(4),\overline{\Delta_t(1)}\right) + 2\kappa\left(\Tr \left[(G_tG^*_t)^2\right],\Delta_t(1)\right)\right).
\end{split}
\label{eq:F_21}
\end{equation}
Note that Lemma~\ref{lem:kappa_eta} is not applicable to the cumulants in the rhs. of~\eqref{eq:F_21},  since the condition $n<2p-2$ is violated. In fact, the smallness of $\big|\widetilde{\mathcal{\mathcal{F}}}_t(2,1)\big|$ does not come from smallness of each of the cumulants, but from the cancellation between them. We now identify this cancellation, at least to leading order. Using the cumulant-moment relations~\eqref{eq:cumulant_moment},  \eqref{eq:E_eta} and the explicit formula for $\Phi_0$ from \eqref{eq:Phi_0_expl}, we get
\begin{equation}
\begin{split}
\eta_t^5\kappa\left(\Delta_t(1,2),\Delta_t(1),\overline{\Delta_t(1)}\right) &= \mathcal{O}(\eta_t),\\
\eta_t^5\kappa\left(\Delta_t(1),\Delta_t(1),\overline{\Delta_t(1,2)}\right)&= \Phi_0\left((1,1),(1,2)\right) +\mathcal{O}(\eta_t) = (8\ii)^{-1} +\mathcal{O}(\eta_t),\\
\eta_t^5\kappa\left(\Delta_t(4),\overline{\Delta_t(1)}\right)&= \Phi_0(4,1)+\mathcal{O}(\eta_t) =(8\ii)^{-1}+\mathcal{O}(\eta_t),\\
\eta_t^5\kappa\left(\Tr \left[(G_tG^*_t)^2\right],\Delta_t(1)\right)&= -(4\ii)^{-1}\eta_t^2\kappa\left(\overline{\Delta_t(1)},\Delta_t(1)\right) -4^{-1}\eta_t^3\kappa\left(\overline{\Delta_t(2)},\Delta_t(1)\right) +\mathcal{O}(\eta_t)\\
&= -(4\ii)^{-1}\Phi_0(1,1) -4^{-1}\Phi_0(1,2) +\mathcal{O}(\eta) = -(8\ii)^{-1}+\mathcal{O}(\eta_t).
\end{split}
\end{equation}
Together with \eqref{eq:F_21} this gives
\begin{equation}
\big|\widetilde{\mathcal{\mathcal{F}}}_t(2,1)\big|\lesssim (N\eta_t^4)^{-1}.
\end{equation}
The remaining details are the same as in the previous case and thus are omitted.

\section{Multiple moment matching: proof of Lemma~\ref{lem:moment_match}}\label{sec:moment_match}

Without loss of generality assume that $\Gamma$ is even (otherwise, replace $\Gamma$ by $\Gamma+1$). We start with the somewhat simpler real case, then we prove the complex case as well.

\noindent\underline{Real case.} Denote\footnote{Here $m\in\N$ should not be confused with $m(z)$ introduced in \eqref{eq:def_m}, the latter one does not appear in Section~\ref{sec:moment_match}.} $\Gamma=2m$ and let $H_{2m}$ be the Hankel matrix associated to the first $2m$ moments of $\chi$:
\begin{equation*}
H_{2m}:=\left( \E \chi^{i+j}\right)_{i,j=0}^{m}\in  \R^{(m+1)\times (m+1)}.
\end{equation*}
For any non-zero vector $\bm{x}=(x_i)_{i=0}^m\in\C^{m+1}\setminus\{0\}$ we have
\begin{equation*}
\langle \bm{x}, H_{2m}\bm{x}\rangle = \E \left\vert P(\chi) \right\vert^2 >0,\quad \text{where}\quad P(\chi):= \sum_{j=0}^m x_j \chi^j.
\end{equation*}
In the last step we used the assumption that the support of $\mu$ contains at least $m+1$ point, so $\mu$ cannot be supported on the roots of the polynomial $P$ of degree at most $m$. Thus, $H_{2m}$ is positive definite, i.e. there exists $c_0 = c_0(m) >0$ such that $H_{2m}\ge c_0$ in the sense of quadratic forms.

Consider any sequence of real numbers $s_j\in [\E \chi^j - c_1, \E \chi^j + c_1]$, $j=0,1,\ldots,2m$, $0<c_1<c_0$. Since $H_{2m}\ge c_0$, for sufficiently small $c_1=c_1(c_0,m)>0$ it holds that 
\begin{equation}
\widetilde{H}_{2m}^{(1)}\ge c_0/2,\quad\text{where}\quad \widetilde{H}_{2m}^{(1)}:=(s_{i+j})_{i,j=0}^m\in \R^{(m+1)\times (m+1)}
\label{eq:H_bound1}
\end{equation}
by continuity. Moreover, for sufficiently large $L>0$ we have
\begin{equation}
\widetilde{H}_{2m}^{(2)}\ge c_0/2,\quad\text{where}\quad\widetilde{H}_{2m}^{(2)} = \widetilde{H}_{2m}^{(2)}(L) :=\left( s_{i+j} -L^{-2}s_{i+j+2} \right)_{i,j=0}^{m-1}\in \R^{m\times m}.
\label{eq:H_bound2}
\end{equation}
By \cite[Theorem~3.1]{moment_problem} (see also \cite{Krein73} for the original result), \eqref{eq:H_bound1} together with \eqref{eq:H_bound2} imply that 
the \emph{truncated  Hausdorff \nc moment problem} for $\{s_j\}_{j=0}^{2m}$ on $[-L,L]$ has a solution, i.e. that there exists a Radon measure $\nu$ supported on $[-L, L]$, such that 
\begin{equation*}
\int_\R x^j\dif\nu(x)=s_j, \quad \text{for}\quad j=0,1,\ldots,2m.
\end{equation*}
Now we choose $\{s_j\}_{j=0}^{2m}$ as a solution to the system of linear equations
\begin{equation}
\sum_{k=0}^j \left[\binom{j}{k}(1-\alpha)^{k/2} \alpha^{(j-k)/2} \E (\chi^G)^{j-k}\right] s_k = \E \chi^j,
\label{eq:s_syst}
\end{equation}
where $\chi^G\sim\mathcal{N}(0,1)$. Since the matrix of this linear system is upper-triangular with non-zero entries on the diagonal, there exists the solution to \eqref{eq:s_syst}. Moreover, it is a small perturbation of $\{\E\chi^j\}_{j=0}^{2m}$ for sufficiently small $\alpha>0$. Let $\chi_\alpha$ be a random variable with the distribution $\nu$, which solves the truncated moment problem for $\{s_j\}_{j=0}^{2m}$. Then \eqref{eq:s_syst} implies that \eqref{eq:moment_match} holds. This finishes the proof of Lemma~\ref{lem:moment_match} in the real case.

\medskip

\noindent\underline{Complex case.}   Let $\mu$ be the distribution of $\chi$. \nc For convenience, we view $\mu$ as a measure on $\R^2$. We start with setting up some notation. For any $s\in \N$, let $\R[x,y]_s$ be the set of all polynomials in variables $x,y$ with real coefficients of degree at most $s$. We also denote
\begin{equation*}
\mathrm{Pos}(K)_s:=\left\lbrace P\in \R[x,y]_s\,:\, P(x,y)\ge 0, \, \forall (x,y)\in K\right\rbrace,
\end{equation*}
for a compact subset $K\subset \R^2$. Finally, for any measure $\nu$ on $\R^2$ and any $\nu$-integrable function $f:\R^2\to \R$ we call $\int_{\R^2} f\dif\nu$ the action of $\nu$ on $f$.

Consider the action of $\mu$ on $\R[x,y]_{2\Gamma+4}$. By the Richter theorem \cite[Theorem~8.1]{moment_problem}   (see also \cite[Satz~4]{Richter57} for the original result) 
there exists a purely atomic measure $\mu_1=\sum_{j=1}^k m_j\delta_{(x_j,y_j)}$, for some $k\in\N$, with the same action on $\R[x,y]_{2\Gamma+4}$,~i.e.
\begin{equation*}
\int P\dif \mu = \int P \dif\mu_1,\quad \forall P\in \R[x,y]_{2\Gamma+4}.
\end{equation*}
Let $L>0$ be sufficiently large, so that $|x_j|,|y_j|< L$, for $j\in [k]$. Denote $K:=[-L,L]^2$ and consider the action of $\mu_1$ on $\R(K)_{\Gamma+2}$. If there exists $P\in \mathrm{Pos}(K)_{\Gamma+2}\setminus \{0\}$ such that $\int P\dif \mu_1=0$, then $\mu_1$ is supported on the zero set of $P$. Therefore,
\begin{equation*}
0=\int P^2 \dif\mu_1 = \int P^2 \dif \mu,
\end{equation*}
where we used that $P^2\in \R[x,y]_{2\Gamma+4}$. However, $P^2(x,y)\ge 0$ for any $x,y\in\R$, so $\mu$ is supported on the zero set of $P^2$. This contradicts to the assumption (ii). In such a way, $\int P\dif \mu_1>0$ for any non-zero $P\in \mathrm{Pos}(K)_{\Gamma+2}$. By a compactness argument, there exists $c_0>0$ such that
\begin{equation*}
\int_K P\dif \mu_1 > c_0,\quad \forall\, P=\sum_{i+j\le \Gamma+2} a_{ij}x^iy^j \in\mathrm{Pos}(K)_{\Gamma+2}\,\,\,\,\text{with}\,\,\,\, |||P|||:=\!\!\!\sum_{i+j\le \Gamma+2} |a_{ij}| =1.
\end{equation*}

Now we present the analogue of \eqref{eq:s_syst} for the 2-dimensional case. Define $\{s_{\gamma_1\gamma_2}\}_{\gamma_1+\gamma_2\le \Gamma+2}$ as the solution to the system 
\begin{equation}
\sum_{\gamma_1=0}^i \sum_{\gamma_2=0}^j \left[\binom{i}{\gamma_1}\binom{j}{\gamma_2} (1-\alpha)^{(\gamma_1+\gamma_2)/2}\alpha^{(i+j-\gamma_1-\gamma_2)/2}\E (\chi^G_1)^{i-\gamma_1}\E (\chi^G_2)^{j-\gamma_2}\right] s_{\gamma_1\gamma_2} = m_{ij}
\label{eq:s_syst2}
\end{equation}
for $i,j\ge 0$, $i+j\le \Gamma+2$, where $m_{i,j}:=\E \left[ (\Re\chi)^i (\Im\chi)^j\right]$ and $\chi_1^G, \chi_2^G$ are independent standard real-valued Gaussian random variables. Since for $\alpha=0$ the matrix of this system coincides with identity, \eqref{eq:s_syst2} has a unique solution in the perturbative regime when $\alpha>0$ is sufficiently small. Moreover, this solution is a small perturbation of $\{m_{ij}\}_{i+j\le \Gamma+2}$. 

Define the linear functional $\mathrm{L}$ on $\R[x,y]_{\Gamma+2}$ by
\begin{equation*}
\mathrm{L}\left(\sum_{i+j\le \Gamma+2} a_{ij}x^iy^j\right):= \sum_{i+j\le \Gamma+2} a_{ij} s_{ij}.
\end{equation*}
For any $P\in \R[x,y]_{\Gamma+2}$ we have
\begin{equation*}
\left\vert \mathrm{L}(P) - \int_K P\dif\mu_1\right\vert\le |||P||| \max_{i+j\le n+2} |s_{ij}-m_{ij}|.
\end{equation*}
In particular, $\mathrm{L}(P)\ge 0$ for any $P\in\mathrm{Pos}(K)_{\Gamma+2}$ for sufficiently small $\alpha>0$.  By \cite[Theorem~8.7]{moment_problem}, this positivity implies that the action of $\mathrm{L}$ restricted to $\R[x,y]_\Gamma$  coincides with the action of some compactly supported on $K$ measure $\nu$, i.e. 
\begin{equation*}
\int_K x^iy^j \dif\nu = s_{ij},\quad i+j\le \Gamma.
\end{equation*}
Now it obviously follows from \eqref{eq:s_syst2} that the random variable $\chi_\alpha$ with distribution $\nu$ satisfies the moment matching condition \eqref{eq:moment_match}.\hfill\qed

\appendix

\section{Proofs of additional technical results}

\subsection{Proof of Lemma~\ref{lem:underline_exp_general}}\label{app:underline}

Denote for short $f_1:=\Tr[\underline{WG}]$. We apply the cumulant-moment relations \eqref{eq:cumulant_moment} for $Y_j:=f_j$ and obtain
\begin{equation}
\kappa\left(\{f_j\}_{j=1}^p\right) = \sum_{\mathcal{P}} (-1)^{\mathfrak{n}(\mathcal{P})-1}(\mathfrak{n}(\mathcal{P})-1)! \prod_{B\in \mathcal{P}} \E \prod_{j\in B} f_j.
\label{eq:cumulant_moment_f}
\end{equation}
Let $B_1\in\mathcal{P}$ be the block containing index $1$. Denote the size of this block by $s=s(\mathcal{P})$ and label its elements by $j_1,j_2,\ldots,j_s$ with $j_1=1$. We apply \eqref{eq:def_under_1} to the product associated to $B_1$ and get
\begin{equation}
\prod_{j\in B_1} f_j = \Tr[\underline{WF}] f_{j_2}\cdots f_{j_s} = \underline{\Tr[WF]f_{j_2}\cdots f_{j_s}} + \sum_{j\in B_1\setminus\{1\}} \widetilde{\E}\left[\Tr[WF] \partial_{\widetilde{W}} f_{j}\right] \prod_{i\in B_1\setminus\{1,j\}} f_j.
\label{eq:under_ext_B1}
\end{equation} 
Fix any index $j_0\in [2,p]$ and consider the sum over all $\mathcal{P}$ in \eqref{eq:cumulant_moment_f} such that $j_0\in B_1$. Apply \eqref{eq:under_ext_B1} to each of these terms and consider only the summands with $\partial_{\widetilde{W}}f_{j_0}$. Summing all these terms back, we obtain the term in the rhs. of the first line of \eqref{eq:underline_exp_general} with $j=j_0$. Thus, it remains to show that 
\begin{equation}
\sum_{\mathcal{P}} (-1)^{\mathfrak{n}(\mathcal{P})-1}(\mathfrak{n}(\mathcal{P})-1)! \prod_{B\in \mathcal{P}\setminus\{B_1\}} \E \bigg[\prod_{j\in B} f_j\bigg] \E\left[\underline{\Tr[WF]f_{j_2}\cdots f_{j_s}}\right]
\label{eq:full_underline_cumulant}
\end{equation}
equals to the second line of \eqref{eq:underline_exp_general}.

In the last factor in the rhs. of \eqref{eq:full_underline_cumulant} we perform the cumulant expansion, \cite[Lemma~3.1]{Knowles20} for $\beta=1$ and \cite[Lemma~3.2]{Knowles20} for $\beta=2$ (see also \cite[Lemma~3.1]{HeKnowles} and \cite[Section~II]{Khorunzhy96}):
\begin{equation}
\E\left[\underline{\Tr[WF]f_{j_2}\cdots f_{j_s}}\right] = \sum_{\ell=3}^L \sum_{a,b=1}^N\sum_{\bm{\alpha}\in\{ab,ba\}^{\ell-1}}\frac{\kappa(ab,\bm{\alpha})}{(\ell-1)!N^{\ell/2}} \E \partial_{\bm{\alpha}} \left[F_{ba} f_{j_2}\cdots f_{j_s}\right]+\mathcal{R}_{L+1}.
\label{eq:cum_expansion_appl}
\end{equation}
Here $L$ is an $N$-independent positive integer which is the same for all $\mathcal{P}$ in \eqref{eq:full_underline_cumulant}, $L$ will be chosen to be sufficiently large in a moment. The error term $\mathcal{R}_{L+1}=\mathcal{R}_{L+1}(B_1)$ admits the bound
\begin{equation}
\left\vert\mathcal{R}_{L+1}(B_1)\right\vert \lesssim \sum_{a,b=1}^N \frac{1}{N^{(L+1)/2}}\sum_{r=0}^\ell \left\lVert\E^{(ab)}\partial_{ab}^r\partial_{ba}^{L-r} \left[F_{ba} f_{j_2}\cdots f_{j_s}\right]\right\rVert_{\infty,ab}.
\end{equation}
Here $\E^{(ab)}$ is the expectation wrt. all entries of $W$ apart from $w_{ab}=\overline{w_{ba}}$, and $\|\cdot\|_{\infty,ab}$ is the infinity-norm in this entry. As $z\in \mathcal{D}_g$ is in the global regime, $\|G(z)\|\lesssim 1$ and thus this norm has an upper bound of order $N^n$ with $n:=\sum_{j=1}^p(n_{1,j}+n_{2,j})$. Therefore,
\begin{equation}
\left\vert\mathcal{R}_{L+1}\right\vert\lesssim N^{-(L+1)/2 +n+2}.
\end{equation}
Choosing a sufficiently large $N$-independent $L\in\N$, we obtain an upper bound of order $N^{-D}$ on $\left\vert\mathcal{R}_{L+1}\right\vert$.

Distributing the derivatives over the factors in \eqref{eq:cum_expansion_appl}, we get
\begin{equation}
\partial_{\bm{\alpha}} \left[F_{ba} f_{j_2}\cdots f_{j_s}\right] = \sum_{P_s(\bm{\alpha})} \partial_{\bm{\alpha}_1} [F_{ba}] \partial_{\bm{\alpha}_{j_2}} [f_{j_2}]\cdots \partial_{\bm{\alpha}_{j_s}} [f_{j_s}].
\label{eq:alpha_deriv}
\end{equation}
We further extend $P_s(\bm{\alpha})$ to the partition $P_p(\bm{\alpha})$ consisting of $n$ blocks by setting $\bm{\alpha}_j:=\emptyset$ for $j\notin B_1$. Now we fix $\ell\in [3,L]$, $a,b\in [N]$, $\bm{\alpha}\in \{ab,ba\}^{\ell-1}$ and a partition $P_p(\bm{\alpha})$. Combine \eqref{eq:cum_expansion_appl} with \eqref{eq:alpha_deriv} and consider only the terms with selected $\ell,a,b,\bm{\alpha}$ and $P_p(\bm{\alpha})$. The argument similar to the one below~\eqref{eq:under_ext_B1} shows that the sum of these terms equals to the corresponding term in the second line of \eqref{eq:underline_exp_general}. This finishes the proof of Lemma~\ref{lem:underline_exp_general}.\hfill\qed

\subsection{Proof of Lemma~\ref{lem:Upsilon_0}}\label{app:Upsilon_0}

Due to the bound in the Gaussian case \eqref{eq:E_1bound}, Lemma~\ref{lem:Upsilon_0} is equivalent to the following statement.

\begin{lemma}\label{lem:E_global} Let $W$ be an $N\times N$ Wigner matrix satisfying Assumption~\ref{ass:momass}. Fix a (small) $\varepsilon>0$ and a (large) $C>0$. It holds that
\begin{equation}
\E^{(W)} \left[\Delta(\bm{k})\overline{\Delta(\bm{l})}\right]\lesssim 1.
\end{equation}
uniformly in $z\in \C$ with $|\Im z|\ge \varepsilon$ and $|z|\le C$, for any fixed $n_1,n_2\in\N\cup\{0\}$ and multi-indices $\bm{k}\in\N^{n_1}$, $\bm{l}\in\N^{n_2}$.
\end{lemma}

\begin{proof}[Proof of Lemma~\ref{lem:E_global}] Assume without loss of generality that $\bm{k}\neq \emptyset$. The proof of Lemma~\ref{lem:E_global} is analogous to the proof of \eqref{eq:E_1bound} and goes through the static expansion~\eqref{eq:E_expansion}. The only additional step is the estimate on the third and higher order cumulant terms naturally arising in the rhs. of~\eqref{eq:E_expansion} in the case when the entries of $W$ are not Gaussian. Specifically, we need to show that
\begin{equation}
\left\vert\E \underline{\Tr [WG^{k_1}]\Delta\big(\bm{k}^{[1]}\big)\overline{\Delta(\bm{l})}}\right\vert\lesssim 1.
\label{eq:E_global_underline}
\end{equation}
Fix a $D>0$. As in the proof of Lemma~\ref{lem:underline_exp_general}, we apply the cumulant expansion and obtain that the lhs. of~\eqref{eq:E_global_underline} without the absolute value equals to
\begin{equation}
\sum_{\ell=3}^L \sum_{a,b=1}^N\sum_{\bm{\alpha}\in\{ab,ba\}^{\ell-1}}\frac{\kappa(ab,\bm{\alpha})}{(\ell-1)!N^{\ell/2}}\E \partial_{\bm{\alpha}} \left[(G^{k_1})_{ba}\Delta\big(\bm{k}^{[1]}\big)\overline{\Delta(\bm{l})}\right] +\mathcal{O}\left(N^{-D}\right),
\label{eq:E_global_underline_exp}
\end{equation}
for some $L\in\N$ which does not depend on $N$. The derivation of \eqref{eq:E_global_underline} from \eqref{eq:E_global_underline_exp} is analogous to the argument \eqref{eq:before_iso_resum}--\eqref{eq:iso_resum}. We omit further details as they are standard.
\end{proof}

\subsection{Proof of Lemma~\ref{lem:Ito_cumulants}}\label{app:Ito_cumulants} We apply \eqref{eq:cumulant_moment} for $\{X_{j,t}\}_{j=1}^p$ and perform differentiation in $t$. By the It\^{o}'s formula, for any partition $\mathcal{P}$ we get
\begin{small} 
\begin{equation}
\frac{\dif}{\dif t} \prod_{B\in \mathcal{P}} \E \prod_{j\in B} X_{j,t} \!=\!\! \sum_{B_0\in\mathcal{P}}\left(\sum_{i\in B_0}\E\Bigg[ \mu_{j,t}\!\!\!\!\!\prod_{j\in B_0\setminus\{i\}}\!\!\! X_{j,t}\Bigg] +\!\!\!\! \sum_{i,h\in B_0, i<h}\!\!\E\Bigg[ \frac{\dif}{\dif t}[X_i,X_h]_t\!\!\!\!\!\prod_{j\in B_0\setminus\{i,h\}}\!\!\! X_{j,t}\Bigg] \right) \!\!\!\prod_{B\in\mathcal{P}\setminus B_0} \!\!\!\!\E \prod_{j\in B} \!X_{j,t}.
\end{equation}
\label{eq:moment_dif}
\end{small}
Fix $i\in [p]$. Summing up the first term in the rhs. of \eqref{eq:moment_dif} over all partitions (together with the $\mathcal{P}$-dependent factor from \eqref{eq:cumulant_moment}), we obtain the first term in the rhs. of \eqref{eq:Ito_cumulants} with index $i$. Similarly, for any fixed $i,h\in [p]$, $i<h$, summing up over $\mathcal{P}$ the second term in the rhs. of \eqref{eq:moment_dif}, we obtain the second term in the rhs. of \eqref{eq:Ito_cumulants}. This finishes the proof of Lemma~\ref{lem:Ito_cumulants}.\hfill $\qed$

\subsection{Proof of Proposition~\ref{prop:GFT}}\label{app:GFT}

In the current argument we do not distinguish between traces of resolvents and of and their conjugates, so one may assume for simplicity that $\bm{l}=\emptyset$. Denote $n:=n_1=\mathfrak{n}(\bm{k})$. We follow the Lindeberg replacement strategy \cite{TaoVu11} and borrow the notations from \cite[Sec.16.1]{erdHos2017dynamical}. The spectral parameter $z\in\mathcal{D}$ with $\eta:=|\Im z|$ remains fixed throughout the proof and thus is omitted from notations. Denote $\mathcal{I}:=\{(i,j)\in [N^2], i\le j\}$ and fix a bijective ordering map $\varphi: \mathcal{I}\to [N(N+1)/2]$. For $\upsilon\in [N(N+1)/2]$, let $W_\upsilon=W_\upsilon^*$ be the matrix whose entries coincide with $w_{ij}^{(1)}$ for $\phi(i,j)\le \upsilon$ and with $w_{ij}^{(2)}$ otherwise. We additionally set $W_0:=W^{(1)}$. Let $Q_\upsilon$ be the matrix which coincides with $W_\upsilon$ everywhere apart from the entry $\varphi^{-1}(\upsilon)$ and the symmetric one. These two entries (or one, if $\varphi^{-1}(\upsilon)$ is diagonal) are set equal to zero. Denote the resolvents of $W_\upsilon$ and $Q_\upsilon$ by $S_\upsilon$ and $R_\upsilon$, respectively. Finally, for any multi-index $\bm{k}$ we define $\Delta_\upsilon(\bm{k})$ as in \eqref{eq:def_Delta}, but with $G$ replaced by $G_\upsilon$. We stress that the deterministic counterparts \eqref{eq:def_m} do not depend on $\upsilon$.

By telescopic summation, we estimate
\begin{equation}
\left\vert \E \left[\Delta^{(2)}(\bm{k}) -  \Delta^{(1)}(\bm{k})\right]\right\vert \le \sum_{\upsilon=1}^{N(N+1)/2} \left\vert\E\left[ \Delta_\upsilon (\bm{k})-\Delta_{\upsilon-1} (\bm{k})\right]\right\vert.
\end{equation}
For each $\upsilon\in [N(N+1)/2]$, we expand both $S_\upsilon$ and $S_{\upsilon-1}$ around $R_\upsilon$:
\begin{equation}
\begin{split}
S_\upsilon = \sum_{d=0}^{\Gamma+1} \mathcal{T}^{(2)}_{\upsilon,d},\quad \text{where}\quad \mathcal{T}^{(2)}_{\upsilon,d}&:=N^{-d/2}(-1)^d(R_\upsilon X_\upsilon^{(2)})^dR_\upsilon,\,\, d\in [0,\Gamma]\\
 \mathcal{T}^{(2)}_{\upsilon,\Gamma+1} &:= N^{-(\Gamma+1)/2}(-1)^{\Gamma+1}(R_\upsilon X^{(2)}_\upsilon)^{\Gamma+1}S_\upsilon.
\end{split}
\label{eq:S_expansion}
\end{equation}
Here $X^{(2)}=(X^{(2)})^*$ has at most two non-zero entries:
\begin{equation}
(X^{(2)}_\upsilon)_{ij}:=\delta_{\varphi(i,j)=\upsilon} N^{1/2}w^{(2)}_{ij},\qquad 1\le i \le N.
\label{eq:def_X}
\end{equation}
For $S_{\upsilon-1}$ the expansion is the same with the only difference that all superscripts $(2)$ in \eqref{eq:S_expansion}--\eqref{eq:def_X} should be replaced by $(1)$. Next, we apply \eqref{eq:S_expansion} to each trace $\Delta_\upsilon(k_j)$ in $\Delta_\upsilon(\bm{k})$:
\begin{equation}
\Delta_\upsilon(k_j) = \sum_{\bm{d}_j} T_{\upsilon, \bm{d}_j}^{(2)}, \qquad  T_{\upsilon, 0}^{(2)}:=\Tr\left[R^{k_j}_\upsilon-m_{k_j}\right],\quad  T_{\upsilon, \bm{d}_j}^{(2)}:=\Tr \prod_{i=1}^{k_j}\mathcal{T}_{\upsilon,d_{j,i}}^{(2)},\, \bm{d}_j\neq 0. 
\label{eq:Delta_expansion_GFT}
\end{equation}
The summation in \eqref{eq:Delta_expansion_GFT} over runs over all vectors with integer coordinates $\bm{d}_j=(d_{j,i})_{i=1}^{k_j}\in [0,\Gamma+1]^{k_j}$, and $\bm{d}_j=0$ stands for a vector with zero coordinates.

Now we multiply \eqref{eq:Delta_expansion_GFT} for all $j\in [n]$, subtract the analogous formula for $\Delta_{\upsilon-1}(\bm{k})$, and take the expectation. Since $X^{(1)}_\upsilon, X^{(2)}_\upsilon$ are independent of $R_\upsilon$ by construction, the moment matching condition~\eqref{eq:GFT_condition} gives
\begin{equation}
\E \prod_{j=1}^n T_{\upsilon, \bm{d}_j}^{(2)} - \E \prod_{j=1}^n T_{\upsilon, \bm{d}_j}^{(1)} = 0,\qquad\forall\,\bm{d}_j\in [0,\Gamma]^{k_j}\,\,\,\text{with}\,\,\, \sum_{j=1}^n |\bm{d}_j|\le \Gamma. 
\label{eq:T_cancel}
\end{equation}
In the remaining cases we give an upper bound for each of the two terms in the lhs. of \eqref{eq:T_cancel} separately without exploiting a potential cancellation. We focus on the estimates for $T^{(2)}$, while for $T^{(1)}$ the argument is identical. Since $X^{(2)}$ has at most two-non-zero entries, $T_{\upsilon, \bm{d}_j}^{(2)}$ for $\bm{d}_j\neq 0$ is a sum of $N$-independent number of terms, each of those is a product of $(R_\upsilon)_{ab}$, $(R_\upsilon^2)_{ab}$ and $(S_\upsilon R_\upsilon)_{ab}$, for some $a,b\in [N]$ (some of these factors may be absent). Expanding $R_\upsilon$ around $S_\upsilon$ similarly to \eqref{eq:S_expansion} and using the local laws \eqref{eq:kG_av}--\eqref{eq:kG_iso} for $S_\upsilon$, we get
\begin{equation}
\left\vert (R_\upsilon)_{ab}\right\vert\prec 1,\quad \left\vert (R_\upsilon^2)_{ab}\right\vert\prec \eta^{-1},\quad \left\vert (S_\upsilon R_\upsilon)_{ab}\right\vert\prec \eta^{-1},\quad \left\vert \Tr\left[ R_\upsilon^{k}-m_k\right]\right\vert \prec \eta^{-k},
\end{equation}
uniformly in $a,b\in [N]$ and for any fixed $k\in\N$. For more details see the discussion above (16.8) in \cite{erdHos2017dynamical}. Therefore,
\begin{equation}
\left\vert T_{\upsilon, \bm{d}_j}^{(2)}\right\vert \prec N^{-|\bm{d}_j|/2} \eta^{-k_j}.
\end{equation}
Using this bound for all $\{\bm{d}_j\}_{j=1}^n$ which do not satisfy conditions of \eqref{eq:T_cancel}, we finish the proof of Proposition~\ref{prop:GFT}.\hfill\qed

\subsection{Proof of~\eqref{eq:kappa22_real}}\label{app:kappa22_real} First, we prove \eqref{eq:kappa22_real} for a GOE matrix $W$. Recall the notations introduced in Section~\ref{sec:E_expansion} and observe that~\eqref{eq:Phi_vanish_trivial} holds also in the real case. Therefore, \eqref{eq:E_full_expansion} implies $\E\eta^{|\bm{k}|}\Delta(\bm{k})=\mathcal{O}(\eta)$ for any fixed multi-index $\bm{k}\neq\emptyset$. Together with the cumulant-moment relation \eqref{eq:cumulant_moment} applied to $\kappa_{2,2}(\eta\Delta(1))$ this gives
\begin{equation}
\kappa_{2,2}(\eta\Delta(1)) = \E |\eta\Delta(1)|^4 - 2\left(\E |\eta\Delta(1)|^2\right)^2 +\mathcal{O}(\eta) = \E\big((1,1),(1,1)\big) -2\left(\E(1,1)\right)^2 +\mathcal{O}(\eta).
\label{eq:kappa22_real_basic_id}
\end{equation}
Using \eqref{eq:E_full_expansion}, we expand the rhs. of \eqref{eq:kappa22_real_basic_id} in powers of $(N\eta)^{-1}$ up to the third order. By \eqref{eq:kappapq_new} it holds that
\begin{equation}
\left\vert\kappa_{2,2}(\eta\Delta(1))\right\vert \lesssim \frac{N^\xi}{N\eta}\left(\frac{1}{N^{1/2}}+\frac{1}{N\eta}\right).
\end{equation}
Therefore, the zeroth and first order terms in this expansion vanish. To establish \eqref{eq:kappa22_real} we thus need to show that the second order vanishes as well, and compute the third order term. An explicit calculation based on \eqref{eq:Phi_recursion} gives\footnote{For the mechanical proof of \eqref{eq:E_explicit} we used a simple Python program whose code is provided in the supplementary material.}
\begin{equation}
\begin{split}
\E(1,1)=&\frac{1}{2} + \frac{\ii\ma_0}{2}\frac{1}{N\eta} -\frac{3\ma^2_0}{4}\frac{1}{(N\eta)^2}-2\ii \ma^3_0\frac{1}{(N\eta)^3}+\mathcal{O}\left(\frac{1}{(N\eta)^4}+\eta\right),\\
\E\big((1,1),(1,1)\big)=&\frac{1}{2} + \ii\ma_0 \frac{1}{N\eta} -2\ma^2_0 \frac{1}{(N\eta)^2} -8\ii \ma^3_0 \frac{1}{(N\eta)^3} +\mathcal{O}\left(\frac{1}{(N\eta)^4}+\eta\right),
\end{split}
\label{eq:E_explicit}
\end{equation} 
 where $\ma_0$ is defined in \eqref{eq:def_a}. Combining \eqref{eq:kappa22_real_basic_id}--\eqref{eq:E_explicit}, we finish the proof of \eqref{eq:kappa22_real} for GOE.  We transfer the GOE result to the case where $W$ has a Gaussian component of order one by the means of Proposition~\ref{prop:E_local}(i). Finally, the Gaussian component is removed via the GFT based on Lemma~\ref{lem:moment_match} and Proposition~\ref{prop:GFT} as in the proof of \eqref{eq:kappapq_new}. This finishes the proof of~\eqref{eq:kappa22_real}.\hfill\qed

\bibliographystyle{plain} 
\bibliography{refs}

\clearpage

\section*{Supplementary Material}

\setcounter{section}{0}
\renewcommand{\thesection}{S\arabic{section}}
\renewcommand{\sectionname}{Section}

In this supplement we present the Python code which we used to compute \eqref{eq:E_explicit}. In general, it allows to compute $\Phi_d(\bm{k},\bm{l})$ introduced in Lemma~\ref{lem:E_full_expansion} in terms of $\ma$ for any $d\in\N\cup\{0\}$ and multi-indices $\bm{k},\bm{l}$.

\begin{lstlisting}
import math
from sympy import symbols
from sympy import simplify

b=1  # b=1 for GOE and b=2 for GUE
a=symbols('a')  # (*@$\ma_0$@*), see (*@\eqref{eq:def_a}@*)

# The following function computes (*@$Exp(d,K,L):= \Phi_d(\bm{k},\bm{l})$@*). 
# Here we encode (*@$\bm{k}, \bm{l}$@*) by one-dimensional arrays.

def Exp(d,K,L):
    n1=len(K)
    n2=len(L)
    
    # By definition, (*@$\Phi_d=0$@*) for (*@$d<0$@*).
    if d<0: return(0)
    #Next, we consider the trivial case when (*@$\bm{k}=\bm{l}=\emptyset$@*).
    elif n1==0 and n2==0: 
        if d==0: return(1)
        else: return(0)
    # If exactly one of the multi-indices (*@$\bm{k},\bm{l}$@*) is empty, then 
    # (*@$\Phi_d$@*) vanishes for all (*@$d$@*) similarly to (*@\eqref{eq:Phi_vanish_trivial}@*).
    elif n1*n2==0: return(0)    
    # In the remaining cases (*@$\bm{k},\bm{l}\neq\emptyset$@*) and we use the recurence relation (*@\eqref{eq:Phi_recursion}@*)
    else:
        k=K[0]  # (*@$k_1$@*)
        K_cut=K[1:]  # (*@$\bm{k}^{[1]}$@*)
        
        #1st term in the first line of (*@\eqref{eq:Phi_recursion}@*)
        T1=0
        for j in range(n2):
            l=L[j]  # (*@$l_j$@*)
            coeff=(2/b)*(1j)**(k-l-1) * 2**(-k-l+1) * l * math.comb(k+l-1,k-1)
            L_cut=L[:j]+L[j+1:] # (*@$\bm{l}^{[j]}$@*)
            #In the following line we use the identity (*@\eqref{eq:a_id}@*).
            T1=T1+(1j/2)*coeff*Exp(d,K_cut,L_cut)
        
        #2nd term in the first line of (*@\eqref{eq:Phi_recursion}@*)
        T2=0
        if b==1:
            Knew=[k+1]+K_cut  # (*@$\big(k_1+1,\bm{k}^{[1]}\big)$@*)
            T2=a*k*Exp(d-1,Knew,L)
                
        #1st term in the second line of (*@\eqref{eq:Phi_recursion}@*)
        T3=0
        for u in range(k):
            u1=u+1
            K_new = [u1,k+1-u1] + K_cut  # (*@$\big(u,k_1+1-u,\bm{k}^{[1]}\big)$@*)
            T3=T3 + a*Exp(d-1,K_new,L)
        
        #2nd term in the second line of (*@\eqref{eq:Phi_recursion}@*)
        T4=0
        if n1>0:
            for i in range(n1-1):
                i1=i+1
                ki=K[i1]  # (*@$k_i$@*)
                K_new = [k+ki+1]+K[1:i1]+K[i1+1:]  # (*@$\big(k_1+k_i+1,\bm{k}^{[1,i]}\big)$@*)
                T4=T4 + (2/b)*a*ki*Exp(d-1,K_new,L)

        #Third line of (*@\eqref{eq:Phi_recursion}@*)
        T5=0
        for j in range(n2):
            l=L[j]  # (*@$l_j$@*)
            L_cut=L[:j]+L[j+1:]  # (*@$\bm{l}^{[j]}$@*)
            for u in range(k):
                u1=u+1
                coeff=(-1)**(k-u1)*(2*1j)**(-l-k-1+u1)*l*math.comb(l+k-u1,l)
                K_new = [u1] + K_cut  # (*@$\big(u,\bm{k}^{[1]}\big)$@*)
                T5=T5 + (2/b)*a*coeff*Exp(d-1,K_new,L_cut)
                
        #Fourth line of (*@\eqref{eq:Phi_recursion}@*)
        T6=0
        for j in range(n2):
            l=L[j]  # (*@$l_j$@*)
            L_cut=L[:j]+L[j+1:]  # (*@$\bm{l}^{[j]}$@*)
            for v in range(l+1):
                v1=v+1
                coeff=(-1)**k*(2*1j)**(-k-l-1+v1)*l*math.comb(l+k-v1, k-1)
                L_new = [v1] + L_cut  # (*@$\big(v,\bm{l}^{[j]}\big)$@*)
                T6=T6 + (2/b)*a*coeff*Exp(d-1,K_cut,L_new)
        
        return(T1+T2+T3+T4+T5+T6)

for d in range(4):  # (*@$d\in \{0,1,2,3\}$@*)
    print(simplify(Exp(d,[1],[1])))  # (*@$\Phi_d(1,1)$@*)
    print(simplify(Exp(d,[1,1],[1,1])))  # (*@$\Phi_d\big((1,1),(1,1)\big)$@*)

\end{lstlisting}

\end{document}